\documentclass[11pt]{article}
\usepackage{graphicx}
\usepackage{algorithm}
\usepackage{algorithmicx}
\usepackage{comment}
\usepackage{algpseudocode}
\usepackage{tikz}
\usepackage[backend=biber,style=alphabetic,maxbibnames=99]{biblatex}
\usepackage[colorlinks=true, citecolor=blue, linkcolor=black, urlcolor=blue]{hyperref}
\usetikzlibrary{chains, arrows.meta, positioning, quotes}
\usepackage{amsfonts,amsmath,amssymb,amsthm}
\usepackage{xcolor}
\usepackage{authblk}
\usepackage{mathrsfs}
\usepackage{mathtools}
\usepackage{cancel}
\usepackage[top=1in, bottom=1in, left=1in, right=1in]{geometry}
\newcommand{\note}[1]
{\textcolor{red}{#1}}
\newcommand{\fixme}[1]
{\textcolor{purple}{#1}}

\newcommand{\blu}[1]
{\textcolor{black}{#1}}
\numberwithin{equation}{section}

\newtheorem{definition}{Definition}[section]
\newtheorem{theorem}[definition]{Theorem}
\newtheorem{assm}[definition]{Assumption}
\newtheorem{lem}[definition]{Lemma}
\newtheorem{proposition}[definition]{Proposition}
\newtheorem{corollary}[definition]{Corollary}
\newtheorem{rem}[definition]{Remark}
\usepackage{amsthm}
\usepackage{enumitem} 

\theoremstyle{definition} 
\newtheorem{algo}{Algorithm}[section]

\usepackage[most]{tcolorbox} 

\newenvironment{algorithm1}[1]{%
  \begin{tcolorbox}[
    breakable,
    colback=white,
    colframe=black,
    boxrule=0.6pt,
    arc=0pt,
    left=6pt,
    right=6pt,
    top=6pt,
    bottom=6pt,
    before skip=10pt,
    after skip=10pt
  ]
  \begin{algo}
  \textbf{(#1)}
  \begin{enumerate}[leftmargin=1.5em,labelsep=0.5em]
}{%
  \end{enumerate}
  \end{algo}
  \end{tcolorbox}
}

\newcommand{\cE}[0]{\mathcal{E}}

\newcommand{\E}[0]{\mathbb{E}}

\newcommand{\R}[0]{\mathbb{R}}

\newcommand{\sT}[0]{\mathsf{T}}

\newcommand{\al}[0]{\alpha}
\newcommand{\be}[0]{\beta}

\newcommand{\de}[0]{\delta}

\newcommand{\ep}[0]{\epsilon} 

\newcommand{\ka}[0]{\kappa}

\newcommand{\La}[0]{\Lambda}

\newcommand{\rh}[0]{\rho}
\newcommand{\te}[0]{\theta}

\newcommand{\nin}[0]{\not\in}

\newcommand{\sub}[0]{\subset}

\newcommand{\subeq}[0]{\subseteq}

\newcommand{\nsubeq}[0]{\not\subseteq}

\newcommand{\bs}[0]{\backslash}
\newcommand{\iy}[0]{\infty}

\newcommand{\rc}[1]{\frac{1}{#1}}
\newcommand{\prc}[1]{\pa{\rc{#1}}}

\newcommand{\fc}[2]{\frac{#1}{#2}}
\newcommand{\sfc}[2]{\sqrt{\frac{#1}{#2}}}

\newcommand{\pl}[0]{\partial}
\newcommand{\nb}[0]{\nabla}

\newcommand{\ab}[1]{\left| {#1} \right|}

\newcommand{\ba}[1]{\left[ {#1} \right]}

\newcommand{\pa}[1]{\left( {#1} \right)}

\newcommand{\ve}[1]{\left\Vert {#1}\right\Vert}

\newcommand{\seta}[2]{\left\{{#1}:{#2}\right\}}

\newcommand{\ol}[1]{\overline{#1}}

\newcommand{\ub}[2]{\underbrace{#1}_{#2}}

\newcommand{\wt}[1]{\widetilde{#1}}

\newcommand{\diam}{\operatorname{diam}}

\newcommand{\Ent}{\operatorname{Ent}}

\newcommand{\poly}{\operatorname{poly}}

\newcommand{\Var}[0]{\operatorname{Var}}
\newcommand{\var}[0]{\text{var}}

\providecommand{\cal}[1]{\mathcal{#1}}
\renewcommand{\cal}[1]{\mathcal{#1}}

\newcommand{\pull}[9]{
#1\ar@/_/[ddr]_{#2} \ar@{.>}[rd]^{#3} \ar@/^/[rrd]^{#4} & &\\
& #5\ar[r]^{#6}\ar[d]^{#8} &#7\ar[d]^{#9} \\}

\newcommand{\cmp}[9]{
\xymatrix{
#1 \ar[r]^{#4}{#5} \ar@/_2pc/[rr]^{#8}_{#9} & #2 \ar[r]^{#6}_{#7} & #3
}
}

\newcommand{\ha}[1]{\ar@{^(->}[#1]}
\newcommand{\ls}[1]{\ar@{-}[#1]}
\newcommand{\sj}[1]{\ar@{->>}[#1]}
\newcommand{\aq}[1]{\ar@{=}[#1]}
\newcommand{\acir}[1]{\ar@{}[#1]|-{\textstyle{\circlearrowright}}}
\newcommand{\acil}[1]{\ar@{}[#1]|-{\textstyle{\circlearrowleft}}}
\newcommand{\ard}[1]{\ar@{.>}[#1]}
\newcommand{\mt}[1]{\ar@{|->}[#1]}
\newcommand{\inm}[1]{\ar@{}[#1]|-{\in}}
\newcommand{\inr}{\ar@{}[d]|-{\rotatebox[origin=c]{-90}{$\in$}}}
\newcommand{\inl}{\ar@{}[u]|-{\rotatebox[origin=c]{90}{$\in$}}}

\newcommand{\sumo}[2]{\sum_{#1=1}^{#2}}

\newcommand{\beq}[1]{\begin{equation}\llabel{#1}}
\newcommand{\eeq}[0]{\end{equation}}
\newcommand{\bal}[0]{\begin{align*}}
\newcommand{\eal}[0]{\end{align*}}
\newcommand{\ban}[0]{\begin{align}}
\newcommand{\ean}[0]{\end{align}}

\newcommand{\llabel}[1]{\label{#1}\text{\fixme{\tiny#1}}}

\newcommand{\arxiv}[1]{\url{http://www.arxiv.org/abs/#1}}

\allowdisplaybreaks[2]

\DeclareFontFamily{U}{wncy}{}
    \DeclareFontShape{U}{wncy}{m}{n}{<->wncyr10}{}
    \DeclareSymbolFont{mcy}{U}{wncy}{m}{n}
    \DeclareMathSymbol{\Sh}{\mathord}{mcy}{"58} 
    
\newcommand{\holden}[1]{\textcolor{purple}{[Holden: #1]}}

\newcommand{\Cd}[0]{C_{LS}^{(d)}}
\newcommand{\CD}[0]{C_{LS}^{(D)}}
\newcommand{\Ctheta}[0]{C_{LS}^{(\theta)}}

\title{Sampling with collective variables: Relative error analysis for umbrella sampling under general conditions}
\author[1]{Holden Lee}
\author[1]{Matheau Santana-Gijzen}
\author[2]{Jonathan Weare}

\affil[1]{Department of Applied Mathematics \& Statistics, Johns Hopkins University}
\affil[2]{Courant Institute of Mathematical Sciences, New York University}
\date{}
\begin{document}
\maketitle

\begin{abstract}
    We consider the problem of estimating the expectation of a function (observable) with respect to a high-dimensional probability distribution, under the assumption that a low-dimensional collective variable map captures all the directions of slow mixing for the distribution. More precisely, under the assumption that the fibers of the collective variable map satisfy a log-Sobolev inequality, we establish non-asymptotic relative-error guarantees for this problem using the eigenvector method for umbrella sampling (EMUS). 
    Our analysis applies to both log-Lipschitz observables and indicator of convex sets, and gives complexity bounds that are polynomial in regularity 
    parameters and inverse relative accuracy, and exponential in the dimension of the collective variable space. As a relative error guarantee, our bounds apply to rare event estimation.
    We consider Gaussian biasing windows defined on a low-dimensional collective variable space and analyze a modified sparse matrix version of EMUS. The bounds quantify the effect of the map through mixing of the conditional distribution and fluctuation of the observable along the fibers. 
\end{abstract}
\tableofcontents
\section{Introduction} 
A key task in statistics, machine learning, engineering, and the natural sciences is to estimate the expectation 
$$\pi[g] = \E_\pi[g] = \int_{\Omega} g(x)\,\pi(dx)$$
under a high-dimensional target distribution $\pi(x) \propto e^{-V(x)}$ on $\Omega\subseteq\mathbb{R}^{D}$, with density known up to normalization constant. 
A standard approach is Markov Chain Monte Carlo (MCMC), which estimates the expectation by drawing approximate samples $X^{(i)}$ produced by a Markov chain and evaluating $\rc N \sumo iN g(X^{(i)})$. 
Simple MCMC methods such as Metropolis-Hastings, Gibbs sampling, and Langevin dynamics effectively estimate the expectation when the underlying Markov chain mixes well and the observable $g$ is not concentrated on a small set. 
However, when either of these conditions fail---for instance, $\pi$ is multimodal or $g$ is the indicator of a rare event---MCMC often requires a prohibitively long running time or large number of samples. 
Although various sequential MCMC and importance-sampling methods have been developed to address multimodality and rare-event estimation, this remains a challenging problem in practice, without a clear understanding of the general conditions under which such methods work.

\blu{Without further structural assumptions or information, sampling from high-dimensional distributions is hard; for instance, exponential query lower bounds exist even for approximately bimodal distributions \cite{ge2018beyond}. We consider the setting where the problem has low-dimensional structure which captures all the directions of slow mixing, and this is given to us in the form of a collective variable mapping \cite{Dinner2020Stratification,lelievre2010free,ThiedeEtAl2016}. Such structure is particularly common in physical problems, where prior knowledge provides the reaction coordinates where mixing is expected to be especially slow.
Based on the eigenvector method for umbrella sampling, we provide the first theoretical guarantees in this setting, which work for a general class of distributions and observables, including both log-Lipschitz observables and indicators of convex sets. Our bounds are non-asymptotic and polynomial in the regularity parameters and inverse relative accuracy, and exponential only in the dimension of the collective variable space.}

\blu{Before detailing our results, we introduce umbrella sampling as a popular algorithm for this setting, situate the collective variable setting within a broad spectrum of weak to strong information, and survey the theoretical work on these settings as well as on estimation of rare events.}

\paragraph{Umbrella sampling.}
For this collective variable setting, umbrella sampling \cite{Torrie1977Umbrella} is a popular method that addresses the dual problems of slow mixing and rare event estimation. 
In this article, we focus on the eigenvector method for umbrella sampling (EMUS) introduced in \cite{ThiedeEtAl2016} and give the first non-asymptotic guarantees that 
addresses both these problems.
EMUS was inspired by umbrella sampling 
and other methods such as the weighted histogram analysis method (WHAM) \cite{Kumar1992WHAM} and the Multistate Bennett Acceptance Ratio (MBAR) \cite{Shirts2008MBAR} method for computing free energies.
Guarantees for EMUS have been given in specific situations or as asymptotic guarantees: 
\cite{ThiedeEtAl2016} derived a central limit theorem and demonstrated the method on free-energy calculations, while  
\cite{Dinner2020Stratification} established broader asymptotic variance results and demonstrated EMUS effectiveness on certain multimodal targets and rare-event estimation problems.  
By giving general conditions under which EMUS is efficient, we significantly broaden the scope of theoretical results, as well as highlight the design and hyperparameter choices that make this method effective. 


Umbrella sampling addresses poor global exploration by replacing one difficult estimation with a family of biased estimations with respect to overlapping distributions. 
It relies on the choice of a collective variable (CV) mapping $\theta: \mathbb{R}^D \rightarrow \mathbb{R}^d$, with $d\ll D$, and window functions $\psi_1, \dots, \psi_L$ which tile the CV space $\R^d$ and localize the target measure in that space. The localized distributions are then defined as $\pi_k \propto \psi_k \pi$.  
The aim of the CV map is to reduce the problem to the directions with slow mixing; localizing the distribution then reduces it to a collection of problems which do not require exploration in the CV directions. 
\blu{What remains is then to combine the samples from individual $\pi_k$ with the correct weights. The EMUS method reduces the problem of estimating the weights to computing the eigenvectors of a window overlap matrix, which can be done efficiently.}
For our analysis, we consider samples to be generated independently within the biased distributions $\pi_k\propto \psi_k\pi|_{K_U}$ restricted to an truncation set $K_U$ and the $\psi_k$ to be Gaussians in the CV space. 

In the EMUS construction, the efficacy of EMUS depends largely on the practitioner's choice of map $\theta$, which can be thought of as a weak form of advice.
In free energy calculations, the collective variable $\theta$ can be chosen based on intuition or scientific principles \cite{Dinner2020Stratification,lelievre2010free,ThiedeEtAl2016}. 
In rare-event sampling or when computing marginals, the problem itself may suggest a particular $\theta$ \cite{Dinner2020Stratification}. 
In practice, the choice of $\theta$ has a substantial effect on the efficacy of the method. 
Our analysis reflects this reality of EMUS by having bounds dependent on how well $\theta$ captures the geometry of $\pi$ and $g$, made rigorous in Section \ref{s: smooth observable}. Specifically, EMUS works best when $\pi$ has good mixing and $g$ has small variation on the level sets of $\te$.  
These quantitative bounds allow our results to apply to a broad choice of $\theta$ and provides insights to when the method works well. 



\paragraph{Sampling with advice.} 
We are interested in the problem of approximately sampling from a density $\pi(x)\propto e^{-V(x)}$ that is multimodal or slowly mixing in certain directions. 
We can categorize different approaches to this problem depending on the strength of extra information, or \emph{advice} that we are given. A collective variable mapping can be thought of as a form of weak advice. We survey what is known theoretically with different levels of advice. For the purpose of theoretical analysis, a common way of formalizing the multimodality \cite{ge2018beyond, pmlr-v291-koehler25a} is to assume that $\pi$ decomposes according to a mixture, 
$\pi = \sum_{i=1}^m w_i\pi_i$ or $\pi = \int \pi_y dw(y)$. 

Under the setting of \emph{no advice},  simulated and parallel tempering are statistical physics-inspired algorithms that 
define a Markov chain on an extended state space involving a temperature variable, using higher temperature for exploration to find the modes. Theoretical guarantees are known under suitable decompositions \cite{madras2002markov,Woodard2009,ge2018beyond,garg2025restricted}. Sequential Monte Carlo (SMC) instead moves particles uni-directionally from high to low temperature and uses resampling of particles for reweighting, with earlier guarantees requiring mode separation
\cite{paulin2018error,mathews2024finite} and later guarantees for more general mixtures \cite{lee2024convergence,han2025polynomial,han2025convergencesequentialmontecarlo}. However, all these methods require an unavoidable and restrictive assumption that weights of different modes do not become too small at intermediate temperatures, creating a bottleneck. Simple distributions such as a mixture of two distributions with different variances violate this assumption and lead to lower bounds, as observed by \cite{woodard2009sufficient} for simulated and parallel tempering and \cite{ge2018beyond}, which gives an exponential lower bound for all algorithms.

This motivates the study of what is possible with stronger advice. A weak form of advice is knowing approximate mode locations (also called warm starts), while a strong form of advice is having a few samples from the target distributions, with the goal of drawing fresh samples. Having a collective variable mapping can be thought of as an especially weak form of advice, as the low-dimensional projection is chosen based on the directions where slow mixing occur; this gives the directions but not the exact locations of different modes.
In the setting of strong advice, sampling is generically solvable with almost optimal sample complexity via running the standard Markov chain from data-based initialization \cite{koehler2023sampling,pmlr-v291-koehler25a, gay2025sampling}. In the setting of weak advice, the reweighted Annealed Leap Point Sampler (ReALPS) \cite{lee2025sampling}, which adds a iterative reweighting scheme to the ALPS algorithm \cite{Roberts_Rosenthal_Tawn_2022,tawn2021annealed} can provably sample under a general assumption that tilting towards warm start points retains a constant portion of the mass.
In the collective variable (CV) setting, a common technique is \emph{stratification}, which seeks to uniformize the amount of sampling across the CV. Algorithms include umbrella sampling \cite{torrie1977nonphysical,thiede2016eigenvector}, 
the Wang--Landau algorithm \cite{wang2001efficient}, and adaptive biasing force \cite{darve2001calculating}, and weighted ensemble simulation \cite{huber1996weighted,chong2025weighted}, a form of SMC that stratifies with the CV. Rigorous theoretical results are limited in this setting, which motivates our work.

Because we do not constrain the distribution in the CV space other than with minimal regularity assumptions, an exponential dependence on the CV dimension $d$ is expected and unavoidable. This is reasonable in light of results in the weak advice setting which have polynomial dependence on the number of modes; the ``modes" in our setting can be thought of as tiling the CV space. Those results do not apply to our setting, as the CV map induces a different structure and we seek relative error guarantees, which are significantly stronger.



\paragraph{Relative error guarantees and rare event analysis.} 
\blu{Assume a non-negative observable $g$. Our goal is to efficiently compute an estimate $\hat{\pi}[g]$ with small relative error, i.e., small error measured on the logarithmic scale,
$$\left|\log \hat{\pi}[g] - \log \pi[g]\right|.$$
Relative accuracy is the appropriate metric for analyzing rare event sampling, and an effective method for relative accuracy should require complexity that does not depend on the scale of $\pi[g]$. 
Simple Monte Carlo does not achieve this goal because it
requires a number of samples $O\prc{\ep^2}$ to achieve \emph{absolute} error $\ep$, and a small absolute error may incur a large multiplicative error. Obtaining relative error guarantees is challenging even when $\pi$ is a simple high-dimensional distribution such as a Gaussian.}

This is especially important for the analysis of rare events.
There are two categories of approaches to the problem of estimating probabilities of and expectations under rare events: simulation-based \cite{au2001estimation,cerou2007adaptive,cerou2012sequential}, which depend on drawing samples (as in umbrella sampling), and asymptotic expansions. Simulation-based techniques provide increased accuracy with more compute, while asymptotic expansions provide an efficient explicit formula that incurs a fixed error, which vanishes in an appropriate limit. The two types of methods can be combined by using asymptotic expansions to provide a proposal distribution for sampling.
Asymptotic methods rely on approximating the density and boundary of the rare-event set around the highest-density point. 
In particular, second-order reliability methods (SORMs) \cite{hu2021second}, assuming a known way to transform $\pi$ into a standard Gaussian density, approximate the transformed boundary by a quadratic function. Building on the theory of Laplace integral approximations \cite{bleistein1975asymptotic,tang2025laplace,katsevich2024laplace}, \cite{katsevich2025asymptotic} give the first asymptotic high-dimensional analysis of such methods, with explicit error bounds. Large deviation theory provides another asymptotic framework \cite{grafke2019numerical}.

Our methods are also geared towards high dimensions, but work with general densities rather than known transformations of Gaussians. Instead of requiring smoothness of the boundary, we allow more general sets/functions and assume a low-dimensional structure through collective variables.

We also note the related problem of computing the volume of a convex body, which can be done efficiently by using MCMC with simulated annealing \cite{dyer1991computing}; a long sequence of works have since then improved the running time \cite{lovasz1993random,lovasz2006simulated,cousins2018gaussian,jia2026reducing}.

\subsection{Problem Statement \& Assumptions}

In the following, we will identify measures with their densities. 
Given a target distribution $\pi(x)\propto e^{-V(x)}$ on $\mathbb{R}^D$, the goal is to estimate the average
\[
\pi[g] = \int g(x) \pi(dx)
\]
for some non-negative observable $g: \mathbb{R}^D \rightarrow [0,\infty)$, with small relative error.
The state space dimension $D$ is often very large. \blu{We assume that a smooth collective variable (CV) map $\te:\R^D \to \R^d$ is given, where the dimension $d<D$ is usually very small. 
We now make several assumptions to make the problem tractable, and to formalize what it means for $\te$ to capture the low-dimensional structure of the distribution. 
First, we assume that most of the mass of $\pi$, and $g$ under $\pi$, is within a bounded set (Assumption \ref{a: good truncation}). Next, we assume regularity properties of the CV map $\te$, in terms of its effect on the measure $\pi$ (Assumption \ref{d: target distribution pushforward}) as well as the map itself (Assumption \ref{a: collective variable}). We formalize the key property that $\te$ captures all directions of slow mixing in terms of a functional inequality satisfied by the conditional measures $\pi(\cdot \mid \te(x)=z)$ (Assumption \ref{a:lsi-cond}). Finally, we assume $g$ is either a log-Lipschitz function or an indicator of a well-conditioned convex set, in terms of the collective variables (Assumption \ref{a:g}). 
In the full statement of the theorem, we consider a more general case where $g$ can vary on the fibers.}

For our analysis, we will need to restrict the distributions to a bounded set. Assume we are given a bounded convex set in the CV space $\mathcal{A} \subset \mathbb{R}^d$, and define its pullback by 
$$K_U \vcentcolon = \theta^{-1}(\mathcal{A}) = \{x \in \mathbb{R}^D : \theta(x) \in \mathcal{A}\}.$$
We first make the following assumption about the set $K_U$ as it relates to $\pi$ and $g$. This ensures that the umbrellas capture both the mass of the distribution $\pi$ and the bulk of the observable $g$.
\begin{assm}[Good truncation]\label{a: good truncation} For a set $K_U \subseteq U \vcentcolon = \{x \in \mathbb{R}^D : S(x) > 0 \} \subseteq 
\R^D$ there exist constants $\delta_\pi, \delta_g$ such that 
                \begin{enumerate} 
            \item $\pi(K_U) \geq 1 - \delta_\pi$, and
            \item $\int_{K_U} g(x)\pi(dx) \geq (1-\delta_g)\pi(g)$.
            \end{enumerate}
            \end{assm}

Next, we introduce regularity assumptions on the the marginal (push-forward) measure $\mu = \theta_\#\pi: \mathbb{R}^{d} \rightarrow \mathbb{R}$.


\begin{assm}[Regularity of push-forward measure] \label{d: target distribution pushforward}
Let $\mu = \theta_\#\pi$ be the push-forward measure on the collective variable space $\mathbb{R}^d$. We assume $\mu(z) \propto e^{-A(z)}$, where $A(z)$ is the potential in the CV space and satisfies: 
\begin{enumerate}
    \item (Smooth potential) The target distribution $\mu(z) \propto e^{-A(z)}$ has a $\ell$-smooth potential, i.e. the Hessian $\nabla^2A(z)$ satisfies: 
    $$-\ell\cdot I_d \preceq \nabla^2A(z) \preceq \ell \cdot I_d.$$
    \item (Bounded potential gradient) The target potential $A(z) = -\log \mu(z)$ is differentiable on the set $\mathcal{A}$, and there exists a constant $G \geq 0$ such that 
    $$\sup_{z \in \mathcal{A}} \|\nabla_z A(z)\| \leq G.$$
\end{enumerate}
\end{assm}
Lipschitzness of the push-forward target potential is necessary to ensure success of the EMUS algorithm:  
Without a gradient bound on $A(z)$, the target measure $\mu(z)$ can decay arbitrarily fast. 
This creates regions where adjacent windows have small overlap, causing inaccurate eigenvector estimation. Note that if $\mathcal A$ contains a critical point of $A$, then $G$ can be bounded in terms of $\ell$ and the diameter of $A$.

To introduce our next assumptions, for any $z \in \mathcal{A}$, define the fiber $\Sigma_z$ to be the submanifold defined by $\Sigma_z = \{ x \in \mathbb{R}^D \mid \theta(x) = z\}.$ Along this submanifold, we define the conditional measure as $\pi_z = \pi(\cdot  \mid \theta(x) = z)$. 

We characterize the mapping $\theta$ through its Jacobian $J(x) = \nabla \theta(x) \in \mathbb{R}^{d \times D}$ and the Gram matrix $G(x) = J(x)J(x)^T \in \mathbb{R}^{d\times d}$.
Let 
\begin{align}
    P(x) &= I_D - J(x)^\sT G(x)^{-1}J(x)
\label{e:proj}
\end{align} be the orthogonal projection onto the tangent space $T_x\Sigma_z$.
The following assumptions are made to be consistent with the setting of \cite{lelievre2009general}.

\begin{assm}[Regularity of CV map]\label{a: collective variable} Let $\pi(x) \propto e^{-V(x)}$. Given $\mathcal{A} \subset \mathbb{R}^d$ and $K_U = \{x \in \mathbb{R}^D \mid \te(x) \in \mathcal{A}\}$, suppose the following assumptions hold:
\begin{enumerate}
    \item The mapping $\theta: \mathbb{R}^d \rightarrow \mathbb{R}^d$ is smooth and $\det(G) \neq 0$. Moreover, there exists a constant $m_\theta > 0$ such that: 
    $$\lambda_{\min}(G(x)) \geq m_\theta, \;\; \forall x \in K_U.$$
    \item  Let $F: \mathbb{R}^D \rightarrow \mathbb{R}^d$ be the local mean force defined as
    $$ F_\alpha = \sum_{\beta=1}^d (G^{-1})_{\alpha, \beta} \nabla \theta_\beta \cdot \nabla V - \textup{div} \left( \sum_{\beta=1}^d G^{-1}_{\alpha, \beta} \nabla \theta_\beta \right). $$
    Defining $P$ as in \eqref{e:proj}, there exists a constant $\Lambda > 0$ bounding the fluctuations of $F$ along the fibers
    $$\Lambda = \sup_{x\in K_U}\|\nabla F(x) P(x)\|_{op}.$$
\end{enumerate}
\end{assm}
As a simple example, a linear projection $\theta(x) = Qx$ satisfies these assumptions with $m_\theta=1$ and if \(
\sup_{x\in K_U}\|\nabla^2V(x)\|_{\mathrm{op}}\le K
\) then $\Lambda \leq K$. 

Next, we encode the key assumption that $\te$ captures all directions of slow mixing, by assuming a log-Sobolev inequality on the conditional measures $\pi_z$.
\begin{assm}[Log-Sobolev inequality for conditional measures]
\label{a:lsi-cond}
    For every $z\in \mathcal A$, the conditional measure $\pi_z$ satisfies an LSI with constant $\Ctheta$: for all functions $f$ such that the RHS is finite,
    $$ \Ent_{\pi_{z}}(f^2) \le 2 C_{LS}^{(\theta)} \int_{\Sigma_z} |P \nabla_x f|^2 d\pi_{z}, \quad \forall z \in \mathcal{A}. $$
\end{assm}

Finally, we make one of two assumptions on the observable $g$.
\begin{assm}[Regularity of $g$]
\label{a:g}
The observable $g:\R^D\to [0,\iy)$ factors as $g = \tilde g \circ \te$, and one of the following holds. 
\begin{enumerate}
    \item (Log-Lipschitz) 
    $\tilde g$ is $L_g$-Lipschitz.
    \item (Indicator of well-conditioned convex set) Suppose that $\tilde g=\mathbf1_E$ where 
    \[B(c,\rho)\subseteq E\subseteq B(c,R)\]
    for some $c\in \R^d$ and $0<\rh<R$. Define the condition number $\kappa=R/\rho$. 
\end{enumerate}
\end{assm}



\subsection{Algorithm}

We now introduce the version of the EMUS algorithm that we will analyze.
To construct the EMUS estimator we first select a collection of functions $\psi_1,\dots,\psi_L$ and define the corresponding ``window'' distributions $\pi_k \propto \psi_k \pi|_{K_U}$ on the truncation domain $K_U$. (Since sampling occurs in the full space $\mathbb{R}^D$, truncated sampling of $\pi_k \propto \pi\cdot\psi_k|_{K_U}$ occurs with respect to $K_U \subset \mathbb{R}^D$.)
Typically the $\psi_k$ are chosen to factor through the CV map, $\psi_k = \tilde \psi_k (\te(x))$, where $\tilde \psi_k:\R^d \to \R$. 
A common choice in practice for $\tilde\psi_k$ are Gaussian distributions,
$$\tilde{\psi}_k(z) \propto \exp\pa{-{\frac{\beta}{2}}\left\|z-y_k\right\|^2},$$
where the $y_k$ are chosen to discretize the CV space.
The corresponding truncated biased measure is then defined in $\mathbb{R}^D$ as
$$\pi_k(dx) \vcentcolon = \frac{1_{K_U}(x)\psi_k(x)}{\int_{K_U}\psi_k(x)\pi(dx)}\pi(dx).$$
Define 
\begin{equation}\label{eq:g^*}
  g^\ast(x):= \frac{g(x)}{\sum_{k=1}^L \psi_k(x)}, \quad 
  1^\ast(x) := \frac{1}{\sum_{k=1}^L \psi_k(x)},\quad 
\psi^\ast_j(x) := \frac{\psi_j(x)}{\sum_{k=1}^L \psi_k(x)}
 \end{equation}
The EMUS then algorithm proceeds as in Algorithm \ref{alg:emus}.

\begin{algorithm}[ht!]
\caption{The EMUS algorithm}
\label{alg:emus}
\begin{algorithmic}[1]
  \Require Window distributions $\{\pi_i\}_{i=1}^L$,
sample sizes $\{N_i\}_{i=1}^L$, and a undirected connected
graph $G$ containing $(i,i)$ for every $i$.
  \Ensure Estimate $\bar\pi[g]$.

  \For{$i=1,\ldots,L$}
    \State Generate samples
      $X_1^i,\ldots,X_{N_i}^i\sim\pi_i$.

    \State Compute
    \Statex
    \begin{equation}\label{eq:barg1}
      \bar g_i^\ast
      :=
      \frac{1}{N_i}
      \sum_{t=1}^{N_i}g^\ast(X_t^i),
      \qquad
      \bar 1_i^\ast
      :=
      \frac{1}{N_i}
      \sum_{t=1}^{N_i}1^\ast(X_t^i).
    \end{equation}

    \For{$j=1,\ldots,L$}
      \State Compute the empirical overlap
      \Statex
      \begin{equation}\label{eq:barF}
        \bar F_{ij}
        :=
        \frac{1}{N_i}
        \sum_{t=1}^{N_i}\psi_j^\ast(X_t^i).
      \end{equation}
    \EndFor
  \EndFor

  \State Define the empirical sparse overlap matrix
\[
\bar F^0_{ij} =
\begin{cases}
\bar F_{ij},
  & i\neq j,\ (i,j)\in G,\\
\bar F_{ii}
  +\displaystyle\sum_{\substack{k\neq i\\(i,k)\notin G}}\bar F_{ik},
  & i=j,\\
0,
  & i\neq j,\ (i,j)\notin G.
\end{cases}
\]

  \State Compute the normalized stationary vector
  \[
    \bar w^{\mathsf T}
    =
    \bar w^{\mathsf T}\bar F^0,
    \qquad
    \sum_{i=1}^L\bar w_i=1.
  \]
  \Statex \hspace{\algorithmicindent}
    This may be done using a QR factorization of $I-\bar F^0$.

  \State Compute and return
  \[
    \bar\pi[g]
    :=
    \frac{\sum_{i=1}^L\bar w_i\bar g_i^\ast}
         {\sum_{i=1}^L\bar w_i\bar 1_i^\ast}.
  \]
\end{algorithmic}
\end{algorithm}

A key step in the EMUS algorithm is to learn the eigenvector $w$ of ${F}$, which is exactly the normalizing constant $Z_i = \int_{K_U} \psi_i(x)\pi(dx)$ up to an unknown scalar $c$.
We make a key modification for the purpose of theoretical analysis: 
We define a sparse matrix $F^0$ by zeroing out entries corresponding to windows that have weak overlap, in a way that preserves the eigenvectors. 
This removes entries that are subject to large relative noise from the empirical estimate $\bar F^0$. 
We later show that under mild conditions, the eigenvector $\bar{w}$ of the approximate sparse matrix $\bar{F}^0$, 
yields a good approximation of the true weight vector $w$. For ease of analysis, We specify the entries to keep using a fixed graph $G$, though we note that they can also be chosen based on a threshold.

\subsection{Main theorem}

We now state our main theorem for the relative error of the EMUS estimator. 
This applies to a log-Lipschitz observable $g$ or an indicator observable $g = 1_{E}\circ \theta$, where $E$ is a convex subset of $\mathbb{R}^d$.
We present a simplified version of our results where we assume the observable $g(x)$ is constant along the fibers of the collective variable map.

\begin{theorem}\label{t: main theorem alg1} 

Given a measure $\pi$ on $\R^D$, a CV map $\te:\R^D\to \R^d$, a convex set $\mathcal A\sub \R^d$, suppose that Assumptions \ref{a: good truncation}, \ref{d: target distribution pushforward}, \ref{a: collective variable}, \ref{a:lsi-cond}, and \ref{a:g} are satisfied.
Let $\epsilon > 0$, $0 < \delta < 1$, and suppose $\max\{\delta_\pi, \delta_g\} \leq 1 - e^{-\epsilon/4}.$ 

There exists a collection of $L$ windows with a connected sparse overlap graph $G$ such that the following holds. 
Compute the EMUS estimate $\bar{\pi}[g]$ from Algorithm \ref{alg:emus}, with $N$ independent samples per umbrella, drawn from a reflected Langevin run of continuous time $T$, initialized from $\mu_{k,0}$ with $\chi^2(\mu_{k,0}\| \pi_k) \leq B$, for all $k$. The parameters are specified as follows.
\begin{enumerate}
    \item 
    For $g$ log-Lipschitz, we can take $N$ and $T$ such that 
\[
N = \fc{e^{\wt O(d^3)}}{\ep^2 \de }\poly(\diam(\cal A), \ell, m_\te^{-1}, \La, C_{LS}^{(\theta)}, L_g)^d, \qquad T = O(C_{LS}^{(\theta)} \log(BN)).
\]
     \item 
    For indicator $g$, we can take $N$ and $T$ such that 
\[
N = \fc{e^{\wt O(d^3)}}{\de }\poly\pa{\rc{\ep},\diam(\cal A), \ell, m_\te^{-1}, \La, C_{LS}^{(\theta)}, \ka, \rh^{-1}}^d, \qquad T = O(C_{LS}^{(\theta)} \log(BN)).
\]   
      \end{enumerate}
Then, with probability at least $1-\delta$, 
$$\left|\log \bar{\pi}[g] - \log \pi[g] \right| \le \epsilon.$$
The total continuous diffusion time is $L \cdot N \cdot T$. 
\end{theorem}

The only difference between the parameter dependencies in the two cases is a dependency on $\frac{1}{\epsilon^d}$ in the indicator case. 
The additional factor of $\frac{1}{\epsilon^d}$ arise from the finer latter spacing required by the lower bound, which is a consequence of the lower bound in Lemma \ref{l: good set to volume}. We give more precise dependencies in the proof in Section \ref{s:pf}.

Our full results in the log-Lipschitz case, which quantify the effects of the mapping $\theta$ are given in Section \ref{s: collective variable mapping}. 
In the log-Lipschitz case the runtime implicitly depends on the fluctuations of the observable along the fibers of the collective variable. 
This is quantified in Section \ref{s: smooth observable} by the constant $\gamma_0$ with 
\[\Ctheta L_{\textup{fiber}}^2 \leq \frac{\gamma_0}{2} \leq \frac{60d}{\sqrt{\pi}},\]
where $\Ctheta$ and $L_{\textup{fiber}}$ denote the log-Sobolev and log-Lipschitz constants along the fibers. 
For a poorly chosen CV map, which leaves lots of variation along the fibers, this constant can degrade the bounds. 
Theorem \ref{t: Explicit Error for Log-lipschitz observable} makes the dependency on $\gamma_0$ explicit in the log-Lipschitz case. 
In the case of an indicator observable, the observable is constant along the fibers so $L_{\textup{fiber}} = 0$. 



\paragraph{Theory vs. practice.} Our analysis makes several assumptions or algorithmic choices that differ from common implementations of umbrella sampling: a sparse overlap matrix, an explicit truncation region, and independent samples within each window. 
\begin{enumerate}
\setlength{\itemsep}{0pt}
\setlength{\parskip}{0pt}
    \item \textbf{Sparse overlap matrix.} Enforcing a sparse overlap matrix allows us to exclude weak-overlap entries whose relative estimation errors are difficult to control, while approximately preserving the exact stationary weights. 
Our analysis does not establish that this modification is necessary for good practical performance. 
\item \textbf{Truncation and truncated sampling.} Since the window placement cannot be arbitrary, restriction to a covered region in the collective-variable space is intrinsic to any implementation of umbrella sampling.
For any finite collection of Gaussian windows, the sum approaches 0 as we enter the tails of the boundary windows; therefore, any analysis of EMUS with gaussian windows requires explicit treatment of these tails. This restriction means that
we choose reflected Langevin dynamics as our algorithm to sample the truncated window distributions. 
Extending the guarantees to non-truncated sampling would require corresponding mixing and tail-moment bounds. 
\item \textbf{Independent samples.} We assume independent samples within each window to obtain the variance reduction used in both Lemma \ref{l: Sampling Error} and the estimation of overlap-matrix entries. 
Practical implementations typically use correlated samples drawn from some MCMC process. 
MCMC draws can be thinned to obtain approximately independent samples.
\end{enumerate}

\section{Proof overview}
We give an overview of the proof of the main theorem, Theorem~\ref{t: main theorem alg1}.

\paragraph{Error decomposition.} In Section~\ref{s: relative error}, we decompose the 
    relative error for estimating $\E_\pi f$ into the error in the empirical eigenvector weights, the finite-sample error within the windows, and the bias caused by sampling from approximate window distributions, as represented by each step in the sequence
    \[
        \bar{S}(\bar{w},f)
        \longrightarrow
        \bar{S}(w,f)
        \longrightarrow
        \widetilde{S}(w,f)
        \longrightarrow
        S(w,f)
    \]
The remaining comparison between $S(w,g)/S(w,1)$ and $\pi[g]$ accounts for the truncation error.

    Theorem~\ref{t: general error decomposition} shows that if, for each $f\in\{g,1\}$, the empirical estimate with empirically estimated weights satisfies
    \begin{equation}\label{e:relative_factor_overview}
        e^{-\eta_f^-}S(w,f)
        \leq
        \bar{S}(\bar w,f)
        \leq
        e^{\eta_f^+}S(w,f),
    \end{equation}
    then the EMUS estimator satisfies a log-relative-error bound. For the general decomposition, the exponents $\eta_f^-$ and $\eta_f^+$ can be controlled using an upper bound on the weighted relative standard deviation
    \begin{equation}\label{e:rel_variance_overview}
        \mathcal{Q}_{\{\mu\}}(w,f)
        :=
        \frac{\sum_k w_k\sqrt{\operatorname{Var}_{\mu_k}(f^*)}}
        {\sum_k w_k\mathbb{E}_{\mu_k}[f^*]}.
    \end{equation}
    To obtain a general decomposition we make two additional assumptions: first that there exists a good path of windows that have overlap $c_p$, and second that the biased samples from $\tilde{\pi}_k$ satisfy $\chi^2(\tilde{\pi}_k,\pi_k) < \tilde{\epsilon}$, for all $k$.
    We later show that both assumptions are satisfied with appropriate window spacing and when running reflected Langevin for appropriate time.

\paragraph{CV space quantities.} In Section~\ref{s: collective variable mapping}, we work in the collective-variable space and show how the spacing and inverse temperature of the Gaussian windows
    \[
        \widetilde{\psi}_k(z)
        \propto
        \exp\!\left(-\frac{\beta}{2}\lVert z-y_k\rVert^2\right)
    \]
    control the umbrella sum $S$ and the good-path overlap constant $c_p$. If the window functions are sufficiently concentrated and spaced sufficiently densely, then neighboring windows have non-negligible overlap, the umbrella sum is nearly uniform on the truncation set, and each biased marginal is strongly log-concave.

    We also show that the log-Sobolev constant $\CD$ on the full space is controlled by the collective-variable map and the window temperature:
    \begin{equation}\label{eq:cd_overview}
        \CD
        \leq
        \Ctheta
        +\frac{1}{m}\Cd
        +\Lambda^2\Cd(\Ctheta)^2.
    \end{equation}
    Here, $\Ctheta$ is the log-Sobolev constant along the collective-variable fibers, $\Cd$ is the log-Sobolev constant in the collective-variable space, and $m$ and $\Lambda$ are geometric constants determined by the collective-variable mapping.

\paragraph{Concrete bounds for log-Lipschitz and indicator observables.} In Sections~\ref{s: smooth observable} and~\ref{s: indicator}, we make the exponents in~\eqref{e:relative_factor_overview} explicit. Section~\ref{s: smooth observable} treats log-Lipschitz observables, while Section~\ref{s: indicator} treats indicators of the form $g=\mathbf{1}_E\circ\theta$ for a convex set $E$. These cases require different arguments because the regularity of the observable has a significant effect on how the individual windows contribute to the EMUS estimate.

    \begin{enumerate}
        \item For a log-Lipschitz observable, the main reduction is Corollary~\ref{c: main_theorem}, which gives a relative-error bound in terms of a uniform upper bound on~\eqref{e:rel_variance_overview}. Proposition~\ref{p: gamma bound relative variance} and Lemma~\ref{l: Stability of Q, Gaussian} establish such a bound for both the exact window distributions $\pi_k$ and the approximate sampling distributions $\widetilde{\pi}_k$.

        \item For an indicator observable, we treat the upper and lower bounds in~\eqref{e:relative_factor_overview} separately. For the upper bound, we majorize $g=\mathbf{1}_E\circ\theta$ by the log-Lipschitz function
        \(
            h_\lambda(x)
            =
            \exp\bigl(-\lambda d(E,\theta(x))\bigr)
        \)
        and apply the results of the preceding section. For the lower bound, we retain only windows satisfying $\pi_k(E)\geq p_E$ for appropriately chosen $p_E$ and show that these windows still capture nearly all of the exact EMUS contribution.
    \end{enumerate}

\paragraph{Sampling with reflected Langevin.} In Section~\ref{s: sampling}, we turn the tolerance for the biased samples $\tilde{\epsilon}$ into a runtime bound when the approximate samples are generated by reflected Langevin diffusion.
    Reflected Langevin dynamics on $K_U$, with stationary distribution $\pi_k$, has mixing controlled by $\CD$:
    \[
        \chi^2(\mu_T\Vert\pi_k)
        \leq
        \exp\!\left(-\frac{2T}{\CD}\right)
        \chi^2(\mu_0\Vert\pi_k).
    \]
    We then use~\eqref{eq:cd_overview} to express this runtime in terms of the collective-variable mapping and the window temperature. Combining this mixing guarantee with the error bounds from the preceding sections proves Theorem~\ref{t: main theorem alg1}.

\section{Relative Error of the EMUS Estimator} \label{s: relative error}
In this section, we bound the relative error of the EMUS estimator 
$$|\log \bar{\pi}[g] - \log \pi[g]|.$$ 
We organize the analysis by separating four sources of error: the error in the eigenvector weights, the finite-sample statistical error within each window, the bias introduced by sampling from an approximate rather than exact window distribution, and the error caused by restricting the target distribution to the influence set $K_U$. 
Each component is controlled by a different argument and can therefore be analyzed independently before the bounds are combined.
To make this decomposition explicit, for a weight vector $w$ and an observable $f$, define

\begin{align}
    \bar{\mathcal{S}}(w, f) &:= \sum_{k=1}^L w_k \bar{f}^*_k, \quad &&\text{(Empirical Estimate)} \\
    \tilde{\mathcal{S}}(w, f) &:= \sum_{k=1}^L w_k \mathbb{E}_{\tilde{\pi}_k}[f^*], \quad &&\text{(Biased Expectation)} \\
    \mathcal{S}(w, f) &:= \sum_{k=1}^L w_k \mathbb{E}_{\pi_k}[f^*], \quad &&\text{(EMUS Expectation).}
\end{align}
where $\bar{f}^*_k$ denotes the sample average in window $k$. With \(\bar{w}\) denoting the empirical eigenvector weights, the EMUS estimator can be written as
\[
\bar{\pi}[g] = \frac{\bar{\mathcal{S}}(\bar{w}, g)}{\bar{\mathcal{S}}(\bar{w}, 1)}. 
\]
The sequence
\[
\bar{S}(\bar{w},f)
\longrightarrow
\bar{S}(w,f)
\longrightarrow
\widetilde{S}(w,f)
\longrightarrow
S(w,f)
\]
isolates the weight, sampling, and bias errors, respectively.
The remaining comparison between
\[
\frac{S(w,g)}{S(w,1)}
\quad\text{and}\quad
\pi[g]
\]accounts for the truncation error.
Additionally, to analyze the sampling error, we define the \textit{Weighted relative standard deviation} $\mathcal{Q}_{\{\mu\}}(w,f)$ for the distributions $\{\mu_k\}$ as:  
\begin{equation}\label{e: rel variance}
    \mathcal{Q}_{\{\mu\}}(w,f) = \frac{\sum_k w_k\sqrt{\Var_{\mu_k}(f^*)}}{\sum_k w_k \mathbb{E}_{\mu_k}(f^*)}.
\end{equation}
When the context implies the sampling distributions $\{\tilde{\pi}_k\}$ or the target distributions $\{\pi_k\}$, we simply write $\mathcal{Q}(w, f)$.

\subsection{Main Result}
We combine the bounds derived in Section \ref{s: relative error} to provide a general non-asymptotic error bound for the EMUS estimator. 
We will present a general theorem in terms of bounds on the constant $\mathcal{Q}_{\{\mu\}}(w,f)$ \eqref{e: rel variance}. 
Our general theorem assumes access to bounds of the form: 
\begin{equation}\label{e: suff_condition}
    e^{-\eta_f^-}\mathcal S(w,f)
\leq
\bar{\mathcal S}(\bar w,f)
\leq
e^{\eta_f^+}\mathcal S(w,f).
\end{equation}
This allows for skewed upper and lower bounds and allows us to apply \ref{t: general error decomposition} to our results in Section \ref{s: collective variable mapping}.
In Section \ref{ss: decomposed error}, we show that the condition in eq. (\ref{e: suff_condition}) holds in general with symmetric upper and lower bounds.  
We later show that under reasonable assumptions on the windows $\psi_i$ and the geometry of the space such bounds can be made rigorous.

\begin{theorem}[General error decomposition]\label{t: general error decomposition} Let $g : \mathbb{R}^D \rightarrow [0,\infty)$ satisfy $\pi[g] > 0$ and write the EMUS estimator $$\bar{\pi}[g] \vcentcolon = \frac{\bar{\mathcal S}(\bar w, g)}{\bar{\mathcal{S}}(\bar w, 1)}.$$
Let Assumption \ref{a: good truncation} hold on the influence set $K_U = \theta^{-1}(\mathcal{A}) \subset \mathbb{R}^D$. Suppose that with probability $1-\delta_f$, for each $f \in \{g,1\}$, 
\begin{equation}\label{e: condition_main_theorem}
e^{-\eta_f^-}\mathcal S(w,f)
\leq
\bar{\mathcal S}(\bar w,f)
\leq
e^{\eta_f^+}\mathcal S(w,f)
\end{equation}
with $\eta_f^-,\eta_f^+\geq 0$. Then, with probability at least $1-2\delta_f$, 
$$\left|
\log\bar\pi[g]-\log\pi[g]
\right|
\leq
\max\left\{
\eta_g^-+\eta_1^+,
\;
\eta_g^++\eta_1^-
\right\} + \max\left\{\log\frac{1}{1-\delta_\pi}, \log\frac{1}{1-\delta_g}\right\}.$$
\end{theorem}

We make additional assumptions on the quality of the samples used to in the Monte Carlo estimate and the existence of a path of windows which has good overlap. 
There assumptions are used to make Theorem \ref{t: general error decomposition} explicit given a sparse overlap matrix. 

\begin{assm}[Good samples]\label{a: good samples}Let $\{X_t^i\}_{t=1}^{N_i} \sim \tilde{\pi}_i$ be the samples from the EMUS estimate. Then for all $i = 1, \dots, L$, $\tilde{\pi}_i$ is absolutely continuous with respect to $\pi_i$ and 
        $$\chi^2(\tilde{\pi}_i\|\pi_i) \leq \tilde{\epsilon}.$$   
\end{assm}
\begin{assm}[Good path]\label{a: good path} {There is $c_p>0$ such that} for any two windows indexed by $i,j \in [1,L]$, there exist at least one path $(i= k_0,\dots,k_l=j)$ such that
        \[\min\{ F_{k_t,k_{t+1}},{F_{k_{t+1},k_{t}}\} \geq c_p.}
        \]
\end{assm}

The following is a corollary of Theorem \ref{t: general error decomposition} which makes the error decomposition  explicit when the eigenvector is computed from the sparse matrix $\bar{F}^0$ and the weighted relative standard deviation $\mathcal{Q}_{\{\mu\}}(w,f)$ is bounded. 

\begin{corollary} \label{c: main_theorem}
Let Assumptions \ref{a: good truncation},  \ref{a: good samples} and \ref{a: good path} hold. Let $\bar{\pi}[g]$ be the EMUS estimator computed on $L$ window using the sparse matrix $\bar F^0$, defined on $G$, with $m = |G|$.
For each $f \in \{g,1\}$, suppose  
$$\max\left\{Q_{\{\pi\}}(w,f),Q_{\{\widetilde\pi\}}(w,f)\right\}\le C_{\mathcal Q_*}.$$
Suppose every window contains $N_k \ge N$ samples. Choose $\epsilon_w>0$, a weight failure probability $\delta_w>0$, and an observable sampling failure probability $\delta_f>0$. If 
\begin{align*}
N&\ge \max\left\{\frac{4C_{\mathcal Q_*}^2}{\delta_f},\frac{32mL^2}{\epsilon_w^2c_p\delta_w}\right\}\\
\shortintertext{and}
\tilde\epsilon
&\le
\min\left\{\frac{1}{4C_{\mathcal Q_*}^2},\frac{\epsilon_w^2c_p}{16L^2}\right\},
\end{align*}
then defining 
$$\eta \vcentcolon =2\epsilon_w + \frac{2C_{\mathcal Q_*}}{\sqrt{N\delta_f}} + 2\sqrt{\widetilde\epsilon}\,C_{\mathcal Q_*},$$
yields, for each $f \in \{g,1\}$, with probability $1 - \delta_w- \delta_f $,
$$e^{-\eta}\mathcal S(w,f)
\leq
\bar{\mathcal S}(\bar w,f)
\leq
e^{\eta}\mathcal S(w,f)$$
and consequently,  with probability $1 - \delta_w- 2\delta_f $,
$$\left|
\log\bar\pi[g]-\log\pi[g]
\right|
\leq
2\eta
+
\max\left\{
\log\frac{1}{1-\delta_\pi},
\log\frac{1}{1-\delta_g}
\right\}.$$
\end{corollary}

\subsection{Error Bounds (Decomposed)}\label{ss: decomposed error}
In this subsection we bound each component of the error explicitly. We later use these bounds to prove Corollary \ref{c: main_theorem}.

\begin{lem} \label{l: upper/lower} Suppose  $\max\left\{
Q_{\{\pi\}}(w,f),Q_{\{\widetilde\pi\}}(w,f)
\right\}
\le C_{\mathcal Q,f}$ and suppose $\max_k\left\{
\frac{\bar w_k}{w_k},
\frac{w_k}{\bar w_k}
\right\}
\le e^{C_w}$, with probability $1 - \delta_w$. Let every window have $N_k \geq N$ samples. If 
\begin{align*}
N&\ge \frac{4C_{\mathcal Q,f}^2}{\delta_f}\\
\shortintertext{and}
\tilde\epsilon
&\le
\frac{1}{4C_{\mathcal Q,f}^2},
\end{align*}
then, with probability at least $1-\delta_f - \delta_w$, 
$$e^{-\eta_f}S(w,f) \le \bar S (\bar w, f) \le e^{\eta_f}S(w,f)$$
where $\eta_f = C_w + \frac{2C_{\mathcal Q,f}}{\sqrt{N\delta_f}} + 2\sqrt{\widetilde\epsilon}\,C_{\mathcal Q,f}.$ Equivalently, the log relative ratio bound is given by $\left|
\log\frac{\bar{\mathcal S}(\bar w,f)}
{\mathcal S(w,f)}
\right|
\le \eta_f.$
\end{lem}
\begin{proof}
    Follows from combining the bounds from Lemmas \ref{l: weight erorr}, \ref{l: Sampling Error} and \ref{l: bias error} with no additional work.
\end{proof}

\begin{lem}[Weight error]\label{l: weight erorr} Assume $\max\{\frac{\bar{w}_k}{w_k},\frac{w_k}{\bar{w}_k}\}\leq e^{C_w}$, with probability $1-\delta_w$, for all $k$. Then, with probability $1-\delta_w$, 
$$e^{-C_w}\bar{\mathcal{S}}({w}, f)\leq\bar{\mathcal{S}}(\bar{w}, f)\leq e^{C_w}\bar{\mathcal{S}}({w}, f).$$
This implies the log relative ratio bound $\left|\log\frac{\bar{\mathcal{S}}(\bar{w}, f)}{\bar{\mathcal{S}}({w}, f)}\right| \leq C_w.$

\end{lem}
\begin{proof}
      For any $f$, 
      \begin{align*}
\frac{\bar{\mathcal{S}}(\bar{w}, f)}{\bar{\mathcal{S}}({w}, f)} = \frac{\sum_{k=1}^M\bar{w}_k\bar{f}^*_k}{\sum_{k=1}^Mw_k\bar{f}^*_k} &= \frac{\sum_{k=1}^Mw_k\frac{\bar{w}_k}{w_k}\bar{f}^*_k}{\sum_{k=1}^Mw_k\bar{f}^*_k}.\\
\shortintertext{Then $\max\{\frac{\bar{w}_k}{w_k},\frac{w_k}{\bar{w}_k}\}\le e^{C_w}$ implies that $e^{-C_w} \le \frac{\bar{w}_k}{w_k} \le e^{C_w}$; therefore }
e^{-C_w}= e^{-C_w} \frac{\sum_{k=1}^Mw_k\bar{f}^*_k}{\sum_{k=1}^Mw_k\bar{f}^*_k}\leq\frac{\sum_{k=1}^Mw_k\frac{\bar{w}_k}{w_k}\bar{f}^*_k}{\sum_{k=1}^Mw_k\bar{f}^*_k} &\leq e^{C_w} \frac{\sum_{k=1}^Mw_k\bar{f}^*_k}{\sum_{k=1}^Mw_k\bar{f}^*_k} = e^{C_w}.\\
\shortintertext{Hence}
e^{-C_w}\bar{\mathcal{S}}({w}, f)\leq\bar{\mathcal{S}}(\bar{w}, f)&\leq e^{C_w}\bar{\mathcal{S}}({w}, f),
      \end{align*}
      which implies 
$$\left|\log\frac{\bar{\mathcal{S}}(\bar{w}, f)}{\bar{\mathcal{S}}({w}, f)}\right| \leq C_w.$$

\end{proof}
\begin{lem}[Sampling error] \label{l: Sampling Error} Assume  $N_k$ i.i.d samples from $\tilde{\pi}_k$ in each window with $N_k \geq N$, where $N\ge\frac{4\mathcal{Q}_{\{\tilde{\pi}\}}(w,f)^2}{\delta}$. Then with probability at least $1 - \delta$: 
$$e^{-\frac{2}{\sqrt{N\delta}}\mathcal{Q}_{\{\tilde{\pi}\}}(w,f)}\tilde{\mathcal{S}}(w, f) \leq \bar{\mathcal{S}}(w, f) \leq e^{\frac{2}{\sqrt{N\delta}}\mathcal{Q}_{\{\tilde{\pi}\}}(w,f)}\tilde{\mathcal{S}}(w, f).$$
Equivalently, the log relative ratio bound is $\left| \log \frac{\bar{\mathcal{S}}(w, f)}{\tilde{\mathcal{S}}(w, f)}\right|  \leq \frac{2}{\sqrt{N\delta}}\mathcal{Q}_{\{\tilde{\pi}\}}(w,f)$.
\end{lem}
\begin{proof}
For any $f$, 
$$ \left| \log \frac{\bar{\mathcal{S}}(w, f)}{\tilde{\mathcal{S}}(w, f)}\right| =  \left| \log \frac{\sum_{k=1}^Mw_k\frac{1}{N_k}\sum_{i=1}^{N_k}f^*(x_k^{(i)})}{\sum_{k=1}^M w_k \mathbb{E}_{\tilde{\pi}_k}(f^*)}\right| = \left| \log 1+\frac{\sum_{k=1}^Mw_k\bigg(\frac{1}{N_k}\sum_{i=1}^{N_k}f^*(x_k^{(i)})-\mathbb{E}_{\tilde{\pi}_k}(f^*)\bigg)}{\sum_{k=1}^M w_k \mathbb{E}_{\tilde{\pi}_k}(f^*)}\right|.$$

Let $A =\frac{\sum_{k=1}^Mw_k\bigg(\frac{1}{N_k}\sum_{i=1}^{N_k}f^*(x_k^{(i)})-\mathbb{E}_{\tilde{\pi}_k}(f^*)\bigg)}{\sum_{k=1}^M w_k \mathbb{E}_{\tilde{\pi}_k}(f^*)}$ then, by Lemma \ref{l: log_pert_bound}, assuming $|A| \leq \frac{1}{2}$, 

$$\left|\log 1+\frac{\sum_{k=1}^Mw_k\bigg(\frac{1}{N_k}\sum_{i=1}^{N_k}f^*(x_k^{(i)})-\mathbb{E}_{\tilde{\pi}_k}(f^*)\bigg)}{\sum_{k=1}^M w_k \mathbb{E}_{\tilde{\pi}_k}(f^*)}\right| \leq 2 \left|\frac{\sum_{k=1}^Mw_k\bigg(\frac{1}{N_k}\sum_{i=1}^{N_k}f^*(x_k^{(i)})-\mathbb{E}_{\tilde{\pi}_k}(f^*)\bigg)}{\sum_{k=1}^M w_k \mathbb{E}_{\tilde{\pi}_k}(f^*)} \right| $$
Applying Chebyshev with $N_k \geq N$ for all $k$ yields, with probability $1-\delta$, 
$$2 \left|\frac{\sum_{k=1}^Mw_k\bigg(\frac{1}{N_k}\sum_{i=1}^{N_k}f^*(x_k^{(i)})-\mathbb{E}_{\tilde{\pi}_k}(f^*)\bigg)}{\sum_{k=1}^M w_k \mathbb{E}_{\tilde{\pi}_k}(f^*)} \right| \leq 2\left|\frac{\frac{1}{\sqrt{N\delta}}\sum_{k=1}^Mw_k\sqrt{\Var_{\tilde{\pi}_k}(f^*)}}{\sum_{k=1}^M w_k \mathbb{E}_{\tilde{\pi}_k}(f^*)}\right| =  \frac{2}{\sqrt{N\delta}}\mathcal{Q}_{\{\tilde{\pi}\}}(w,f).$$
Since $|A| \leq  \frac{1}{\sqrt{N\delta}}\mathcal{Q}_{\{\tilde{\pi}\}}(w,f)$ enforcing $ \frac{1}{\sqrt{N\delta}}\mathcal{Q}_{\{\tilde{\pi}\}}(w,f) \leq \frac{1}{2}$, or equivalently $N \geq \frac{4\mathcal{Q}_{\{\tilde{\pi}\}}(w,f)^2}{\delta}$, is sufficient for $A \leq \frac{1}{2}$.
Given 
$$\left| \log \frac{\bar{\mathcal{S}}(w, f)}{\tilde{\mathcal{S}}(w, f)}\right|  \leq \frac{2}{\sqrt{N\delta}}\mathcal{Q}_{\{\tilde{\pi}\}}(w,f)$$
implies 
$$-\frac{2}{\sqrt{N\delta}}\mathcal{Q}_{\{\tilde{\pi}\}}(w,f) \leq \log \frac{\bar{\mathcal{S}}(w, f)}{\tilde{\mathcal{S}}(w, f)} \leq \frac{2}{\sqrt{N\delta}}\mathcal{Q}_{\{\tilde{\pi}\}}(w,f).$$
Exponentiating both sides and multiplying by $\tilde{\mathcal{S}}(w, f)$ yields 
$$e^{-\frac{2}{\sqrt{N\delta}}\mathcal{Q}_{\{\tilde{\pi}\}}(w,f)}\tilde{\mathcal{S}}(w, f) \leq \bar{\mathcal{S}}(w, f) \leq e^{\frac{2}{\sqrt{N\delta}}\mathcal{Q}_{\{\tilde{\pi}\}}(w,f)}\tilde{\mathcal{S}}(w, f).$$
\end{proof}

\begin{lem}[Bias error]\label{l: bias error} Let $\tilde{\epsilon} >0$ bound the $\chi^2$-divergence between $\tilde{\pi}_k$ and $\pi_k$ as in Assumption \ref{a: good samples}, with $\tilde{\epsilon} \leq \frac{1}{4\mathcal \mathcal{Q}_{\{{\pi}\}}(w,f)^2}$. Then 
$$e^{-2\sqrt{\tilde{\epsilon}}\;\mathcal{Q}_{\{{\pi}\}}(w,f)}{\mathcal{S}}(w, f)\leq {\tilde{\mathcal{S}}(w, f)}\leq e^{2 \sqrt{\tilde{\epsilon}}\;\mathcal{Q}_{\{{\pi}\}}(w,f)}{\mathcal{S}}(w, f).$$
Equivalently, the log relative ratio bound is $\left| \log \frac{\tilde{\mathcal{S}}(w, f)}{{\mathcal{S}}(w, f)}\right| \leq 2\sqrt{\tilde{\epsilon}}\; \mathcal{Q}_{\{{\pi}\}}(w,f).$
\end{lem}
\begin{proof}
    For any f, 
    $$\left| \log \frac{\tilde{\mathcal{S}}(w, f)}{{\mathcal{S}}(w, f)}\right| = \left|\log 1+ \frac{\sum_{k=1}^M w_k \bigg(\mathbb{E}_{\tilde{\pi}_k}(f^*)-\mathbb{E}_{{\pi_k}}(f^*)\bigg)}{\sum_{k=1}^M w_k \mathbb{E}_{{\pi}_k}(f^*)} \right| = \left|\log 1+ \frac{\sum_{k=1}^M w_k\int_\Omega\big(\frac{d\tilde{\pi}_k}{d\pi_k}-1\big)f^*\pi_k(dx)}{\sum_{k=1}^M w_k \mathbb{E}_{{\pi_k}}(f^*)} \right|.$$
    Let $A=\frac{\sum_{k=1}^M w_k\int_\Omega\big(\frac{d\tilde{\pi}_k}{d\pi_k}-1\big)f^*\pi_k(dx)}{\sum_{k=1}^M w_k \mathbb{E}_{{\pi_k}}(f^*)}$ then, by Lemma \ref{l: log_pert_bound}, assuming $A \leq \frac{1}{2}$
    $$\left|\log 1+ \frac{\sum_{k=1}^M w_k\int_\Omega\big(\frac{d\tilde{\pi}_k}{d\pi_k}-1\big)f^*\pi_k(dx)}{\sum_{k=1}^M w_k \mathbb{E}_{{\pi_k}}(f^*)} \right| \leq 2 \left|\frac{\sum_{k=1}^M w_k \bigg(\mathbb{E}_{\tilde{\pi}_k}(f^*)-\mathbb{E}_{{\pi_k}}(f^*)\bigg)}{\sum_{k=1}^M w_k \mathbb{E}_{{\pi}_k}(f^*)}\right|.$$
    Applying Cauchy-schwarz yields, 
    $$2\left|\frac{\sum_{k=1}^M w_k \bigg(\mathbb{E}_{\tilde{\pi}_k}(f^*)-\mathbb{E}_{{\pi_k}}(f^*)\bigg)}{\sum_{k=1}^M w_k \mathbb{E}_{{\pi}_k}(f^*)}\right|\leq 2\frac{\sum_{k=1}^M w_k\sqrt{\Var_{\pi_k}(f^*)}\sqrt{\chi^2(\tilde{\pi}_k \; || \; \pi_k)}}{\sum_{k=1}^M w_k \mathbb{E}_{{\pi_k}}(f^*)} \leq 2 \sqrt{\tilde{\epsilon}}\;\mathcal{Q}_{\{{\pi}\}}(w,f).$$
    Since $|A| \leq   \sqrt{\tilde{\epsilon}}\;\mathcal{Q}_{\{{\pi}\}}(w,f)$ enforcing $ \sqrt{\tilde{\epsilon}}\;\mathcal{Q}_{\{{\pi}\}}(w,f) \leq \frac{1}{2}$, or equivalently $\tilde{\epsilon} \leq \frac{1}{4\mathcal Q (w,f)^2}$, is sufficient for $A \leq \frac{1}{2}$. Given 
    $$\left| \log \frac{\tilde{\mathcal{S}}(w, f)}{{\mathcal{S}}(w, f)}\right|\leq 2 \sqrt{\tilde{\epsilon}}\;\mathcal{Q}_{\{{\pi}\}}(w,f)$$
    implies 
$$-2\sqrt{\tilde{\epsilon}}\;\mathcal{Q}_{\{{\pi}\}}(w,f)\leq \log\frac{\tilde{\mathcal{S}}(w, f)}{{\mathcal{S}}(w, f)}\leq 2 \sqrt{\tilde{\epsilon}}\;\mathcal{Q}_{\{{\pi}\}}(w,f)$$
Exponentiating both sides and multiplying by ${\mathcal{S}}(w, f)$ yields
$$e^{-2\sqrt{\tilde{\epsilon}}\;\mathcal{Q}_{\{{\pi}\}}(w,f)}{\mathcal{S}}(w, f)\leq {\tilde{\mathcal{S}}(w, f)}\leq e^{2 \sqrt{\tilde{\epsilon}}\;\mathcal{Q}_{\{{\pi}\}}(w,f)}{\mathcal{S}}(w, f).$$

\end{proof}

\begin{lem}[Truncation error]\label{l: truncation error}
     Let Assumption \ref{a: good truncation} hold with influence set $K_U \subseteq \Omega$ with $\delta_\pi, \delta_g > 0$ then 
    $$\mathcal{E}_{trunc} = \left|\log \frac{S(w,g)}{S(w,1)} - \log \pi[g] \right| \leq \max\left\{\log\frac{1}{1-\delta_\pi}, \log\frac{1}{1-\delta_g}\right\}.$$
\end{lem}
\begin{proof}
    Given the true weights $w_k = Z_k$ and that $\pi_k \propto \psi_k\cdot\pi(x)|_{K_U}$ we have 
    \begin{align*}
        \left|\log \frac{S(w,g)}{S(w,1)} - \log \pi[g] \right| &= \left|\log \frac{\sum_k Z_k \int_{K_U} \frac{g(x)}{\sum_j\psi_j(x)} \frac{\psi_k(x)\pi(x)}{Z_k}dx}{\sum_k Z_k \int_{K_U} \frac{1}{\sum_j\psi_j(x)} \frac{\psi_k(x)\pi(x)}{Z_k}dx} -\log\pi[g] \right|\\
        &= \left|\log \frac{ \int_{K_U} g(x)\frac{\sum_k\psi_k(x)}{\sum_k\psi_k(x)} \pi(dx)}{\int_{K_U} \frac{\sum_k \psi_k(x)}{\sum_k\psi_k(x)}\pi(dx)} -\log\pi[g] \right|\\
        &= \left|\log \frac{ \int_{K_U} g(x) \pi(dx)}{\int_{K_U}\pi(dx)} -\log\pi[g] \right|
    = \left|\log \frac{ \int_{K_U} g(x) \pi(dx)}{\pi(K_U)\pi[g]}\right|.
    \end{align*}
    Paying careful attention to the sign of $\log$ we will apply Assumptions \ref{a: good truncation}
    to upper bound the quotient and lower bound the quotient:  
 $\frac{ \int_{K_U} g(x) \pi(dx)}{\pi(K_U)\pi[g]} \leq \frac{1}{1-\delta_\pi}$ and $\frac{ \int_{K_U} g(x) \pi(dx)}{\pi(K_U)\pi[g]}\geq 1-\delta_g$. this yields
$$\left|\log \frac{ \int_{K_U} g(x) \pi(dx)}{\pi(K_U)\pi[g]}\right|\leq \max\left\{\log\frac{1}{1-\delta_\pi}, \log\frac{1}{1-\delta_g}\right\},$$
where we note that $-\log(1-x) = \log\frac{1}{1-x}$.
\end{proof}
\subsection{Eigenvector Error} \label{ss: eigenvector error}
In this subsection, we first show that the sparse matrix $F^0$ can be used as a substitute for the full matrix $F$, as both matrices have the same eigenvector. 
We then show that the empirically estimated sparse matrix is a good relative approximate of the true sparse matrix. 
The relative error of the empirical sparse matrix is later used to bound the relative error of the eigenvector.

\begin{proposition}[Sparse matrix] \label{p: sparse matrix}
    Let $F$ be an $L\times L$ irreducible and reversible stochastic matrix. Let $G$ be a set of index pairs $(i,j)$ representing a good path structure, such that: 
\begin{enumerate}
    \item $(i,i) \in G$ for all $i =1, \dots, L$.
    \item if $(i,j) \in G$ then $(j,i) \in G$.
    \item if $(i,j) \in G$ then $F_{ij} \geq c$ and $F_{ji} \geq c$ for some constant $c >0.$ 
    \item The graph defined by $G$ is connected (satisfying Assumption \ref{a: good path}).
\end{enumerate}
We define a new sparse matrix $F^0$ as follows: 
\begin{equation}\label{e: F^0}
\begin{array}{cc}
  F_{ij}^0 = \bigg\{ & 
    \begin{array}{cc}
      F_{ij} & \text{ if } (i,j) \in G \text{ and } i \neq j\\
      F_{ii} + \sum_{k \notin G, k \neq i} F_{ik} & \text{ if } i = j\\
      0 & \text{ if } (i,j) \notin G
    \end{array}.
\end{array}
\end{equation}
Then, the matrix $F^0$ is also an irreducible, reversible and stochastic matrix and it has the same unique stationary distribution as $F$. 
\end{proposition}  
\begin{proof}
    $F^0$ is stochastic since the row sum satisfies 
    $$\sum_{j=1}^L F^0_{ij} = F_{ii} + \sum_{k \notin G, k \neq i} F_{ik} + \sum_{(i,j) \in G} F_{ij} = \sum_{j=1}^L F_{ij} =1.$$
$F^0$ is irreducible since the graph $G$ is connected. Lastly, $F$ is reversible therefore for stationary distribution $\pi$ $F$ satisfies detailed balance i.e. $\pi_i F_{ij} = \pi_jF_{ji}$. By the symmetry assumption on $(i,j) \in G$ (2), the detailed balance equations hold for $\pi_i F_{ij}^0 = \pi_jF_{ji}^0$, for stationary distribution $\pi$. Therefore, $F^0$ is reversible with stationary distribution $\pi$ and since $F^0$ is irreducible $\pi$ is unique. 
\end{proof}

\begin{lem}\label{l: Fij approx} Let Assumptions \ref{a: good samples} and \ref{a: good path} hold. Define $\bar{F}_{ij} = \frac{1}{N_i} \sum_{t=1}^{N_i}\psi^*_j(X^i_t)$ with i.i.d samples $\{X^i_t\}_{t=1}^{N_i} \sim \tilde{\pi}_i$, $\tilde{F}_{ij} = \mathbb{E}_{\tilde{\pi}_i} \psi^*_j(x)$ and $F_{ij} = \mathbb{E}_{\pi_i} \psi^*_j(x)$. Given $\epsilon > 0$, for any fixed $(i,j) \in G$, if $N_i \geq \frac{32L^2}{\epsilon^2c_p\delta}$ and $\chi^2(\tilde{\pi}_i\;||\;\pi_i)<\tilde{\epsilon}\leq\frac{c_p\epsilon^2}{16L^2}$, then with probability $1-\delta$
$$|\frac{\bar{F}_{ij}}{F_{ij}}-1| \leq \frac{\epsilon}{2L}.$$

\end{lem}
\begin{proof}
    We decompose the error into the error from statistical estimate and the error from the biased distribution
    $$\left|\bar{F}_{ij} - F_{ij}\right| \leq \left|\bar{F}_{ij}-\tilde{F}_{ij}\right|+\left|\tilde{F}_{ij} -F_{ij} \right|.$$
    Then consider the error from the biased distribution
    $$\left|\tilde{F}_{ij} -F_{ij} \right| = \int_\Omega \psi_j^*\big(\frac{d\tilde{\pi}_i}{d\pi_i}-1\big)d\pi_i\leq\sqrt{\Var_{\pi_i(x)}(\psi_j^*)}\cdot\sqrt{\chi^2(\tilde{\pi}_i \; || \; \pi_i)},$$
    by Cauchy-Schwarz. Recalling that $\psi_j^*(x) \in [0,1]$ for all $x \in \Omega$ and $F_{ij} = \mathbb{E}_{\pi_i}\psi_j^*$ yields
    $$\sqrt{\Var_{\pi_i(x)}(\psi_j^*)}\cdot\sqrt{\chi^2(\tilde{\pi}_i \; || \; \pi_i)} \leq \sqrt{F_{ij}}\cdot \sqrt{\chi^2(\tilde{\pi}_i \; || \; \pi_i)}\leq  \sqrt{F_{ij}}\cdot \sqrt{\tilde{\epsilon}},$$
    by Assumption \ref{a: good path}. Therefore we have
    \begin{equation}\label{e: Fij approx}
        \left|\tilde{F}_{ij} -F_{ij} \right| \leq \sqrt{F_{ij}}\cdot \sqrt{\tilde{\epsilon}}.
    \end{equation}
The statistical error can be bounded by Chebyshev's inequality with pr. $1-\delta$, 
    $$\left|\bar{F}_{ij}-\tilde{F}_{ij}\right| = \left|\frac{1}{N_i}\sum_{t=1}^{N_i}\psi_j^*(x_t) - \mathbb{E}_{\tilde{\pi}_i} \psi_j^*\right| \leq \frac{\sqrt{\var_{\tilde{\pi}_i}(\psi_j^*)}}{\sqrt{N_i\delta}} \leq \frac{\sqrt{\tilde{F}_{ij}}}{\sqrt{N_i\delta}}.$$
Using the bound in (\ref{e: Fij approx}) yields, 
$$\left|\bar{F}_{ij}-\tilde{F}_{ij}\right| \leq \frac{\sqrt{F_{ij}+\sqrt{\tilde{\epsilon}F_{ij}}}}{\sqrt{N_i\delta}}.$$
Combining both bounds gives us
$$\left|\bar{F}_{ij} -F_{ij} \right| \leq \sqrt{F_{ij}}\cdot \sqrt{\tilde{\epsilon}}+\frac{\sqrt{F_{ij}+\sqrt{\tilde{\epsilon}F_{ij}}}}{\sqrt{N_i\delta}}.$$
Dividing through by $F_{ij}$ and recalling Assumption \ref{a: good path}  says $F_{ij} \geq c_p$ yields, 
$$\left|\frac{\bar{F}_{ij}}{F_{ij}}-1\right| \leq \sqrt{\frac{\tilde{\epsilon}}{F_{ij}}} + \sqrt{\frac{\frac{1}{F_{ij}}+\sqrt{\frac{\tilde{\epsilon}}{F_{ij}^3}}}{N_i\delta}}\leq \sqrt{\frac{\tilde{\epsilon}}{c_p}} + \sqrt{\frac{\frac{1}{c_p}\big(1+\sqrt{\frac{\tilde{\epsilon}}{c_p}}\big)}{N_i\delta}}.$$
Applying $\tilde{\epsilon} \le \frac{c_p\epsilon^2}{16L^2}$ and $N_i \geq \frac{32L^2}{\epsilon^2c_p\delta}$ yields the desired result.
\end{proof}

\begin{lem}\label{l: O'cinneide bound} \textbf{(O'Cinneide bound)} Let $A$, $\tilde{A}$ be $L \times L$ irreducible stochastic matrices with stationary distributions $\pi, \tilde{\pi}$ satisfying $A_{ij} = 0 \iff \tilde{A}_{ij} = 0$ then 

$$\max_k \max\{\frac{\tilde{\pi}_k}{\pi_k}, \frac{\pi_k}{\tilde{\pi}_k}\} \leq \bigg(\max_{i\neq j}\max\{\frac{\tilde{A}_{ij}}{A_{ij}}, \frac{A_{ij}}{\tilde{A}_{ij}}\} \bigg)^{2L}.$$
\end{lem}
\begin{rem}
There is an alternative estimator which takes the product of telescoping ratios between the relative umbrella weights, satisfying Assumption \ref{a: good path}. 
This estimator achieves guarantees with the maximal path length $\ell$  in place of $L$, the total number of windows. 
It remains an open question whether Lemma \ref{l: O'cinneide bound} can be replaced with a bounds dependent only on $\ell$. 
\end{rem}
\begin{lem}\label{l: weight error precise}
    Let $F$ be an irreducible, and reversible stochastic matrix with stationary distribution $w$. Let $\bar{F}^0$ be the empirical sparse matrix define by $G$, with $m = |G|$, and computed from $N>N_i$ (for all $i$) samples per window. Let the assumptions and  conditions of Lemma \ref{l: Fij approx}, $N_i \geq \frac{32mL^2}{\epsilon^2c_p\delta} $ and $\tilde{\epsilon}\leq\frac{\epsilon^2c_p}{16L^2}$, hold. Then, for any $\epsilon > 0$, with probability $1-\delta$:
    
    $$\max_k\max\{\frac{\bar{w}^0_k}{w_k},\frac{w_k}{\bar{w}^0_k}\}\leq e^{2\epsilon} \text{ for all } k \in [1,L].$$
\end{lem}
\begin{proof}
    The proof can be decomposed as follows: 
    \begin{enumerate}
        \item For each $(i,j) \in G$, by Lemma \ref{l: Fij approx}, Given $N_i \geq \frac{32mL^2}{\epsilon^2c_p\delta}$ samples, with probability $1-\frac{\delta}{m}$, the entries of the empirical sparse matrix satisfy the bound
        $$1 - \frac{\epsilon}{2L}\leq \frac{\bar{F}_{ij}^0}{F_{ij}^0}\leq 1 + \frac{\epsilon}{2L}.$$
        The reciprocal is bounded by 
        $$\frac{F^0_{ij}}{\bar{F}^0_{ij}}\leq \frac{1}{1-\frac{\epsilon}{L}}\leq 1 + \frac{2\epsilon}{2L}.$$
        Therefore the maximum of the two quotients is upper bounded by 
        $$\max_{(i,j) \in G}\max\{\frac{\bar{F}_{ij}^0}{F_{ij}^0},\frac{F_{ij}^0}{\bar{F}_{ij}^0}\}\leq 1 + \frac{\epsilon}{L}.$$
        Note that we can apply Lemma \ref{l: Fij approx} to the sparse matrices since the non-zero entries of the sparse matrix remain unchanged for $i\neq j$.
\item By Proposition \ref{p: sparse matrix}, the sparse matrix $F^0$ and $F$ share the same stationary distribution. Therefore $w(F^0) = w(F)=w$. Applying this to Lemma \ref{l: O'cinneide bound} we have, with probability $1 - m\cdot\frac{\delta}{m}$,
$$\max_k\max\{\frac{\bar{w}^0_k}{w_k},\frac{w_k}{\bar{w}^0_k}\} \leq \bigg(\max_{(i,j) \in G}\max\{\frac{\bar{F}_{ij}^0}{F_{ij}^0},\frac{F_{ij}^0}{\bar{F}_{ij}^0}\} \bigg)^{2L}.$$
\item Combining the bounds yields 
$$\max_k\max\{\frac{\bar{w}^0_k}{w_k},\frac{w_k}{\bar{w}^0_k}\} \leq \bigg(1 + \frac{\epsilon}{L}\bigg)^{2L}\leq e^{2\epsilon}.$$
    \end{enumerate}
\end{proof}
\subsection{Proof of Section Main Result} 
\begin{proof}
    Using that $\bar{\pi}[g] = \frac{\bar{\mathcal S}(\bar w, g)}{\bar{\mathcal{S}}(\bar w, 1)}$ we consider the following:
    \begin{align*}
        \left|\log\bar\pi[g]-\log\pi[g]\right| &= \left| \log\frac{\bar{\mathcal S}(\bar w, g)}{\bar{\mathcal{S}}(\bar w, 1)} - \log\frac{{\mathcal S}(w, g)}{{\mathcal{S}}( w, 1)} + \log\frac{{\mathcal S}(w, g)}{{\mathcal{S}}( w, 1)} - \log\pi[g]\right|\\
        &\leq \left|\log\frac{\bar{\mathcal S}(\bar w, g)}{\bar{\mathcal{S}}(\bar w, 1)} - \log\frac{{\mathcal S}(w, g)}{{\mathcal{S}}( w, 1)} \right| + \left|\log\frac{{\mathcal S}(w, g)}{{\mathcal{S}}( w, 1)} - \log\pi[g] \right|\\
        &=\left|\log\frac{\bar{\mathcal S}(\bar w, g)}{{\mathcal S}(w, g)} - \log\frac{\bar{\mathcal{S}}(\bar w, 1)}{{\mathcal{S}}( w, 1)}\right| + \left|\log\frac{{\mathcal S}(w, g)}{{\mathcal{S}}( w, 1)} - \log\pi[g] \right|
    \end{align*}
     %
    Then by assumption $e^{-\eta_f^-}\mathcal S(w,f)\leq\bar{\mathcal S}(\bar w,f)\leq e^{\eta_f^+}\mathcal S(w,f),$ which implies that $e^{-\eta_f^-}\leq\frac{\bar{\mathcal S}(\bar w,f)}{\mathcal S(w,f)}\leq e^{\eta_f^+}.$ Taking the $\log$ of both sides implies $-\eta_f^- \leq \log \frac{\bar{\mathcal S}(\bar w,f)}{\mathcal S(w,f)} \leq {\eta_f^+}.$
    Therefore 
    $$\left|\log\frac{\bar{\mathcal S}(\bar w, g)}{{\mathcal S}(w, g)} - \log\frac{\bar{\mathcal{S}}(\bar w, 1)}{{\mathcal{S}}( w, 1)}\right| \leq \max\left\{\left|-(\eta_g^- +\eta_1^+)\right|, \eta_g^+ + \eta_1^-\right\} = \max\left\{\eta_g^- +\eta_1^+, \eta_g^+ + \eta_1^-\right\}.$$
    Furthermore, applying Lemma \ref{l: truncation error} yields 
    $$\left|\log \frac{S(w,g)}{S(w,1)} - \log \pi[g] \right| \leq \max\left\{\log\frac{1}{1-\delta_\pi}, \log\frac{1}{1-\delta_g}\right\}.$$
    Combining the two bounds gives the final result. 
\end{proof}
Corollary \ref{c: main_theorem} follows by applying Lemma \ref{l: weight error precise} to show $\max_k\left\{
\frac{\bar w_k}{w_k},
\frac{w_k}{\bar w_k}
\right\}
\le e^{C_w}$, with probability $1 - \delta_w$ is satisfied for $C_w = 2\epsilon$. Then it is direct application of Lemma \ref{l: upper/lower} to Theorem \ref{t: general error decomposition} with $C_w = 2\epsilon$.

\section{Collective Variable Mapping \& Space} \label{s: collective variable mapping}
This section details the geometric framework for the collective variable mapping and the functional inequalities which quantify the mixing on the full space and the collective variable space (Section \ref{ss: log lipschitz constants}). 
Additionally, we quantity $S_{\min}, S_{\max}$ and $c_p$ (Section \ref{ss: gaussian net}).

To quantify the error bounds of the EMUS estimator we will rely on log-Sobolev inequalities to describe the concentration and mixing dynamics of the biased measure $\pi_k \propto \pi\cdot\psi_k$ on the full space $\mathbb{R}^D$.
We denote the LSI constant of the measure on $\mathbb{R}^D$ by $\CD$. 
In our setting, we only have control over this LSI constant by the choice of biasing functions $\psi_k$ in the collective variable space $\mathbb{R}^d$, and the choice of collective variable mapping $\theta: \mathbb{R}^D \rightarrow \mathbb{R}^d$. 

It is then natural to bound the LSI constant on the full space $\CD$ by the LSI constant $\Cd$ on the pushforward marginal $\mu_k$ in the collective variable space $\mathbb{R}^d$ and the LSI constant of the conditional measure $\pi_{k,z}$, along the fibers of the collective variable mapping denoted by $\Ctheta$.
The dependency of $\CD$ on $\Cd$ and $\Ctheta$ is made explicit in \cite{lelievre2009general} and we adapt their main result to our setting while keeping the notation general.

\subsection{Log-Sobolev Constant Bounds} \label{ss: log lipschitz constants}
We define the LSI constants for the decomposed components of a distribution $\pi$. Following \cite{lelievre2009general} we consider the marginal measure $\mu = \theta_\# \pi$ on $\mathbb{R}^d$ and $\pi_{z} = \pi(\cdot |\theta(x)=z)$ on $\Sigma_z$. 
Our goal is to bound the log-Sobolev constant of $\pi$ using the log-Sobolev constants for $\mu$ and $\pi_z$.
We will apply this to the distributions $\pi_k$.

We suppose:
\begin{enumerate}
    \item The marginal measure $\mu$ satisfies an LSI with constant $\Cd$: for all functions $f$ such that the RHS is finite,
    $$ \Ent_{\mu}(f^2) \le 2 C_{LS}^{(d)} \int_{\mathbb{R}^d} |\nabla_z f|^2 d\mu. $$
    \item The conditional measure $\pi_z$ satisfies an LSI with constant $\Ctheta$: for all functions $f$ such that the RHS is finite,
    $$ \Ent_{\pi_{z}}(f^2) \le 2 C_{LS}^{(\theta)} \int_{\Sigma_z} |P \nabla_x f|^2 d\pi_{z}, \quad \forall z \in \mathcal{A}. $$
\end{enumerate}
\begin{rem}\label{r: Bakry-Emery}\textup{(The Bobkov-Ledoux Bound) Because the potential $A(z)$ is $\ell$-smooth and the windows $\tilde{\psi}_k$ are Gaussians with inverse temperature $\beta$, the tilted marginal potential $A(z) + \frac{\beta}{2}\|z-y_k\|^2$ is strongly convex for $\beta > \ell$. Moreover, since $\mathcal{A}$ is convex, by Bobkov-Ledoux (Propositon 3.1 \cite{BobkovLedoux2000}), the LSI inequality holds and the LSI constant is bounded by 
$$C_{LS}^{(d)} \leq \frac{1}{\beta-\ell}.$$}
\end{rem}

 In this subsection, by adapting \cite{lelievre2009general} to our setting, we quantify $\CD$ in terms of $\Ctheta$ and $\Cd$. 
 Additionally, we prove a corollary of their result that applies to observable functions which are constant along the CV fibers. 
Their result is presented using the spectral gap notation $R = \frac{1}{\CD}$, $r = \frac{1}{\Cd}$ and $\rho = \frac{1}{\Ctheta}$.
For best readability we will present their result using the spectral gap notation and the corollary to our result in our notation; this will drastically simplify our later analysis.
\begin{theorem}(\cite[Theorem 1.2]{lelievre2009general})\label{t: LSI_bounds_CV_fiber} Let Assumptions \ref{a: collective variable} hold, 
and suppose $\mu$ satisfies a LSI with constant $C_{LS}^{(d)}$ and each $\pi_z$ satisfies a LSI with constant $C_{LS}^{(\theta)}$.
Then $\pi$ satisfies an LSI constant with constant
\begin{equation}\label{e: spectral gap notation LSI}
    R = \frac{1}{2} \left( rm + \frac{\Lambda^2 m}{\rho} + \rho - \sqrt{ \left( rm + \frac{\Lambda^2 m}{\rho} + \rho \right)^2 - 4 rm \rho } \right).
\end{equation}
Using the notation $R = \frac{1}{\CD}$, $r = \frac{1}{\Cd}$ and $\rho = \frac{1}{\Ctheta}$ it follows that 
\begin{equation}\label{e: LSI constant notation}
    \CD \leq \Ctheta + \frac{1}{m}\Cd + \Lambda^2\Cd(\Ctheta)^2.
\end{equation}
\end{theorem}
\begin{proof} We differ the proof of the spectral gap equality (\ref{e: spectral gap notation LSI}) to \cite{lelievre2009general}. 
We show the upper bound on the spectral gap algebraically as follows. 
Let $A = rm + \frac{\Lambda^2 m}{\rho} + \rho$ and $B = rm\rho$ then 
$$R = \frac{1}{2}\left( A  - \sqrt{A^2 - 4B} \right) = \frac{1}{2}\left( A  - \sqrt{A^2 - 4B} \right)\frac{ A  +\sqrt{A^2 - 4B} }{ A  + \sqrt{A^2 - 4B} } = \frac{1}{2}\frac{4B}{ A  + \sqrt{A^2 - 4B}} \geq \frac{B}{A}.$$
Where the last inequality comes applying $\frac{1}{A}\geq \frac{1}{\sqrt{A^2-4B}}.$ Therefore, 
$$\frac{1}{R} \leq \frac{A}{B} = \frac{rm + \frac{\Lambda^2m}{\rho} + \rho}{mr\rho} = \frac{1}{\rho} + \frac{\Lambda^2}{\rho^2r} + \frac{1}{mr}.$$
Change of notation with $R = \frac{1}{\CD}$, $r = \frac{1}{\Cd}$ and $\rho = \frac{1}{\Ctheta}$ yields 
$$\CD \leq \Ctheta + \frac{1}{m}\Cd + \Lambda^2\Cd(\Ctheta)^2.$$
\end{proof}
\subsection{Constructing a Gaussian Net}\label{ss: gaussian net}
This subsection makes explicit the choice of Gaussian lattice on the collective variable space $\mathbb{R}^d$.
All the analysis of this section is performed within the collective variable space $\mathbb{R}^d$. 
The Gaussian biasing functions $\tilde{\psi}_k \propto e^{-\beta\|x-y_k\|^2}$ are chosen with inverse temperature $\beta$. 
Choosing a high enough inverse temperature relative to the potential of the push-forward measure $\mu(z) \propto e^{-A(z)}$ allows for control over the spectral gap of the windows $\mu_k \propto \tilde{\psi}_ke^{-A(z)}$. 

\begin{center}
    \begin{tikzpicture}[node distance=2.5cm, auto]
  \node (RD) {$\mathbb{R}^D$};
  \node (Rd) [below of=RD] {$\mathbb{R}^d$};
  \node (R) [right of=Rd] {$\mathbb{R}$};

  \draw[->] (RD) to node [swap] {$\theta$ (CV map)} (Rd);
  \draw[->] (Rd) to node [swap] {$\tilde{\psi}_k$} (R);
  \draw[->, dashed] (RD) to node {$\psi_k$} (R);

  \node[right of=R, node distance=4cm, text width=4cm, font=\small] {
    $\psi_k(x) = \tilde{\psi}_k(\theta(x))$ 
  };
\end{tikzpicture}
\end{center}

Additionally, the Gaussian biasing functions are chosen to have lattice spacing $\Delta \propto \frac{1}{d\sqrt{\beta}}$. 
This condition enforces that the Gaussian biasing functions form a net $S(x)$ that fully covers the observable. 
Since, to enforce coverage, the lattice spacing of the biasing functions is chosen inversely proportional to the dimension this causes dimensioning signal in $\frac{\tilde{\psi}_k(z)}{\tilde{S}(z)}$. 
This causes the lower bound $c_p$ from Assumptions \ref{a: good path} to have an exponential dependence on the collective variable dimension $d$. 
However, in practice, $d$ is chosen small ($\leq 3$) and therefore the exponential dependence does not significantly impact our results.

Our results depend on the uniformity of $S(z)$, which we measure by 
\begin{equation}
    r_S := \frac{\max_{z\in\mathcal{A}} S(z)}{\min_{z\in\mathcal{A}}S(z)}.
\end{equation}
This ratio gives us some notion for the variation in the importance factor $\frac{1}{S(z)}$. 
This is necessary to control the relative variance of the quotient observable $\frac{g(z)}{S(z)}$. 
The following lemmas show that the padded Gaussian net can control the ratio $r_S$.

For a set $S$, let  $B_r(S)$ denote the neighborhood of radius $r$ around $S$.
\begin{proposition}
\label{p:sum-integral}
    Let $\La$ be a lattice and $C$ be a set in $\R^d$. Let $D$ be the determinant of the fundamental parallelopiped for $\La$ and $r = \min{r'}{B_{r'}(\La) = \R^d}$ be the covering radius of the lattice (the minimum $r'$ such that every point is within radius $r'$ of a point of $\La$).  
    Let $\psi  = e^{-U}:\R^d\to \R_{>0}$ be differentiable. 
    For a set $C$, let $C_r^+=\seta{x\in C^c}{d(x,\pl C)\le r}$ and $C_r^- = \seta{x\in C}{d(x,\pl C)\le r}$.
Then 
\begin{align*}    
D \sum_{x\in \La \cap C} \psi(x) - \int_C \psi(x)\,dx
&\le  \int_{\R^d} \psi(y) \pa{e^{r\sup_{z\in B_r(y)} \|\nb U(z)\|}-1}\,dy + 
\int_{C_r^+} \psi(x)\,dx\\
D \sum_{x\in \La \cap C} \psi(x) - \int_C \psi(x)\,dx
&\ge  \int_{\R^d} \psi(y) \pa{e^{-r\sup_{z\in B_r(y)} \|\nb U(z)\|}-1}\,dy - 
\int_{C_r^-} \psi(x)\,dx.
\end{align*}
\end{proposition}
Note that in particular, for $C=\R^d$, only the first term remains.
\begin{proof}
For $x\in \La$, $V_x = \seta{y\in \R^d}{d(y,x) = d(y,\La)}$; that is, it is the set of points whose closest lattice point is $x$. Note that each $V_x$ is a translate of the fundamental parallelopiped centered at $x$.
We can write
\begin{align*}
    D \sum_{x\in \La \cap C} \psi(x) - \int_C \psi(x)\,dx
    &= 
     \sum_{x\in \La \cap C} \ba{ \int_{V_x}(\psi(x) -\psi(y))\,dy }  + \ba{\int_{\bigcup_{x\in \La\cap C}V_x} \psi(x)\,dx - \int_C \psi(x)\,dx}.
\end{align*}
By the mean value theorem, the summand in the first term can be upper and lower bounded as
\begin{align*}
    \int_{V_x}(\psi(x) -\psi(y))\,dy
    &= \int_{V_x} \psi(y) (e^{-U(x)+U(y)} - 1)\,dy\\
    &\le \int_{V_x} \psi(y) (e^{\ve{x-y} \sup_{z\in \ol{xy}}\ve{\nb U(z)}}-1) \,dy\\
    &\le \int_{V_x} \psi(y) (e^{r\sup_{z\in B_r(y)}\ve{\nb U(z)}}-1) \,dy\\
    \int_{V_x}(\psi(x) -\psi(y))\,dy
    &\ge \int_{V_x} \psi(y) (e^{-\ve{x-y} \sup_{z\in \ol{xy}}\ve{\nb U(z)}}-1) \,dy\\
    &\ge \int_{V_x} \psi(y) (e^{-r\sup_{z\in B_r(y)}\ve{\nb U(z)}}-1) \,dy.
\end{align*}
Then summing over $x$ gives the first error terms.

For the second term, it suffices to show that 
\begin{align*}
    \pa{\bigcup_{x\in \La \cap C} V_x} \bs C &\subeq C_r^+\\
    C \bs \pa{\bigcup_{x\in \La \cap C} V_x}  &\subeq C_r^-.
\end{align*}
First, consider $y\in \pa{\bigcup_{x\in \La \cap C} V_x} \bs C$. Then $y\in V_x$ for some $x\in \La\cap C$, and $\ve{y-x}\le r$. Hence $d(y,C)\le r$, and because $y\nin C$, we have $y\in C_r^+$. 
Second, consider $y\in C\bs \pa{\bigcup_{x\in \La \cap C} V_x} $. Then $y\in V_x$ for some $x\in \La \bs C$. Since $\ve{y-x}\le r$, we get $d(y,C^c)\le r$ and hence $y\in C_r^-$.
\end{proof}
\[
\sup_{y\in B_r(x)} \ve{\nb \psi(y)}dx
\le \be (R+r) e^{-\fc{\be ( \max\{R-r,0\})^2}2}. 
\]

For the Gaussian windows, we derive a bound for the first error term.
\begin{lem}\label{l:gaussian-error}
    For $\psi(x)=\rc{(2\pi/\be)^{d/2}} e^{-\fc{\ve{x}^2}2}$ (the density of a gaussian with variance $1/\be$), 
    \begin{align*}
        \int_{\R^d} \psi(y) (e^{r\sup_{z\in B_r(y)} \|\nb U(z)\|}-1)\,dy &\le e^{\sqrt{d\be}r + \fc{3}2\be r^2}-1\\
        \int_{\R^d} \psi(y) (e^{-r\sup_{z\in B_r(y)} \|\nb U(z)\|}-1)\,dy &\ge -\sqrt{d\be }r - \be r^2.
    \end{align*}
\end{lem}
\begin{proof}
We calculate
\begin{align*}
    \int_{\R^d} \psi(y) (e^{r\sup_{z\in B_r(y)} \|\nb U(z)\|}-1)\,dy 
    &\le \int_{\R^d} \psi(y) (e^{\be r (\ve{y}+r)}-1) \,dy \\
\shortintertext{Noting that $\ve{y}$ is 1-Lipschitz, $N(0,1/\be)$ satisfies a log-Sobolev inequality with constant $1/\be$, and using sub-gaussianity of Lipschitz functions (Lemma \ref{l: herbst argument}),}
    &\le e^{\be r \E_{N(0,1/\be)} \ve{y} + \fc{\be r^2}2 + \be r^2}-1\\
\shortintertext{Noting $\E \ve{y}\le (\E \ve{y}^2)^{1/2}\le \sfc{d}{\be}$,}
    &\le e^{\sqrt{d\be}r + \fc{3\be r^2}2} - 1.
\end{align*}
For the second term, by Jensen's inequality,
\begin{align*}
    \int_{\R^d} \psi(y) (e^{-r\sup_{z\in B_r(y)} \|\nb U(z)\|}-1)\,dy
    &\ge \int_{\R^d} \psi(y) (e^{-r(\ve{y}+r)}-1)\,dy\\
    &\ge e^{-r(\E \ve{y}+r)}-1\\
    &\ge e^{-\sqrt{d\be} r - \be r^2}-1
    \ge -\sqrt{d\be} r - \be r^2.
\end{align*}
\end{proof}

\begin{lem}(Net Coverage Bounds Upper \& Lower) \label{l: Gaussian net coverage bounds} Let $S(z) = \sum_k \left(\frac{\beta}{2\pi}\right)^\frac{d}{2}\exp(-\frac{\beta}{2}\left\|z-y_k\right\|^2)$ with $z \in \mathcal{A}$. Moreover, let the Gaussian net form a lattice $A$ with spacing $\Delta = \left\|y_i - y_j\right\|$
then 
$$\left(\frac{\beta}{2\pi}\right)^\frac{d}{2}\exp(-\frac{d\beta\Delta^2}{2}) \leq S(z) \leq \left(\frac{\beta}{2\pi}\right)^\frac{d}{2} \left(1 + \left(\frac{2\pi}{\beta\Delta^2}\right)^\frac{1}{2}\right)^d.$$
\end{lem}
\begin{proof}
    We start with the lower bound. Define the covering radius $\eta$ of the set $Y = \{y_1, \dots, y_L\}$ relative to $K_u$ as the maximum distance between any point in $z \in \mathcal{A}$ and $Y$
    $$\eta \vcentcolon= \sup_{z \in \mathcal{A}}\min_{y_k \in Y} \left\|z-y_k \right\|.$$
    Then consider 
    $$\sum_k \left(\frac{\beta}{2\pi}\right)^\frac{d}{2}\exp(-\frac{\beta}{2}\left\|z-y_k\right\|^2) \geq \left(\frac{\beta}{2\pi}\right)^\frac{d}{2}\exp(-\frac{\beta}{2} \sup_{z \in \mathcal{A}}\min_{y_k \in Y} \left\|z-y_k \right\|) = \left(\frac{\beta}{2\pi}\right)^\frac{d}{2}\exp(-\frac{\beta\eta^2}{2}).$$
    If the Gaussian's form a lattice with spacing $\Delta = \|y_i-y_j\|$ then the covering radius is given by 
    $$\eta = \sqrt{\sum_{i=1}^d\left(\frac{\Delta}{2}\right)^2} \leq \sqrt{d\frac{\Delta^2}{4}}.$$
    Therefore 
    $$S(z) \geq \left(\frac{\beta}{2\pi}\right)^\frac{d}{2}\exp(-\frac{d\beta\Delta^2}{2}).$$
    To find the upper bound let $z = (z^{(1)}, \dots, z^{(d)})$ and $y_k = (y_k^{(1)}, \dots, y_k^{(d)})$ and note that 
    $$\left\| z- y_k\right\|^2 = \sum_{i=1}^d (z^{(i)} - y_k^{(i)})^2.$$
    Then $\psi_k(z)$ is expressible as 
    $$\psi_k(z) = \left(\frac{\beta}{2\pi}\right)^\frac{d}{2}\exp(-\frac{\beta}{2}\sum_{i=1}^d (z^{(i)} - y_k^{(i)})^2) = \left(\frac{\beta}{2\pi}\right)^\frac{d}{2}\prod_{i=1}^d\exp(-\frac{\beta}{2}(z^{(i)} - y_k^{(i)})^2).$$
    Now the Gaussian lattice in $\mathbb{R}^d$ the set of centers $Y = \{y_1, \dots, y_k\}$ is expressible as $Y=\{y = (n_1\Delta, \dots, n_d \Delta)\; | \; n_1, \dots, n_d \in \mathbb{Z}\}$. Therefore 
    $$S(z) = \left(\frac{\beta}{2\pi}\right)^\frac{d}{2}\sum_{n_1 \in \mathbb{Z}} \cdots \sum_{n_d \in \mathbb{Z}}\left(\prod_{i=1}^d\exp(-\frac{\beta}{2}(z^{(i)} - n_i\Delta)^2)\right)= \left(\frac{\beta}{2\pi}\right)^\frac{d}{2}\prod_{i=1}^d\left(\sum_{n_i\in \mathbb{Z}} \exp(-\frac{\beta}{2}(z^{(i)} - n_i\Delta)^2)\right).$$
    Each one dimensional sum, by symmetry, is maximized when $z^{(i)}$ is a lattice point hence 
    $$\sum_{n_i\in \mathbb{Z}} \exp(-\frac{\beta}{2}(z^{(i)} - n_i\Delta)^2) \leq 1 + 2\sum_{k=1}^\infty \exp(-\frac{\beta}{2}k^2\Delta^2) \leq 1 + \left(\frac{2\pi}{\beta\Delta^2}\right)^\frac{1}{2},$$
    where the second inequality is given by the integral bound $2\sum_{k=1}^\infty \exp(-\frac{\beta}{2}k^2\Delta^2) \leq 2\int_0^\infty\exp(-\frac{\beta\Delta^2t^2}{2})dt = \left(\frac{2\pi}{\beta\Delta^2}\right)^\frac{1}{2}$. This yields, 
    $$S(z) \leq\left(\frac{\beta}{2\pi}\right)^\frac{d}{2} \left(1 + \left(\frac{2\pi}{\beta\Delta^2}\right)^\frac{1}{2}\right)^d.$$
\end{proof}

The following two Lemmas make use of padding windows to give bounds on the ratio $r_s = \frac{S_{\max}}{S_{\min}}$ and the specify a lower bound on $c_p$.
The padding is taken around the truncation set $\mathcal{A}$ as follows. 
\begin{definition}(Padded window set) Let $\Lambda_{lat} = \Delta \mathbb{Z}^d$ and $N_p \geq 1$ then the padded window set consists of Gaussian windows with centers in 
$$Y = \Lambda_{lat} \cap \bigl(\mathcal A+[-N_p\Delta,N_p\Delta]^d\bigr).$$
\end{definition}
\begin{lem}\label{l: log bound S_max/S_min}(Net Coverage Bounds Ratio) \label{l: Gaussian net coverage bounds} Let $S(z) = \sum_k \left(\frac{\beta}{2\pi}\right)^\frac{d}{2}\exp(-\frac{\beta}{2}\left\|z-y_k\right\|^2)$ with $z \in \mathcal{A}$. Let the Gaussian net form a lattice $A$ so that $\mathcal{A} \subset A$ , with spacing $\Delta = \left\|y_i - y_j\right\|$ and assume $\beta\Delta^2 \leq 1$.  Moreover, assume that the lattice $A$ has $N_A$ interior Gaussians and is padded with $N_p \geq \frac{3}{\sqrt{\beta\Delta^2}}$ exterior Gaussians in each dimension. Then 
$$\log(r_s) = \log\left(\frac{S_{\max}}{S_{\min}}\right) \leq 5d\left(\sqrt{\frac{\beta\Delta^2}{\pi}}+\frac{1}{N_p\sqrt{\pi\beta\Delta^2}}e^{-\frac{9}{2}}\right).$$ 
\end{lem}
\begin{proof}
  Let $z = (z^{(1)}, \dots, z^{(d)})$ and $y_k = (y_k^{(1)}, \dots, y_k^{(d)})$ and note that 
    $$\left\| z- y_k\right\|^2 = \sum_{i=1}^d (z^{(i)} - y_k^{(i)})^2.$$
    Then $\tilde{\psi}_k(z)$ is expressible as 
    $$\tilde{\psi}_k(z) = \left(\frac{\beta}{2\pi}\right)^\frac{d}{2}\exp(-\frac{\beta}{2}\sum_{i=1}^d (z^{(i)} - y_k^{(i)})^2) = \left(\frac{\beta}{2\pi}\right)^\frac{d}{2}\prod_{i=1}^d\exp(-\frac{\beta}{2}(z^{(i)} - y_k^{(i)})^2).$$
    Now the Gaussian lattice in $\mathbb{R}^d$ the set of centers $Y = \{y_1, \dots, y_k\}$ is expressible as $Y=\{y = (n_1\Delta, \dots, n_d \Delta)\; | \; n_1, \dots, n_d \in \mathbb{Z}\}$. Therefore 
    $$S(z) = \left(\frac{\beta}{2\pi}\right)^\frac{d}{2}\sum_{n_1 \in \mathbb{Z}} \cdots \sum_{n_d \in \mathbb{Z}}\left(\prod_{i=1}^d\exp(-\frac{\beta}{2}(z^{(i)} - n_i\Delta)^2)\right)\leq\left(\frac{\beta}{2\pi}\right)^\frac{d}{2}\prod_{i=1}^d\left(\sum_{n_i\in \mathbb{Z}} \exp(-\frac{\beta}{2}(z^{(i)} - n_i\Delta)^2)\right).$$
    For the upper bound, on a finite lattice, each one dimensional sum, by symmetry, is maximized when $z^{(i)}$ is the centered lattice point hence
    \begin{align*}\sum_{k = -(N_p+N_A)}^{N_p+N_A} \exp(-\frac{\beta}{2}(z^{(i)} - k\Delta)^2) &\leq 1 + 2\sum_{k=1}^{N_p+N_A} \exp(-\frac{\beta}{2}k^2\Delta^2) \\
    &\leq 1 + 2\sum_{k=1}^{N_p} \exp(-\frac{\beta}{2}k^2\Delta^2) +2\sum_{k=N_p+1}^{N_p+N_A} \exp(-\frac{\beta}{2}k^2\Delta^2).
    \end{align*}
For the lower bound, given $z_0^{(i)} \in [k\Delta, (k +1)\Delta]$ for $k \in [n_{-},n_{-} + (N_A -1)]$ (an interior point) we have that $z_0^{(i)}$ is expressible as $z_0^{(i)} = (k+ \alpha)\Delta$, for $\alpha \in [0,1]$. Then, for  $j \in [-(N_p-1),N_p]$, consider $(z_0^{(i)} - (k+j)\Delta)^2 = ((k+\alpha)\Delta-(k+j)\Delta)^2 = (\alpha - j)^2\Delta^2.$ Then we have the bound $(\alpha - j)^2\Delta^2 \leq (|j| +1)^2\Delta^2$ for $j \leq 0$ and $(\alpha - j)^2\Delta^2 \leq j^2\Delta^2$ for $j \geq 1$. Then the finite summation, for any $k \in [n_{-},n_{-} + (N_A -1)]$, can be lower bounded by

\begin{align*}\sum_{k = n_{-}-N_p}^{n_{-}+N_p+N_A} \exp(-\frac{\beta}{2}(z^{(i)} - k\Delta)^2) &\geq \sum_{j = -(N_p-1)}^{0} \exp(-\frac{\beta}{2}((|j|+1)\Delta)^2) +\sum_{j = 1}^{N_p} \exp(-\frac{\beta}{2}((j\Delta)^2),\\ 
\shortintertext{by re-indexing,}
&= \sum_{j = 1}^{N_p} \exp(-\frac{\beta}{2}j^2\Delta^2)+\sum_{j = 1}^{N_p} \exp(-\frac{\beta}{2}j^2\Delta^2)\\
&\geq 2\sum_{j = 1}^{N_p} \exp(-\frac{\beta}{2}j^2\Delta^2).
\end{align*}
Where the first inequality follows from noting that all the terms in the original summation are non-negative. 

Applying the upper and lower bound on the one dimensional sums and comparing the ratio yields 
\begin{align*}
    r_S &\leq \frac{\left(\frac{\beta}{2\pi}\right)^\frac{d}{2}\prod_{i=1}^d\left(1 + 2\sum_{k=1}^{N_p} \exp(-\frac{\beta}{2}k^2\Delta^2) +2\sum_{k=N_p+1}^{N_p+N_A} \exp(-\frac{\beta}{2}k^2\Delta^2)\right)}{\left(\frac{\beta}{2\pi}\right)^\frac{d}{2}\prod_{i=1}^d2\sum_{k=1}^{N_p}\exp(-\frac{\beta}{2}\Delta^2k^2)}\\
    &=\left(\frac{1 + 2\sum_{k=1}^{N_p} \exp(-\frac{\beta}{2}k^2\Delta^2) +2\sum_{k=N_p+1}^{N_p+N_A} \exp(-\frac{\beta}{2}k^2\Delta^2)}{2\sum_{k=1}^{N_p}\exp(-\frac{\beta}{2}\Delta^2k^2)}\right)^d\\
    &=\left(\frac{1 + 2\sum_{k=1}^{N_p} \exp(-\frac{\beta}{2}k^2\Delta^2) +2\sum_{k=N_p+1}^{N_p+N_A} \exp(-\frac{\beta}{2}k^2\Delta^2)}{2\sum_{k=1}^{N_p}\exp(-\frac{\beta}{2}\Delta^2k^2)}\right)^d.
\end{align*}
Taking the log of this term yields
$$\log(r_s) \leq d \log\left(\frac{1 + 2\sum_{k=1}^{N_p} \exp(-\frac{\beta}{2}k^2\Delta^2) +2\sum_{k=N_p+1}^{N_p+N_A} \exp(-\frac{\beta}{2}k^2\Delta^2)}{2\sum_{k=1}^{N_p}\exp(-\frac{\beta}{2}\Delta^2k^2)} \right).$$
To simplify notation let $A = \sum_{k=1}^{N_p}\exp(-\frac{\beta}{2}\Delta^2k^2) $ and $T = \sum_{k=N_p+1}^{N_p+N_A} \exp(-\frac{\beta}{2}k^2\Delta^2)$, the right term then reduces to $d\log\left(1 + \frac{1}{2A} + \frac{T}{A}\right)$. To obtain a final bound we first lower bound $A$ and then upper bound $T$. 
\begin{align*}
    A &= \sum_{k=1}^{N_p}\exp(-\frac{\beta}{2}\Delta^2k^2)\geq \int_1^{N_p+1}e^{-\frac{\beta}{2}\Delta^2z^2}dz=\sqrt{\frac{2\pi}{\beta\Delta^2}}\int_1^{N_p+1}\sqrt{\frac{\beta\Delta^2}{2\pi}}e^{-\frac{\beta}{2}\Delta^2z^2}dz
\end{align*}
We convert this to a standard normal by $z = \frac{1}{\sqrt{\beta\Delta^2}}u$ which yields 
$$A \geq \sqrt{\frac{2\pi}{\beta\Delta^2}}\int_{\sqrt{\beta\Delta^2}}^{(N_p+1)\sqrt{\beta\Delta^2}}\frac{1}{\sqrt{2\pi}}e^{-\frac{u^2}{2}}du\geq \sqrt{\frac{2\pi}{\beta\Delta^2}}\int_{1}^{3}\frac{1}{\sqrt{2\pi}}e^{-\frac{u^2}{2}}du \geq\frac{1}{7}\sqrt{\frac{2\pi}{\beta\Delta^2}}.$$
Where the last inequality comes from the fact that $\beta\Delta^2 \leq 1$ and assuming that $N_p \geq \frac{3}{\sqrt{\beta\Delta^2}}$. Next T is upper bounded as 
\begin{align*}
    T = \sum_{k=N_p+1}^{N_p+N_A} \exp(-\frac{\beta}{2}k^2\Delta^2) \leq \int_{N_p}^\infty e^{-\frac{\beta}{2}\Delta^2z^2}dz \leq \frac{1}{2N_p\beta\Delta^2}e^{-\frac{N_p^2\beta\Delta^2}{2}}, 
\end{align*}
where the last inequality is the Gaussian tail bound. Applying the inequality $\log(1+z) \leq z$ to the log ratio bound $log(1 + \frac{1}{2A} + \frac{T}{A})$ yields
\begin{align*}
d\log(1 + \frac{1}{2A} + \frac{T}{A}) &\leq 5d\sqrt{\frac{\beta\Delta^2}{2\pi}}+5d\frac{1}{N_p\sqrt{\pi\beta\Delta^2}}e^{-\frac{N_p^2\beta\Delta^2}{2}}.\\
 \shortintertext{Given that $N_p \geq \frac{3}{\sqrt{\beta\Delta^2}}$ yields}
 &\leq 5d\sqrt{\frac{\beta\Delta^2}{2\pi}}+5d\frac{1}{N_p\sqrt{\pi\beta\Delta^2}}e^{-\frac{9}{2}}.
\end{align*}
\end{proof}

Here we establish a lower bound on $c_p$ if the sparse matrix is constructed by only entries $F_{ij}$ from adjacent Gaussians.

\begin{lem}\label{l: c_p lowerbound} Let $\mathcal{A} \subset \mathbb{R}^d$ be convex and let $G := \sup_{z\in\mathcal{A}}\left\|\nabla A(z)\right\|$. Suppose that the Gaussian centers $y_k$ belong to a translate of $\Delta\mathbb{Z}^d$, with \(\|y_i-y_j\|=\Delta\). Form an \(N_p\)-layer padded lattice around \(\mathcal A\), so that
\(
\max_{1\leq k\leq L}
\operatorname{dist}(y_k,\mathcal A)
\leq
\sqrt d\,N_p\Delta
\).
Then 
\[
\min\{F_{ij},F_{ji}\}
\geq
c_p
:=
 \frac{e^{
-\frac{d\beta\Delta^2N_p^2}{2}
-\frac{(\beta\Delta+G)^2}{4\beta}
}}{{2^d\left(1 + \sqrt{\frac{2\pi}{\beta\Delta^2}}\right)^d}}.
\]
\end{lem}
\begin{proof}
    We consider 
    \[
F_{ij}
:=
\frac{
\displaystyle
\int_{\mathcal A}
\frac{
\psi_j(z)
\psi_i(z)
}{
S(z)
}
e^{-A(z)}\,dz
}{
\displaystyle
\int_{\mathcal A}
\psi_i(z)
e^{-A(z)}\,dz
}.
\]
Normalizing constants cancel in the numerator and denominator so Lemma \ref{l: Gaussian net coverage bounds} gives $S(z) \leq \left(1 + \sqrt{\frac{2\pi}{\beta\Delta^2}}\right)^d$.
Therefore we are left to consider
 \[
F_{ij}
:=
\frac{
\displaystyle
\int_{\mathcal A}
{
\psi_j(z)
\psi_i(z)
}
e^{-A(z)}\,dz
}{
\displaystyle
\int_{\mathcal A}
\psi_i(z)
e^{-A(z)}\,dz
}.
\]
We will show that the numerator is bounded by a constant times the denominator. 
First define the projection mapping of a point $z \in \mathbb{R}^d$ onto $\mathcal{A}$ by 

$$Tz = q + \frac{z - q}{2}, \qquad q = {\arg\min}_{x \in \mathcal{A}} \left\|y_j-x\right\|.$$

Since $\mathcal{A}$ we have that for $q,z \in \mathcal{A}$ their linear combination also exists in $\mathcal{A}$ therefore $T(\mathcal{A}) \subseteq \mathcal{A}$. Then 
$$\int_{\mathcal A}
{
\psi_j(z)
\psi_i(z)
}
e^{-A(z)}\,dz \geq \int_{T(\mathcal A)}
{
\psi_j(z)
\psi_i(z)
}
e^{-A(z)}\,dz = \frac{1}{2^d}\int_{\mathcal A}
{
\psi_j(Tz)
\psi_i(Tz)
}
e^{-A(Tz)}\,dz.$$
Consider 
\begin{align*}
{
\psi_j(Tz)
\psi_i(Tz)
}
e^{-A(Tz)}&=e^{
-\frac{\beta}{2}\|Tz-y_i\|^2
-\frac{\beta}{2}\|Tz-y_j\|^2
-A(Tz)
}
\\
&\quad=
e^{
-\frac{\beta}{2}
\left(
\|Tz-y_i\|^2
+\|Tz-y_j\|^2
-\|z-y_i\|^2
\right)
-A(Tz)+A(z)
}
e^{-\frac{\beta}{2}\|z-y_i\|^2-A(z)}.
\end{align*}
Completing the square on the first term yields 
\[
\|Tz-y_i\|^2
+\|Tz-y_j\|^2
-\|z-y_i\|^2
=
\|q-y_j\|^2-\langle y_j - y_i, z-q\rangle-\frac12\|z-q\|^2.
\]
Because ${\arg\min}_{x \in \mathcal{A}} \left\|y_j-x\right\|$, $\|q-y_j\|= \operatorname{dist}(y_j,\mathcal{A})=D$. Since \(\|y_j-y_i\|=\Delta\), we have
\[
\langle y_j - y_i, z-q\rangle\geq-\Delta\|z-q\|.
\]
Furthermore, \(A\) is \(G\)-Lipschitz and
\[
\|Tz-z\|=\frac12\|z-q\|,
\]
so
\[
-A(Tz)+A(z)\geq -G\|T_z - z\|\geq-\frac{G}{2}\|z-q\|.
\]
Which yields
\begin{align*}
&-\frac{\beta}{2}(\|Tz-y_i\|^2
+\|Tz-y_j\|^2
-\|z-y_i\|^2)-A(Tz)+A(z)\\
&\quad\ge -\frac{\beta}{2}\operatorname{dist}(y_j,\mathcal{A})^2 +\frac{\beta}{4}\|z-q\|^2-\frac{\beta\Delta}{2}\|z-q\|-\frac{G}{2}\|z-q\|
\end{align*}
Completing the square in $\|z-q\|$ yields 
\[
\frac{\beta}{4}\|z-q\|^2
-\frac{\beta\Delta+G}{2}\|z-q\|
=
\frac{\beta}{4}
\left(
\|z-q\|-\frac{\beta\Delta+G}{\beta}
\right)^2
-\frac{(\beta\Delta+G)^2}{4\beta},
\]
and hence
\[
\frac{\beta}{4}\|z-q\|^2
-\frac{\beta\Delta+G}{2}\|z-q\|
\geq
-\frac{(\beta\Delta+G)^2}{4\beta}.
\]
Therefore,
\begin{align*}
e^{
-\frac{\beta}{2}(\|Tz-y_i\|^2
+\|Tz-y_j\|^2-\|z-y_i\|^2)
-A(Tz) + A(z)
}
\geq
e^{
-\frac{\beta}{2}\operatorname{dist}(y_j,\mathcal A)^2
-\frac{(\beta\Delta+G)^2}{4\beta}
}.
\end{align*}
Plugging this back into the integral gives
$$\int_{\mathcal A}
{
\psi_j(Tz)
\psi_i(Tz)
}
e^{-A(Tz)}\,dz \ge e^{
-\frac{\beta}{2}\operatorname{dist}(y_j,\mathcal A)^2
-\frac{(\beta\Delta+G)^2}{4\beta}
}\int_{\mathcal A}
\psi_i(z)
e^{-A(z)}\,dz.$$
Which cancels with the denominator to yield 
    \[
F_{ij}
:=
\frac{
\displaystyle
\int_{\mathcal A}
\frac{
\psi_j(z)
\psi_i(z)
}{
S(z)
}
e^{-A(z)}\,dz
}{
\displaystyle
\int_{\mathcal A}
\psi_i(z)
e^{-A(z)}\,dz
}\ge \frac{e^{
-\frac{\beta}{2}\operatorname{dist}(y_j,\mathcal A)^2
-\frac{(\beta\Delta+G)^2}{4\beta}
}}{{2^d\left(1 + \sqrt{\frac{2\pi}{\beta\Delta^2}}\right)^d}}.
\]
Since \(\max_j\operatorname{dist}(y_j,\mathcal A)
\leq \sqrt{d}N_p\Delta\) we get 
$$F_{ij} \ge \frac{e^{
-\frac{d\beta\Delta^2N_p^2}{2}
-\frac{(\beta\Delta+G)^2}{4\beta}
}}{{2^d\left(1 + \sqrt{\frac{2\pi}{\beta\Delta^2}}\right)^d}}.$$
\end{proof}

The following Corollary makes the lower bound on $c_p$ explicit given an appropriate choice of $\beta$, $\Delta$ and $N_p$. For a log-Lipschitz observable these lower bounds are given in terms of the choice parameter $\gamma_0 > 0$ and the dimension of the CV space $d \ge 1$.
In the case of an indicator observable these lower bounds are given in terms of the choice parameter $\epsilon > 0$ and the dimension of the CV space $d \ge 1$.
We note that while the bounds have an exponential dependence on $d$ which does not pose any degradation to our bounds.
This is because in the practical setting we assume $d$ to be small i.e. $d \le 3$. 

Additionally, our bounds have an exponential dependence on $\gamma_0$. 
This dependence is slightly more nuanced as Theorem \ref{t: Explicit Error for Log-lipschitz observable} imposes a restriction on $\gamma_0$ of the form $\Ctheta L_{fiber}^2 \leq \frac{\gamma_0}{2} \leq \frac{60d}{\sqrt{\pi}}$. 
This implies that the observable fluctuations along the fiber must be well behaved relative to the geometry of the conditional distribution along that fiber, otherwise the bounds are exponentially bad. 
However, this is consistent with the behavior of the algorithm. 
We know that a poor choice of collective variable mapping can lead to catastrophic algorithmic performance.

\begin{corollary}\label{c: c_p explicit log-lipschitz}
    (Explicit Overlap Lower Bound). Suppose the hypotheses of Lemma \ref{l: c_p lowerbound} hold. Fix \(\gamma>0\), and choose
    \[
\beta>G^2,
\qquad
\Delta=\frac{\gamma\sqrt{\pi}}{120d\sqrt{\beta}}, \qquad N_p  =\max\{\frac{3}{\sqrt{\beta\Delta^2}},\frac{20d}{\sqrt{e^9\pi}}\sqrt{\frac{1}{\beta\Delta^2}}\}
\]
then
$$c_{p,LL}\ge\frac{
\exp\left[-\max(\frac{9d}{2},\frac{d^3}{100})
-\frac14-\left(\frac{\gamma\sqrt{\pi}}{120d}\right)^2
-\frac{\gamma\sqrt{\pi}}{240d}
\right]
}{
2^d
\left(
1+\frac{120\sqrt{2}\,d}{\gamma}
\right)^d
}.$$
If instead for $\epsilon > 0$ we require  
$$\Delta = \min\left\{\frac{(1-e^{-\epsilon/11)}}{6d\sqrt{\beta}}, \frac{\sqrt{\pi}}{120d\sqrt{\beta}}\right\}$$
then
$$c_{p,ind}\ge\frac{
\exp\left(
-\max(\frac{9d}{2},\frac{d^3}{100})
-\frac12
\right)
}{
2^d
\left(
1+\frac{72\sqrt{2\pi}d}{\epsilon}
\right)^d
}.$$
\end{corollary}
\begin{proof}
    The proof follows by applying the restrictions on $\beta, \Delta$ and $N_p$ to $c_p$ from Lemma \ref{l: c_p lowerbound}.
    First 
    $$c_p
:=
 \frac{\exp\left(
-\frac{d\beta\Delta^2N_p^2}{2}
-\frac{(\beta\Delta+G)^2}{4\beta}\right)
}{{2^d\left(1 + \sqrt{\frac{2\pi}{\beta\Delta^2}}\right)^d}} =  \frac{\exp\left(
-\frac{d\beta\Delta^2N_p^2}{2}
-\frac{\beta\Delta^2}{4}-\frac{G\Delta}{2}-\frac{G^2}{4\beta}\right)
}{{2^d\left(1 + \sqrt{\frac{2\pi}{\beta\Delta^2}}\right)^d}}.$$
Applying $N_p$ to the first term in the exponential and using that $\frac{20}{2\sqrt{e^9\pi}} \leq \frac{1}{10}$ yields, 
$$\exp
(-\frac{d\beta\Delta^2N_p^2}{2}) \geq \exp(-\max(\frac{9d}{2},\frac{d^3}{100})).$$
Applying $\beta> G^2$ to the last term gives
$$\exp(-\frac{G^2}{4\beta}) \ge \exp(-\frac{1}{4}).$$
Using that $\Delta=\frac{\gamma\sqrt{\pi}}{120d\sqrt{\beta}}$ on the second term gives 
$$\exp(-\frac{\beta\Delta^2}{4})\ge \exp(-\left(\frac{\gamma\sqrt{\pi}}{240d}\right)^2).$$
Lastly using  $\Delta=\frac{\gamma\sqrt{\pi}}{120d\sqrt{\beta}}$ and $\beta > G^2$ yields 
$$\exp(-\frac{G\Delta}{2}) \geq \exp(-\frac{\gamma\sqrt{\pi}}{240d}).$$
Putting all these bounds together yields 
$$c_p\ge\frac{
\exp\left[-\max(\frac{9d}{2},\frac{d^3}{100})
-\frac14-\left(\frac{\gamma\sqrt{\pi}}{120d}\right)^2
-\frac{\gamma\sqrt{\pi}}{240d}
\right]
}{
2^d
\left(
1+\frac{120\sqrt{2}\,d}{\gamma}
\right)^d
}.$$
For the bound when $$\Delta = \min\left\{\frac{(1-e^{-\epsilon/11)}}{6d\sqrt{\beta}}, \frac{\sqrt{\pi}}{120d\sqrt{\beta}}\right\}$$ we note that 
$\min\left\{\frac{(1-e^{-\epsilon/11)}}{6d\sqrt{\beta}}, \frac{\sqrt{\pi}}{120d\sqrt{\beta}}\right\} \ge \frac{\epsilon}{72d\sqrt \beta}$. This follows from the fact that 
\(1-e^{-\epsilon/11}
\geq
1-\frac{1}{1+\epsilon/11}
=
\frac{\epsilon}{11+\epsilon}
\geq
\frac{\epsilon}{12}.\) Which implies
\[
\frac{1-e^{-\epsilon/11}}{6d}
\geq
\frac{\epsilon}{72d}.
\]
Using $\Delta = \frac{\epsilon}{72d\sqrt{\beta}}$ yields the bound on the denominator. 
For the numerator we note that 
\[\min\left\{\frac{(1-e^{-\epsilon/11)}}{6d\sqrt{\beta}}, \frac{\sqrt{\pi}}{120d\sqrt{\beta}}\right\} \leq \frac{\sqrt{\pi}}{120d\sqrt{\beta}}\] 
and noting that each of the two terms is less than $\frac{1}{8}$ gives a lower bound of $\exp(-\frac{1}{2})$.
Substitution yields the final bound.
\end{proof}

\section{Analysis of Log-Lipschitz Observable}\label{s: smooth observable}
\subsection{Main Result}
We first collect the parameter choices required for the log-Lipschitz observable analysis into a single algorithm. 
The Gaussian window centers are placed on a padded  lattice in the collective variable space, and the sparse graph $G$ is taken to connect each window to itself and to its nearest-neighbor. 
The probability overlap constant in this construction is given as $c_p^*$, from Corollary \ref{c: c_p explicit log-lipschitz}.
The EMUS estimator is then computed using the corresponding sparse overlap matrix \(F^0\), as defined in Section 1. 
Algorithm \ref{alg:gaussian_net_construction} specifies the inverse temperature \(\beta\), lattice spacing \(\Delta\), and padding \(N_p\). 

To modularize the components controllable by the practitioner, we decompose the analysis into the variance over the fibers and the variance over the CV space.
In doing so we introduce of the conditional expectations $\bar{g}(z) \vcentcolon = \mathbb{E}_{\pi}[g \; |\; \theta = z]$ and $\overline{g^2}(z) \vcentcolon = \mathbb{E}_{\pi}[g^2 \;|\; \theta = z]$ and their corresponding log-Lipschitz constants $L_{\bar{g}}$ and $L_{\overline{g^2}}$. 
These constants cannot be further decomposed without imposing additional restrictions on the mapping $\theta$.
We define 
\begin{equation}\label{e: def L_g}
    L_g = \max\left(L_{\bar{g}},L_{\overline{g^2}}\right)
\end{equation}
as the geometric constant which quantifies the fluctuations of $g$.
\begin{algorithm}[H]
\caption{Gaussian Net \& Sparse Graph Construction (log-Lipschitz Observable)}
\label{alg:gaussian_net_construction}
\begin{algorithmic}
\State \textbf{Inputs:} Target distribution $\pi(x)$ with potential $V$, observable $g$, geometric constants $\ell, G$ (Assumptions \ref{a: collective variable}), $L_g$ (definition \ref{e: def L_g}), targets $\delta_\pi, \delta_g$ (Assumption \ref{a: good truncation}), dimension $d$, and parameter $\gamma_0 \in (0, 120d/\sqrt{\pi}]$.
\State \textbf{Identify:} Compact set $K_U$ such that $$\pi(K_U^c) \leq \delta_\pi \text{ and } \int_{K_U^c} g(x)\pi(dx) \leq \delta_g \pi[g].$$
\State \textbf{Choose:} $\beta$ such that $\beta > \max\left\{\ell + \frac{4}{\gamma_0}L^2_g, 2\ell, G^2\right\}$.
\State \textbf{Choose:} lattice spacing $\Delta =\frac{\gamma_0\sqrt{\pi}}{120 d \sqrt{\beta}}$.
\State \textbf{Choose:} padding $N_p = \max\left\{\frac{3}{\sqrt{\beta\Delta^2}}, \frac{20d}{\sqrt{e^9\pi}}\sqrt{\frac{1}{\beta\Delta^2}}\right\}$.
\State \textbf{Construct:} padded lattice $\Lambda_{pad}$ and graph $\mathcal G := \{(i,j) : \; i = j \text{ or } \|y_i - y_j\| = \Delta\}$
\For{each lattice point $y_k \in \Lambda_{pad}$}
    \State Define ${\psi}_k(x) \gets \left( \frac{\beta}{2\pi} \right)^{d/2} \exp\left( -\frac{\beta}{2} \|\theta(x) - y_k\|^2 \right)$
\EndFor
\State \Return The collection of windows $\{\psi_k\}_{k=1}^L$ and the graph $\mathcal G$.
\end{algorithmic}
\end{algorithm}

\begin{theorem}\label{t: Explicit Error for Log-lipschitz observable}
(Explicit Error for Log-Lipschitz Observables).
Let Assumption \ref{d: target distribution pushforward} hold. Let $\bar{\pi}[g]$ be the EMUS estimator computed on $L$ window using the sparse matrix $F^0$, defined on $G$, with $m = |G|$. Assume the push-forward potential $A(z)$ is $\ell$-smooth. Let $g : \mathbb{R}^D \to \mathbb{R}_{>0}$ be an observable such that the fiber fluctuations satisfy $C_{LS}^{(\theta)}L_{fiber}^2 \le \gamma_0/2$ for some $\gamma_0 \in (0, 120d/\sqrt{\pi}]$. Let $L_g = \max \{L_{\bar{g}}, L_{g^2}\}$.

Suppose the Gaussian net is constructed using parameters $\beta, \Delta,$ and $N_p$ exactly as defined in Algorithm \ref{alg:gaussian_net_construction}. Define the overlap constant lower bound $c_p^*$, for the adjacent Gaussian sparse matrix, according to Corollary \ref{c: c_p explicit log-lipschitz}:
$$ c_p^* = \frac{
\exp\left[-\max(\frac{9d}{2},\frac{d^3}{100})
-\frac14-\left(\frac{\gamma_0\sqrt{\pi}}{120d}\right)^2
-\frac{\gamma_0\sqrt{\pi}}{240d}
\right]
}{
2^d
\left(
1+\frac{120\sqrt{2}\,d}{\gamma_0}
\right)^d
}. $$
Furthermore, define the uniform variance bound $C_{\mathcal{Q}}^*$ as:
$$ C_{\mathcal{Q}}^* := 2(1 + 0.25 e^{-3\gamma_0})^{1/4} \exp(6\gamma_0). $$

For any $\epsilon > 0$ and $\delta > 0$, assume the compact set $K_U$ is chosen such that the truncation parameters satisfy $\max\{\delta_\pi, \delta_g\} \le 1 - e^{-\epsilon/4}$. If the MCMC sampling precision $\tilde{\epsilon}$ and the number of samples per window $N$ satisfy:
$$ \tilde{\epsilon} \le \min\left( \frac{e^{-3\gamma_0}}{4}, \, \frac{\epsilon^2c_p^*}{4\cdot16^2L^2C_{\mathcal Q_*}^2} \right), $$
$$ N \ge \frac{6144mL^2C_{\mathcal Q_*}^2}{\epsilon^2c_p^*\delta} ,$$
so that $
\tilde{\epsilon} = \mathcal{O}\left( \frac{c_p^*\epsilon^2}{L^2}\right)
$ and $
N = \Omega\left( \exp(\gamma_0)\frac{mL^2}{c_p^*\epsilon^2 \delta} \right).
$
Then the EMUS estimator guarantees relative error bounded by $\epsilon$ with probability at least $1 - \delta$:
$$e^{-\frac{3\epsilon}{8}}\mathcal S(w,f)
\leq
\bar{\mathcal S}(\bar w,f)
\leq
e^{\frac{3\epsilon}{8}}\mathcal S(w,f),$$ 
for each $f \in \{g,1\}$, and consequently, 
$$ | \log \bar{\pi}[g] - \log \pi[g] | \le \epsilon. $$
\end{theorem}

\begin{rem}\label{rem: L number}
While Theorem \ref{t: Explicit Error for Log-lipschitz observable} treats the total number of windows $L$ as a fixed parameter, $L$ is inherently dependent on the spacing of the Gaussian windows $\Delta$ as they relate to the relevant set $\mathcal{A}$ on the collective variable space. 
A lattice placed over $\mathcal{A}$ of diameter $\text{diam}(\mathcal{A})$ requires roughly $N_A = \text{diam}(\mathcal{A})/\Delta$ interior points per dimension and the number of the padding windows $N_p \geq \frac{3}{\sqrt{\beta\Delta^2}}$ is made explicit in Lemma \ref{l: Gaussian net coverage bounds}. 
   
Since $\Delta \propto \frac{1}{d\sqrt{\beta}}$ as specified in Algorithm \ref{alg:gaussian_net_construction} implies $\frac{1}{\sqrt{\beta\Delta^2}}\propto d$. This gives us that $N_p \propto \frac{1}{\sqrt{\beta\Delta^2}}$ scales as $\Omega(d)$. 
Combined with the scaling of the padding windows $N_p$ the total number of windows $L = (N_a + N_p)^d$ scales as
$$L = \mathcal{O}\left((\text{diam}(\mathcal{A})d\sqrt{\beta} + d)^d\right).$$

This highlights a relevant trade-off in umbrella sampling: choosing a colder temperature (larger $\beta$) to control the log-Lipschitz constant by better capturing the changes in $\pi$ requires a smaller spacing between windows $\Delta$.
Thereby, the total number of windows $L$ we require increases polynomially in $\beta$ and exponentially in $d$. 
Once again the exponential dependence in $d$ is benign as we assume $d$ is small e.g. $d \leq 3$.
Lastly, to better quantify $\text{diam}(\mathcal{A})$ we would need additional assumptions on the mapping $\theta$. 
\end{rem}

\subsection{Stability of Relative Variance}\label{ss: relative variance gaussian}

In this section, we state a main proposition that quantifies a bound on the relative variance in terms of the biasing function temperature $\beta$ and the spacing $\Delta$. 
The proof is decomposed into two supporting lemmas that make explicit a bound on the relative variance in the case of log-Lipschitz observable functions.
The first Lemma (Lemma \ref{l: Gaussian Variance Bound}) provides a bound under the true distribution $\pi_k$. 
The second Lemma (Lemma \ref{l: Stability of Q, Gaussian}) demonstrates that $\mathcal{Q}$ is stable under the biased distribution $\tilde{\pi}_k$, yielding a bound that's applicable in the case of inexact sampling.

Lemma \ref{l: Gaussian Variance Bound} and Lemma \ref{l: Stability of Q, Gaussian} impose the restriction on the fiber fluctuations $L^2_{fiber} = O\left(\frac{1}{\Ctheta}\right)$. 
Since it is unreasonable to assume control over $L_{fiber}$ or $\Ctheta$, this is inherited as an assumption. 
Furthermore, given that $\Cd$ is controllable through choice of temperature $\beta$, in the context of Remark \ref{r: log-lipschitz decomposition} we require that $\frac{1}{\Cd} = O\left(\max\left(L_{\bar{g}},L_{\overline{g^2}}\right)\right)$ and $\frac{1}{\Cd} = O\left(L_S\right)$. 
The following proposition makes these scalings explicit by appropriate choice of temperature $\beta$ and spacing $\Delta$. 

\begin{rem}\label{r: g = 1}
In order to apply our results in this section to the main theorem (Theorem \ref{t: general error decomposition}), we require a uniform upper bound $C_{\mathcal{Q}}$ such that $\max\{\mathcal{Q}_{\{\cdot\}}(w, g), \mathcal{Q}_{\{\cdot\}}(w, 1)\} \le C_{\mathcal{Q}}$.
Note that when $f = 1$, the function is constant. 
Consequently, the conditional expectations $\bar{g}(z)$ and $\overline{g^2}(z)$ are constant and the associated log-Lipschitz constants vanish so $L_{\bar{g}} = 0$ and $L_{\overline{g^2}} = 0$. 
Therefore the bounds in this section for an arbitrary $g$ also apply as upper bounds for $g = 1$.
\end{rem}

\begin{proposition}\label{p: gamma bound relative variance} Let Assumptions \ref{a: good truncation}, \ref{a: good samples} and \ref{a: good path} hold. Assume the push-forward potential $A(z)$ is $\ell$-smooth, let $L_g = \max\left(L_{\bar{g}},L_{\overline{g^2}}\right)$, and the fiber fluctuations satisfy $\mathcal{C}_{LS}^{(\theta)} L_{fiber}^2 \leq \gamma_0 / 2$. For any $\gamma_0 \in (0, 120d/\sqrt{\pi}]$, choose the  parameters as follows: 

\begin{enumerate}
    \item Choose the temperature $\beta$ such that
    $$\beta > \max\left\{2\ell,\ell + \frac{4}{\gamma_0}L^2_g\right\}.$$
    \item Choose the lattice spacing $\Delta$ as a function of $\beta$ such that 
    $$\Delta = \frac{\gamma_0\sqrt{\pi}}{120 d \sqrt{\beta}}.$$
    \item Choose the padding $N_p$ to satisfy 
    \begin{align*}
    N_p \geq \max\left\{\frac{3}{\sqrt{\beta\Delta^2}},\frac{240d}{\sqrt{e^9\pi}\gamma_0}\sqrt{\frac{1}{\beta\Delta^2}}\right\}.
    \end{align*}
\end{enumerate}

Then, the weighted relative standard deviation under the exact window distributions satisfies:

$$\max\{\mathcal{Q}_{\{\pi\}}(w, g), \mathcal{Q}_{\{\pi\}}(w, 1)\} \leq (\exp(7\gamma_0/2) - 1)^{1/2}$$

Furthermore, if the biased sampling distributions satisfy $\chi^2(\tilde{\pi}_k \| \pi_k) \leq \tilde{\epsilon}$ for all $k$, and the bias is sufficiently small such that $\tilde{\epsilon}^{1/2}\exp(3\gamma_0/2) < 1$, then the relative standard deviation under the biased distributions satisfies:

$$\max\{\mathcal{Q}_{\{\tilde{\pi}\}}(w, g),\mathcal{Q}_{\{\tilde{\pi}\}}(w, 1)\} \leq C_{\mathcal{Q}_*}$$
where the constant $C_{\mathcal{Q}_*}$ is defined as: 

$$C_{\mathcal{Q}_*} \vcentcolon= \frac{\exp\left(6\gamma_0\right) (1+\tilde{\epsilon})^{1/4}}{1 - \tilde{\epsilon}^{1/2}\exp(3\gamma_0/2)}.$$

\end{proposition}

\begin{lem}\label{l: Gaussian Variance Bound}(Variance Bound for Exact Sampling)
Let $g: \mathbb{R}^D \to \mathbb{R}_{>0}$ be an observable and write $g^*(x) = g(x)/S(x)$. Define the conditional expectation $m_1^*(z) := \mathbb{E}_{\pi_k}[g^* \mid \theta = z]$ on the CV space. Suppose:
\begin{enumerate}
\item $m_1^*$ is $L_{m_1^*}$-log-Lipschitz on $\mathcal{A} \subset \mathbb{R}^d$,
\item For every $z \in \mathcal{A}$, the conditional map $x \mapsto g^*(x)$ restricted to the fiber $\Sigma_z$ is $L_{\text{fiber}}$-log-Lipschitz with respect to the fiber-tangent metric,
\item The marginal $\mu_k$ on the CV space satisfies an LSI with constant $C_{LS}^{(d)}$,
\item For every $z \in \mathcal{A}$, the conditional measure $\pi_k(\cdot \mid \theta = z)$ satisfies an LSI on the fiber $\Sigma_z$ with constant $C_{LS}^{(\theta)}$.
\end{enumerate}
Then
\[
\frac{\mathbb{E}_{\pi_k}[(g^*)^2]}{\mathbb{E}_{\pi_k}[g^*]^2} \;\leq\; \exp\!\left( 2 C_{LS}^{(\theta)} L_{\text{fiber}}^2 \right) \cdot \exp\!\left( 2 C_{LS}^{(d)} L_{m_1^*}^2 \right),
\]
and consequently
\[
\mathrm{Var}_{\pi_k}(g^*) \;\leq\; \left[ \exp\!\left( 2 C_{LS}^{(\theta)} L_{\text{fiber}}^2 + 2 C_{LS}^{(d)} L_{m_1^*}^2 \right) - 1 \right] \mathbb{E}_{\pi_k}[g^*]^2.
\]
Defining
\[
\gamma \;:=\; \exp\!\left( 2 C_{LS}^{(\theta)} L_{\text{fiber}}^2 +2 C_{LS}^{(d)} L_{m_1^*}^2 \right),
\]
the weighted relative standard deviation under $\{\pi_k\}$ satisfies $Q_{\{\pi\}}(w, g) \leq (\gamma-1)^{1/2}$.
\end{lem}

\begin{proof}
We bound the ratio $\mathbb{E}_{\pi_k}[(g^*)^2] / \mathbb{E}_{\pi_k}[g^*]^2$. By the tower property,
\[
\mathbb{E}_{\pi_k}[(g^*)^2] \;=\; \mathbb{E}_{\mu_k}\!\left[ \mathbb{E}_{\pi_k}[(g^*)^2 \mid \theta] \right], \qquad
\mathbb{E}_{\pi_k}[g^*] \;=\; \mathbb{E}_{\mu_k}\!\left[ \mathbb{E}_{\pi_k}[g^* \mid \theta] \right] = \mathbb{E}_{\mu_k}[m_1^*(\theta)].
\]
Applying the above, we get that 
\begin{equation}\label{eq:decomp-56}
\frac{\mathbb{E}_{\pi_k}[(g^*)^2]}{\mathbb{E}_{\pi_k}[g^*]^2} \;=\; \underbrace{\frac{\mathbb{E}_{\mu_k}\!\left[ \mathbb{E}_{\pi_k}[(g^*)^2 \mid \theta] \right]}{\mathbb{E}_{\mu_k}\!\left[ \mathbb{E}_{\pi_k}[g^* \mid \theta]^2 \right]}}_{=\vcentcolon T_1} \cdot \underbrace{\frac{\mathbb{E}_{\mu_k}\!\left[ \mathbb{E}_{\pi_k}[g^* \mid \theta]^2 \right]}{\mathbb{E}_{\mu_k}\!\left[ \mathbb{E}_{\pi_k}[g^* \mid \theta] \right]^2}}_{=\vcentcolon T_2}.
\end{equation}
First we bound $T_1$ so that
\begin{equation}\label{eq:T1-bound}
T_1  \;=\; \frac{\mathbb{E}_{\mu_k}\!\left[ \mathbb{E}_{\pi_k}[(g^*)^2 \mid \theta] \right]}{\mathbb{E}_{\mu_k}\!\left[ \mathbb{E}_{\pi_k}[g^* \mid \theta]^2 \right]} \;\leq\; \sup_{z \in \mathcal{A}} \frac{\mathbb{E}_{\pi_k}[(g^*)^2 \mid \theta = z]}{\mathbb{E}_{\pi_k}[g^* \mid \theta = z]^2} \;\leq\; e^{2 C_{LS}^{(\theta)} L_{\text{fiber}}^2},
\end{equation}
where the last inequality follows from Corollary \ref{c: Herbst}. Next we bound $T_2$, again by Corollary \ref{c: Herbst},
\begin{equation}\label{eq:T2-bound}
T_2 \;=\;\frac{\mathbb{E}_{\mu_k}\!\left[ \mathbb{E}_{\pi_k}[g^* \mid \theta]^2 \right]}{\mathbb{E}_{\mu_k}\!\left[ \mathbb{E}_{\pi_k}[g^* \mid \theta] \right]^2}\;=\; \frac{\mathbb{E}_{\mu_k}[(m_1^*)^2]}{\mathbb{E}_{\mu_k}[m_1^*]^2} \;\leq\; e^{2 C_{LS}^{(d)} L_{m_1^*}^2}.
\end{equation}
Noting $\mathrm{Var}_{\pi_k}(g^*) = \mathbb{E}_{\pi_k}[(g^*)^2] - \mathbb{E}_{\pi_k}[g^*]^2$ gives
\[
\frac{\mathrm{Var}_{\pi_k}(g^*)}{\mathbb{E}_{\pi_k}[g^*]^2} \;\leq\; \exp\!\left( 2 C_{LS}^{(\theta)} L_{\text{fiber}}^2 + 2 C_{LS}^{(d)} L_{m_1^*}^2 \right) - 1.
\]
Then
\[
Q_{\{\pi\}}(w, g) \;=\; \frac{\sum_k w_k \sqrt{\mathrm{Var}_{\pi_k}(g^*)}}{\sum_k w_k \mathbb{E}_{\pi_k}[g^*]} \;\leq\; (\gamma-1)^{1/2} \cdot \frac{\sum_k w_k \mathbb{E}_{\pi_k}[g^*]}{\sum_k w_k \mathbb{E}_{\pi_k}[g^*]} \;=\; (\gamma-1)^{1/2}.
\]

\end{proof}

\begin{lem}\label{l: Stability of Q, Gaussian}(Stability of $\mathcal{Q}$)  Let Assumption \ref{a: good samples} hold.  Let $\{\tilde{\pi}_k\}$ be the collection of biased distributions with $\chi^2(\tilde{\pi}_k \,\|\, \pi_k) \leq \tilde{\epsilon}$ for all $k$. Suppose:
\begin{enumerate}
\item $m_2^*(z) := \mathbb{E}_{\pi_k}[(g^*)^2 \mid \theta = z]$ is $L_{m_2^*}$-log-Lipschitz on $\mathcal{A}$,
\item For every $z \in \mathcal{A}$, the function $x \mapsto g^*(x)$ restricted to the fiber $\Sigma_z$ is $L_{\text{fiber}}$-log-Lipschitz,
\item $\mu_k$ satisfies an LSI on $\mathbb{R}^d$ with constant $C_{LS}^{(d)}$, and $\pi_k(\cdot \mid \theta = z)$ satisfies an LSI on $\Sigma_z$ with constant $C_{LS}^{(\theta)}$ for every $z$.
\end{enumerate}
If the sampling bias is sufficiently small that $\tilde{\epsilon}^{1/2} \gamma^{1/2} < 1$ (with $\gamma$ from Lemma \ref{l: Gaussian Variance Bound}), then
\[
Q_{\{\tilde{\pi}\}}(w, g) \;\leq\; \left[ \frac{ \exp\!\left( 3 C_{LS}^{(\theta)} L_{\text{fiber}}^2 + \tfrac{1}{2} C_{LS}^{(d)} L_{m_2^*}^2 \right) (1+\tilde{\epsilon})^{1/4}}{(1 - \tilde{\epsilon}^{1/2} \gamma^{1/2})}\right]\gamma^{1/2}.
\]

\end{lem}
\begin{proof} We begin by bounding the following ratio using Cauchy--Schwarz  
\begin{equation}\label{eq:var-ratio-bound}
\frac{\mathbb{E}_{\tilde{\pi}_k}[(g^*)^2]}{\mathbb{E}_{{\pi}_k}[(g^*)^2]}\;\leq\; \frac{\mathbb{E}_{\pi_k}\!\left[ (g^*)^4 \right]^{1/2}}{\mathbb{E}_{\pi_k}\!\left[ (g^*)^2 \right]} \cdot (1 + \tilde{\epsilon})^{1/2}.
\end{equation}
We then $L^4/L^2$ component as
\[
(*)^2 \;:=\; \frac{\mathbb{E}_{\pi_k}[(g^*)^4]}{\mathbb{E}_{\pi_k}[(g^*)^2]^2}.
\]
Again we decompose using the tower property to get 
\begin{equation}\label{eq:decomp-57}
(*)^2 \;=\; \underbrace{\frac{\mathbb{E}_{\mu_k}\!\left[ \mathbb{E}_{\pi_k}[(g^*)^4 \mid \theta] \right]}{\mathbb{E}_{\mu_k}\!\left[ \mathbb{E}_{\pi_k}[(g^*)^2 \mid \theta]^2 \right]}}_{=: T_1'} \cdot \underbrace{\frac{\mathbb{E}_{\mu_k}\!\left[ m_2^*(\theta)^2 \right]}{\mathbb{E}_{\mu_k}\!\left[ m_2^*(\theta) \right]^2}}_{=: T_2'},
\end{equation}
with $\mathbb{E}_{\mu_k}\!\left[ m_2^*(\theta)^2 \right] = \mathbb{E}_{\mu_k}\!\left[ \mathbb{E}_{\pi_k}[(g^*)^2 \mid \theta]^2 \right]$. We proceed in a similar fashion to Lemma \ref{l: Gaussian Variance Bound} so that 

\[
T_1' \;\leq\; \sup_{z \in \mathcal{A}} \frac{\mathbb{E}_{\pi_k}[(g^*)^4 \mid \theta = z]}{\mathbb{E}_{\pi_k}[(g^*)^2 \mid \theta = z]^2} \; \leq\; e^{12C_{LS}^{(\theta)} L_{\text{fiber}}^2},
\]
where the last line follows from Corollary \ref{c: Herbst}. Similarly, we use Corollary \ref{c: Herbst} to bound $T_2'$, 
\[T_2' \;=\; \frac{\mathbb{E}_{\mu_k}\!\left[ m_2^*(\theta)^2 \right]}{\mathbb{E}_{\mu_k}\!\left[ m_2^*(\theta) \right]^2} \; \leq \; e^{2 C_{LS}^{(d)} L_{m_2^*}^2}.\]
Therefore combining the two bounds yields 
\[
(*) \;=\; \sqrt{T_1' \cdot T_2'} \;\leq\; \exp\!\left( 6C_{LS}^{(\theta)} L_{2\text{fiber}}^2 + C_{LS}^{(d)} L_{m_2^*}^2 \right).
\]
Lastly, we apply Cauchy--Schwarz to get the following
\[
\left| \mathbb{E}_{\tilde{\pi}_k}[g^*] - \mathbb{E}_{\pi_k}[g^*] \right| \;\leq\; \mathbb{E}_{{\pi}_k}[(g^*)^2]^{1/2} \cdot \chi^2(\tilde{\pi}_k \,\|\, \pi_k)^{1/2} \;\leq\; \mathbb{E}_{{\pi}_k}[(g^*)^2]^{1/2} \cdot \tilde{\epsilon}^{1/2}.
\]
Using the second moment bound from Lemma \ref{l: Gaussian Variance Bound} yields
\begin{equation}\label{eq:exp-tilde-pi}
\mathbb{E}_{\tilde{\pi}_k}[g^*] \;\geq\; \mathbb{E}_{\pi_k}[g^*]\!\left( 1 - \tilde{\epsilon}^{1/2} \gamma^{1/2} \right).
\end{equation}
Combining the above bounds we get 
\begin{align*}
Q_{\{\tilde{\pi}\}}(w, g)
&= \frac{\sum_k w_k \sqrt{\mathrm{Var}_{\tilde{\pi}_k}(g^*)}}{\sum_k w_k \mathbb{E}_{\tilde{\pi}_k}[g^*]}\leq \frac{\sum_k w_k \sqrt{\frac{\mathbb{E}_{\tilde{\pi}_k}[(g^*)^2]}{\mathbb{E}_{{\pi}_k}[(g^*)^2]}\mathbb{E}_{{\pi}_k}[(g^*)^2]}}{\sum_k w_k \mathbb{E}_{\tilde{\pi}_k}[g^*]}  \\
&\leq \frac{ \exp\!\left( 3C_{LS}^{(\theta)} L_{\text{fiber}}^2 + \tfrac{1}{2} C_{LS}^{(d)} L_{m_2^*}^2 \right) (1+\tilde{\epsilon})^{1/4} \sum_k w_k \sqrt{\mathbb{E}_{{\pi}_k}[(g^*)^2]}}{ (1 - \tilde{\epsilon}^{1/2} \gamma^{1/2}) \sum_k w_k \mathbb{E}_{\pi_k}[g^*] }\\
&= \frac{ \exp\!\left( 3C_{LS}^{(\theta)} L_{\text{fiber}}^2 + \tfrac{1}{2} C_{LS}^{(d)} L_{m_2^*}^2 \right) (1+\tilde{\epsilon})^{1/4}}{(1 - \tilde{\epsilon}^{1/2} \gamma^{1/2})}\gamma^{1/2}.
\end{align*}
Where we note that the second inequality follows from the proof of Lemma \ref{l: Gaussian Variance Bound} and the last line follows from the Lemma \ref{l: Gaussian Variance Bound} bound. 
\end{proof}
\begin{rem}\label{r: log-lipschitz decomposition} Since the umbrella sum $S(x) = \sum_k (\frac{\beta}{2\pi})^\frac{d}{2}\exp(-\beta\|\theta(x) - y_k\|^2/2)$ is constant along the fibers $\theta(x) = z$. It follows that 
$$m^*_1(z) = \mathbb{E}_{\pi_k}[g^* \mid \theta = z] = \frac{1}{S(z)}\mathbb{E}_{\pi_k}[g \mid \theta = z] = \frac{\bar{g}(z)}{S(z)}.$$
Given Lemma \ref{l: Gaussian Variance Bound}, the log-Lipchitz constant for $m^*_1$ is then bounded so that $L_{m^*_1}^2 \leq 2L_{\bar{g}}^2 + 2L_{S}^2$. The same clearly holds for $m^*_2$ with $L^2_{m^*_2} \leq 2L_{\overline{g^2}}^2 + 2L_{S^2}^2 \leq 2L_{\overline{g^2}}^2 + 8L_{S}^2$, where the last inequality follows from $L_{S^2} \leq 2 L_S$. 
Note that in the case where $g = 1$, the log-Lipschitz constants $L_{\bar{g}} = 0$ and $L_{\overline{g^2}} = 0$. 

The log-Lipschitz constants $L_{\bar{g}}$ and $L_{\overline{g^2}}$ depend on the mapping $\theta$ and the function $g$, both which the practitioner has little control over. However, the log-Lipschitz constant $L_{S}$ can be controlled through the window temperature $\beta$ and spacing $\Delta$. 
\end{rem}

\begin{proof}(Proof of Proposition \ref{p: gamma bound relative variance})
We consider the exponential terms from Lemma \ref{l: Gaussian Variance Bound} and Lemma \ref{l: Stability of Q, Gaussian}, given by $\exp\!\left( 2 C_{LS}^{(\theta)} L_{\text{fiber}}^2 +2 C_{LS}^{(d)} L_{m_1^*}^2 \right)$ and $\exp\!\left( 3 C_{LS}^{(\theta)} L_{\text{fiber}}^2 + \tfrac{1}{2} C_{LS}^{(d)} L_{m_2^*}^2 \right)$, respectively. 
Using the bounds on $L_{m_1^*}^2$ and $L_{m_2^*}^2$ from Remark \ref{r: log-lipschitz decomposition} we have 
\begin{align*}
\exp\!\left( 2 C_{LS}^{(\theta)} L_{\text{fiber}}^2 +2 C_{LS}^{(d)} L_{m_1^*}^2 \right) &\leq \exp\!\left( 2 C_{LS}^{(\theta)} L_{\text{fiber}}^2 +4 C_{LS}^{(d)} L_{\bar{g}}^2 + 4 C_{LS}^{(d)} L_{S}^2 \right)\\ 
\exp\!\left( 3C_{LS}^{(\theta)} L_{\text{fiber}}^2 + \tfrac{1}{2} C_{LS}^{(d)} L_{m_2^*}^2 \right) &\leq \exp\!\left( 3 C_{LS}^{(\theta)} L_{\text{fiber}}^2 + C_{LS}^{(d)}L^2_{\overline{g^2}} + 4C_{LS}^{(d)}L^2_S\right)
\end{align*}
Each term can be upper bounded by:
\begin{equation}
\exp\!\left( 3C_{LS}^{(\theta)} L_{\text{fiber}}^2 + 4C_{LS}^{(d)}L^2_{g} + 4C_{LS}^{(d)}L^2_S\right).
\end{equation}
By Remark \ref{r: Bakry-Emery} yields the upper bound $\Cd \leq \frac{1}{\beta-\ell}$. 
    Since $\beta$ is chosen as $\beta > \ell + \frac{4}{\gamma_0}L^2_g$ this yields an upper bound on the first term of
    $$\Cd L^2_g \leq \frac{L^2_g}{\ell + \frac{4}{\gamma_0}L^2_g - \ell} = \frac{L^2_g}{\frac{4}{\gamma_0}L^2_g}= \frac{\gamma_0}{4}.$$
    To upper bound the $\Cd L^2_S$ term note that Lemma \ref{l: Lipshitz generic bound} applied to $f(x) = \log S(x)$ yields $L^2_s \leq 2\beta\log(r_s),$ where $r_s = \frac{S_{\max}}{S_{\min}}$. 
    Applying this to Lemma \ref{l: log bound S_max/S_min} yields 
    $$\Cd L^2_s \leq \frac{10d\beta}{\beta-\ell}\left(\sqrt{\frac{\beta\Delta^2}{\pi}}+\frac{1}{N_p\sqrt{\pi\beta\Delta^2}}e^{-\frac{9}{2}}\right).$$
    We note that the bound from Lemma \ref{l: log bound S_max/S_min} can be applied to Lemma \ref{l: Lipshitz generic bound}, since $S_{\max}(z)$ for $z \in \mathbb{R}^d$ is obtained for $z \in \mathcal{A}$.    
    Let $\beta \geq 2\ell$ and apply $\Delta =  \frac{\gamma_0\sqrt{\pi}}{120 d \sqrt{\beta}}$ to the first, this yields
    $$\frac{10d\beta}{\beta-\ell}\sqrt{\frac{\beta\Delta^2}{\pi}}\leq \gamma_0\frac{120d\beta}{\beta-\ell}\sqrt{\frac{\beta\frac{\pi}{120^2d^2\beta}}{\pi}} = \gamma_0\frac{10d\beta}{120d\beta}\frac{1}{\beta-\ell}=\gamma_0\frac{\beta}{12(\beta-\ell)}\leq \frac{\gamma_0}{12}.$$
    Where the last inequality comes from applying $\beta \geq 2\ell$ as $\beta - \ell \geq \frac{\beta}{2}$. Lastly applying $ N_p \geq \frac{240d}{\sqrt{e^9\pi}\gamma_0}\sqrt{\frac{1}{\beta\Delta^2}}$ to the second term yields
    $$\frac{10d\beta}{\beta-\ell}\frac{1}{N_p\sqrt{\pi\beta\Delta^2}}e^{-\frac{9}{2}}\leq\frac{10d\beta}{\beta-\ell}\frac{1}{\frac{240d}{\sqrt{e^9\pi}\gamma_0}\sqrt{\frac{1}{\beta\Delta^2}}\sqrt{\pi\beta\Delta^2}}e^{-\frac{9}{2}} \leq \frac{10\gamma_0}{240}\frac{\beta}{(\beta-\ell)} \leq \frac{\gamma_0}{12}.$$
    Combining the bounds gives that 
    $$\exp\!\left( 3C_{LS}^{(\theta)} L_{\text{fiber}}^2 + 4C_{LS}^{(d)}L^2_{g} + 4C_{LS}^{(d)}L^2_S\right) \leq \exp(3\frac{\gamma_0}{2} + 4\frac{\gamma_0}{4} + 4\frac{2\gamma_0}{12}) < \exp(\frac{7\gamma_0}{2}).$$
    
Substituting this bound back into Lemma \ref{l: Gaussian Variance Bound} yields $\gamma \leq\exp(\frac{7\gamma_0}{2})$ and
$$Q_{\{\pi\}}(w, g) \leq (\exp\left(\frac{7\gamma_0}{2}\right)-1)^{1/2}.$$
Similarly, substitution into Lemma \ref{l: Stability of Q, Gaussian} yields
$$Q_{\{\tilde{\pi}\}}(w, g) \;\leq\; \left[ \frac{ \exp\left(\frac{7\gamma_0}{2}\right)(1+\tilde{\epsilon})^{1/4}}{(1 - \tilde{\epsilon}^{1/2} \exp\left(\frac{3\gamma_0}{2}\right))}\right]\gamma^{1/2}.$$
Therefore we have that 
\begin{align*}
    \max\left\{Q_{\{\tilde{\pi}\}}(w, g), Q_{\pi}(w, g)\right\} &\leq \left[ \frac{\exp\left(\frac{7\gamma_0}{2}\right)(1+\tilde{\epsilon})^{1/4}}{(1 - \tilde{\epsilon}^{1/2} \exp\left(\frac{3\gamma_0}{2}\right))}\right]\exp\left(\frac{7\gamma_0}{4}\right)\\
&<\frac{\exp\left(6\gamma_0\right)(1+\tilde{\epsilon})^{1/4}}{(1 - \tilde{\epsilon}^{1/2} \exp\left(\frac{3\gamma_0}{2}\right))}. 
\end{align*}
Letting $C_{\mathcal Q} = \frac{\exp\left(6\gamma_0\right)(1+\tilde{\epsilon})^{1/4}}{(1 - \tilde{\epsilon}^{1/2} \exp\left(\frac{3\gamma_0}{2}\right))}$ yields the result. 
By Remark \ref{r: g = 1} these bounds are also applicable to $g = 1$. 
    
\end{proof}

\subsection{Proof of Section Main Result}

\begin{proof} First we use Proposition \ref{p: gamma bound relative variance} with $\tilde{\epsilon} < \exp(-3\gamma_0)$ to enforce for $f \in \{g,1\}$
$$\max\{\mathcal{Q}_{\{{\pi}\}}(w, f),\mathcal{Q}_{\{\tilde{\pi}\}}(w, f)\} \leq C_{\mathcal{Q}_*},$$
with $C_{\mathcal{Q}_*} \vcentcolon= \frac{\exp\left(6\gamma_0\right) (1+\tilde{\epsilon})^{1/4}}{1 - \tilde{\epsilon}^{1/2}\exp(3\gamma_0/2)}.$

The rest follows from Corollary \ref{c: main_theorem}, which says that if $N\ge \max\left\{\frac{4C_{\mathcal Q_*}^2}{\delta_f},\frac{8mL^2}{\epsilon_w^2c_p\delta_w}\right\}$ and $\tilde\epsilon\le\min\left\{\frac{1}{4C_{\mathcal Q_*}^2},\frac{\epsilon_w^2c_p}{4L^2}\right\}$ then,  with probability $1 - \delta_w- 2\delta_f $,
$$\left|
\log\bar\pi[g]-\log\pi[g]
\right|
\leq
4\epsilon_w + \frac{4C_{\mathcal Q_*}}{\sqrt{N\delta_f}} + 4\sqrt{\widetilde\epsilon}\,C_{\mathcal Q_*}
+
\max\left\{
\log\frac{1}{1-\delta_\pi},
\log\frac{1}{1-\delta_g}
\right\}.$$
We will proceed by choosing $\delta_w = \delta_f = \frac{\delta}{3}$, and showing that each component in the previous summation can be bounded by $\frac{\epsilon}{4}$.
This choice of $\delta$ means that we require $N\ge \max\left\{\frac{12C_{\mathcal Q_*}^2}{\delta},\frac{24mL^2}{\epsilon_w^2c_p^*\delta}\right\}$. 
To bound the first term we have free choice over $\epsilon_w$, therefore choosing $\epsilon_w = \frac{\epsilon}{16}$ gives the desired $\frac{\epsilon}{4}$ bound. 
This means we require $N\ge \max\left\{\frac{12C_{\mathcal Q_*}^2}{\delta},\frac{24\cdot16^2mL^2}{\epsilon^2c_p^*\delta}\right\}$ and $\tilde\epsilon\le\min\left\{\frac{1}{4C_{\mathcal Q_*}^2},\frac{\epsilon^2c_p^*}{4\cdot 16^2L^2}\right\}$.

For the second term,  $\frac{4C_{\mathcal Q_*}}{\sqrt{N\cdot\frac{\delta}{3}}} \leq \frac{\epsilon}{4}$ is implied by $N \geq \frac{3\cdot 16^2 C_{\mathcal Q_*}^2}{\delta \epsilon^2}$. Therefore we require  
$$N\ge \max\left\{\frac{12C_{\mathcal Q_*}^2}{\delta},\frac{4\cdot24\cdot16^2mL^2}{\epsilon^2c_p^*\delta},\frac{3\cdot 16^2 C_{\mathcal Q_*}^2}{\delta \epsilon^2}\right\}.$$ 

For the third term, $4\sqrt{\widetilde\epsilon}\,C_{\mathcal Q_*} \leq \frac{\epsilon}{4}$ is implied by $\tilde\epsilon \leq \frac{\epsilon^2}{16^2C_{\mathcal Q_*}^2}$. Therefore we require
$$\tilde\epsilon\le\min\left\{\frac{1}{4C_{\mathcal Q_*}^2},\frac{\epsilon^2c_p^*}{4\cdot4\cdot 16^2L^2}, \frac{\epsilon^2}{16^2C_{\mathcal Q_*}^2}\right\}.$$

This means that $2\eta \leq \frac{3\epsilon}{4}$. So consequently we have that $\eta \leq \frac{3\epsilon}{8}$ which implies $$e^{-\frac{3\epsilon}{8}}\mathcal S(w,f)
\leq
\bar{\mathcal S}(\bar w,f)
\leq
e^{\frac{3\epsilon}{8}}\mathcal S(w,f).$$
The assumption $\max\{\delta_\pi, \delta_g\} \le 1 - e^{-\epsilon/4}$ directly implies that the last term $\max\left\{
\log\frac{1}{1-\delta_\pi},
\log\frac{1}{1-\delta_g}
\right\} \leq \frac{\epsilon}{4}$. 

Lastly, we note that $C^2_{\mathcal{Q}^*} \ge 4$ and to to get that, $\frac{24\cdot16^2mL^2C_{\mathcal Q_*}^2}{\epsilon^2c_p^*\delta} \geq \max\left\{\frac{12C_{\mathcal Q_*}^2}{\delta},\frac{4\cdot24\cdot16^2mL^2}{\epsilon^2c_p^*\delta},\frac{3\cdot 16^2 C_{\mathcal Q_*}^2}{\delta \epsilon^2}\right\}$ as a requirement for the number of samples $N$ and $\frac{\epsilon^2c_p^*}{4\cdot16^2L^2C_{\mathcal Q_*}^2} \le \min\left\{\frac{1}{4C_{\mathcal Q_*}^2},\frac{\epsilon^2c_p^*}{4\cdot4\cdot 16^2L^2}, \frac{\epsilon^2}{16^2C_{\mathcal Q_*}^2}\right\}$ as a requirement for $\tilde{\epsilon}$. 
Recalling that we also require $\tilde{\epsilon} < \exp(-3\gamma_0)$ yields the final bound. 
\end{proof}


   
\section{Analysis of Indicator Observable}\label{s: indicator}

    
In Section \ref{s: smooth observable}, we relied on the log-Lipschitzness of the observable to directly bound the weighted relative standard deviation $\mathcal{Q}_{\{\mu\}}(w,f)$. 
In the case of an indicator observable obtaining a direct bound on the weighted relative standard deviation poses numerous problems. 
The most relevant being that the indicator creates an unavoidable variance explosion for windows located far from the target observable set which prevents us from applying Chebyshev's inequality globally. 

Instead, to find bounds for the indicator observable we show that the condition in Theorem \ref{t: general error decomposition} eq. (\ref{e: condition_main_theorem}) holds for $\bar{S}(\bar w, g)$.
To formalize the observable in this section we define
$$g(x) := 1_E\circ \theta(x),$$
for a convex set $E$. 
This defines $f: \mathbb{R}^D \rightarrow \mathbb R$ but allows the set to be defined on the collective variable space $E \subseteq \mathcal{A} \subseteq \mathbb{R}^d$.
To show that eq. (\ref{e: condition_main_theorem}) holds 
We define well-behaved estimators that serve as upper and lower bounds, 
$$ \bar{\mathcal{S}}(\bar w, (g_k)_{k\in W})=:\bar{\mathcal S}^{\textup{lower}}(\bar w, g)
        \le 
        \bar{\mathcal S}(\bar w, g)\le 
        \bar{\mathcal S}^{\textup{upper}}(\bar w, g):=\bar{\mathcal{S}}(\bar{w}, e^{-\lambda d(E,\theta(x))}).$$
Where for the lower bound we define $W= \left\{ k : \E_{\pi_k}[g]\ge p_E\right\}$, so that the
empirical average is taken over a good set of windows. 
Upper and lower bounding $ \bar{\mathcal S}^{\textup{upper}}(\bar w, g)$ and $\bar{\mathcal S}^{\textup{lower}}(\bar w, g)$ by $S(w,g)$, respectively.
Allows us to show that eq. (\ref{e: condition_main_theorem}) is satisfied. 

To apply Theorem $\ref{t: general error decomposition}$ we require eq. (\ref{e: condition_main_theorem}) to be met for $f \in \{g, 1\}$.
However, we note that $f = 1$ is log-Lipschitz with constant $0$. 
Therefore, the upper and lower bounds follow from the first part of Theorem \ref{t: Explicit Error for Log-lipschitz observable}. Since in this case $L_g = L_{fiber} = 0$, Theorem \ref{t: Explicit Error for Log-lipschitz observable} applies with a simplified choice of parameters $\beta$ and $\Delta$. 
Moreover, this gives full range to the parameter choice $\gamma_0$, which can be taken so that $\gamma_0 = 1$. 
Thereby, the work in this section is dedicated to finding upper and lower bounds for when $f  = g$ with $g(x) = 1_E \circ \theta$.  

Note that we can bound the error of the numerator and denominator $\bar{\mathcal S}(\bar w,g)$ and $\bar{\mathcal S}(\bar w,1)$ separately, in terms of $\mathcal S( w, g)$ and $\mathcal S( w, 1)$, respectively. This is important because we use a different argument for the upper and lower bounds (as an upper bound for the ratio requires an upper bound for the numerator and a lower bound for the denominator).

The structure of our argument is the following. We define more well-behaved estimators that serve as upper and lower bounds, 
    \begin{align*}
    \bar{\mathcal S}^{\textup{lower}}(\bar w, g)
        \le 
        \bar{\mathcal S}(\bar w, g)\le 
        \bar{\mathcal S}^{\textup{upper}}(\bar w, g)
    \end{align*}
    where this is an inequality of random variables that always holds. Using the well-behavedness of these estimators, we can use the bounds for $\mathcal{E}_{\text{weight}},\mathcal{E}_{\text{sample}},\mathcal{E}_{\text{bias}}$ to show they are close to the exact upper and lowerbound quantities. Then we show that the upper and lower bounds don't miss too much: 
    \begin{align*}
        \bar{\mathcal S}^{\textup{upper}}(\bar w, g)
        \le_{\text{whp,\,by $\mathcal E$ bounds}} (1+\ep) {\mathcal S}^{\textup{upper}}(w, g)
        \le (1+\ep)^2{\mathcal S}(w, g).
    \end{align*}
\subsection{Main Result}
We first collect the parameter choices required for the indicator observable analysis into a single algorithm. 
The Gaussian window centers are placed on a padded  lattice in the collective variable space, and the sparse graph \(\mathcal G\) is taken to connect each window to itself and to its nearest-neighbor. 
The probability overlap constant in this construction is given as $c_p^*$, from Corollary \ref{c: c_p explicit log-lipschitz}.
The EMUS estimator is then computed using the corresponding sparse overlap matrix \(F^0\), as defined in Section 1. 
Algorithm \ref{alg:indicator_net_construction} specifies the inverse temperature \(\beta\), lattice spacing \(\Delta\), and padding \(N_p\) so that the upper- and lower-bound arguments of Sections \ref{ss:upper bound for number} and \ref{ss:lower bound for numerator} apply simultaneously to the same Gaussian net. 
The additional requirements on the number of samples per window \(N\) and the sampling accuracy \(\widetilde\epsilon\) are stated explicitly in Theorem \ref{t:indicator final bound}.
\begin{algorithm}[H]
\caption{Gaussian Net \& Sparse Graph Construction (Indicator Observable)}
\label{alg:indicator_net_construction}
\begin{algorithmic}
\State \textbf{Inputs:} Target distribution $\pi(x)$ with potential $V$, observable $g = 1_E\circ\theta$, so that $B(c,\rho) \leq E \leq B(c, R)$, geometric constants $\ell$, $G_+ = \max\{1, G\}$, $L_g =  2\left(\frac d\rho+G^+\kappa\right)
\left(1+\frac{22}{5\varepsilon_0}\right), \kappa = \fc R\rho$ and $d$, targets $\delta_\pi, \delta_g$ and hyperparameter $\epsilon>0$.
\State \textbf{Identify:} Compact set $K_U$ such that $$\pi(K_U^c) \leq \delta_\pi \text{ and } \int_{K_U^c} g(x)\pi(dx) \leq \delta_g \pi[g].$$
\State \textbf{Choose:} $\beta$ such that $\beta > \max\left\{2l, \; {G_+^2}, \; \ell +{4L_g^2}, \; \frac{17908G_+^4d\kappa^2\log\left(\frac{2}{1-e^{-\epsilon/11}}\right)}{\rho_-^2\epsilon^2}, \;\frac{71632G_+^2d^3}{\rho_-^2\epsilon^2}\right\}$.
\State \textbf{Choose:} lattice spacing $\Delta =\min\left\{\frac{(1-e^{-\epsilon/11)}}{6d\sqrt{\beta}}, \frac{\sqrt{\pi}}{120d\sqrt{\beta}}\right\}$.
\State \textbf{Choose:} padding $N_p = \max\left\{\frac{3}{\sqrt{\beta\Delta^2}}, \frac{20d}{\sqrt{e^9\pi}}\sqrt{\frac{1}{\beta\Delta^2}}\right\}$.
\State \textbf{Construct:} padded lattice $\Lambda_{pad}$ and graph $\mathcal G := \{(i,j) : \; i = j \text{ or } \|y_i - y_j\| = \Delta\}$
\For{each lattice point $y_k \in \Lambda_{pad}$}
    \State Define ${\psi}_k(x) \gets \left( \frac{\beta}{2\pi} \right)^{d/2} \exp\left( -\frac{\beta}{2} \|\theta(x) - y_k\|^2 \right)$
\EndFor
\State \Return The collection of windows $\{\psi_k\}_{k=1}^L$ and the graph $\mathcal G$.
\end{algorithmic}
\end{algorithm}
\begin{theorem}\label{t:indicator final bound}
    Let Assumptions \ref{a: good truncation}, \ref{a: good samples}, \ref{a: good path} and \ref{a: collective variable} hold. 
    Let $\bar{\pi}[g]$ be the EMUS estimator computed on $L$ window using the sparse matrix $F^0$, defined on $G$, with $m = |G|$.
Let $E \subset \mathcal{A} \subseteq \mathbb{R}^d$ be a convex set with geometric condition number $\kappa = \frac{R}{\rho}$. Let $\mu \propto e^{-A(z)}$ have a  $G$-Lipschitz potential. 
Define $g := 1_E \circ \theta$. Suppose the Gaussian net is constructed using parameters $\beta, \Delta,$ and $N_p$ exactly as defined in Algorithm \ref{alg:indicator_net_construction}. Define the overlap constant lower bound $c_p^*$, for the adjacent Gaussian sparse matrix, according to Corollary \ref{c: c_p explicit log-lipschitz}:
$$ c_p^* =\frac{
\exp\left(
-\max(\frac{9d}{2},\frac{d^3}{100})
-\frac12
\right)
}{
2^d
\left(
1+\frac{72\sqrt{2\pi}d}{\epsilon}
\right)^d
}). $$
Furthermore, define the uniform variance bound $C_{\mathcal{Q}}^*$ as:
$$ C_{\mathcal{Q}}^* := 2(1 + 0.25 e^{-3})^{1/4} \exp(6). $$
Let $r_s \leq r_{s,*} := \exp\left(
\frac1{24}
+
\frac{5d}{3\sqrt\pi}e^{-9/2}
\right)$, as in Remark \ref{rem: ratio}. For any $\epsilon > 0$ and $\delta > 0$, assume the compact set $K_U$ is chosen such that the truncation parameters satisfy $\max\{\delta_\pi, \delta_g\} \le 1 - e^{-\epsilon/2}$. If the MCMC sampling precision $\tilde{\epsilon}$ and the number of samples per window $N$ satisfy:
$$ \tilde{\epsilon} \le \min\left\{
\frac1{64},
\frac{3}{32r_{s,*}^2},
\frac{\epsilon^2c_p}{69696L^2},
\frac{\epsilon^2}{11616r_{s,*}^2}, \frac{e^{-3}}{4}, \, \frac{\epsilon^2c_p^*}{7744L^2C_{\mathcal Q_*}^2}
\right\}, $$
$$ N \ge \max\left\{
\frac{64r_{s,*}^2}{\delta},
\frac{836,352mL^2}{\epsilon^2c_p^*\delta},
\frac{69,696r_{s,*}^2}{\epsilon^2\delta},
\frac{139,392mL^2C_{\mathcal Q_*}^2}{\epsilon^2c_p^*\delta}
\right\},$$
so that $
\tilde{\epsilon} = \mathcal{O}\left(\exp(-d)\frac{\epsilon^2c_p^*}{L^2C_{\mathcal Q_*}^2}\right)
$ and $
N = \Omega\left(\exp(d)\frac{mL^2C_{\mathcal Q_*}^2}{\epsilon^2\delta c_p^*}\right).
$
Then the EMUS estimator guarantees relative error, with probability at least $1 - \delta$, 
$$\left|
\log\bar\pi[g]-\log\pi[g]
\right| < \epsilon.$$
\end{theorem}
The following remark makes the upper bound on $r_s$ clear. 

\begin{rem} \label{rem: ratio}
    The geometric regularity constant \(r_s=S_{\max}/S_{\min}\) is controlled explicitly by the construction of the Gaussian net. Indeed, Lemma \ref{l: c_p lowerbound} gives
\[
\log r_s
\le
5d\left(
\sqrt{\frac{\beta\Delta^2}{\pi}}
+
\frac{e^{-9/2}}
{N_p\sqrt{\pi\beta\Delta^2}}
\right).
\]The parameter choices in Algorithm \ref{alg:indicator_net_construction} imply
\[
\Delta =\min\left\{\frac{(1-e^{-\epsilon/11)}}{6d\sqrt{\beta}}, \frac{\sqrt{\pi}}{120d\sqrt{\beta}}\right\},
\qquad
N_p  =\max\{\frac{3}{\sqrt{\beta\Delta^2}},\frac{20d}{\sqrt{e^9\pi}}\sqrt{\frac{1}{\beta\Delta^2}}\}.
\]Consequently,
\[
\log r_s
\le
\frac1{24}
+
\frac{5d}{3\sqrt\pi}e^{-9/2},
\]and hence
\[
r_s
\le
r_{s,d}^*
:=
\exp\left(
\frac1{24}
+
\frac{5d}{3\sqrt\pi}e^{-9/2}
\right).
\]In particular, when \(d\le3\),
\[
r_s^2\le (r_{s,*})^2<2.
\]
\end{rem}
\subsection{Upper bound for numerator}\label{ss:upper bound for number}
For $g=1_E\circ \theta$ (an indicator function in the collective variable space), consider $h_\lambda=e^{-\lambda d(E, \theta(x))}$, which is $\lambda$-log-Lipschitz in the CV space. Choose $\lambda$ large enough so that $\bar{\mathcal{S}}(\bar w, h_\lambda)\le (1+\ep)\bar{\mathcal{S}}(\bar w, g)$ whp. (This can follow from the corresponding estimate with $w$ plus a ratio bound between $\bar w$ and $w$.) Then whp (using the bounds for log-Lipschitz $h_\lambda$)
\[
\bar{\mathcal{S}}(\bar w, g) \le 
\bar{\mathcal{S}}(\bar w, h_\lambda)\le 
(1+\ep)\mathcal{S}(w, h_\lambda)
\le (1+\ep)^2 \mathcal{S}(w, g).
\]

\begin{proposition}\label{p: upper_bound_g_L} Let Assumptions \ref{a: good samples}, \ref{a: good path} and \ref{d: target distribution pushforward} hold. 
Let $E \subset \mathcal{A} \subseteq \mathbb{R}^d$ be a convex set with geometric condition number $\kappa = \frac{R}{\rho}$. Let $\mu \propto e^{-A(z)}$ have a  $G$-Lipschitz potential. 
Define $g := 1_E \circ \theta$.
Suppose the Gaussian net is constructed using parameters $\beta, \Delta,$ and $N_p$ as defined in Algorithm \ref{alg:gaussian_net_construction}, with $L_g = 2\left(\frac{d}{\rho}+G\kappa\right)(1 + \frac{8}{5\epsilon})$. 
Assume the setting and requirements for $N$ and $\tilde \epsilon$ of Theorem \ref{t: Explicit Error for Log-lipschitz observable}. Then, with probability $1-\delta$,
$$\bar{S}(\bar{w},g) \leq e^\epsilon S(w,g).$$
\end{proposition}
The proof of this result is deferred till after proving an intermediary lemma. The intermediary lemma gives a sufficient results for proving the final upper bound. 

\begin{lem}\label{l: vol_upper_bound_argument}
    Let $E \subset \mathcal{A} \subseteq \mathbb{R}^d$ be a convex set with geometric condition number $\kappa = \frac{R}{\rho}$. Let $\mu \propto e^{-A(z)}$ have a  $G$-Lipschitz potential. Define $a := \frac{d}{\rho}+G\kappa$ and assume $\lambda > a$. Then
$$\int_{\mathcal{A}\smallsetminus E} e^{-\lambda d(E, z)} \mu(dz) \leq \frac{a}{\lambda - a}\mu(E).$$
\end{lem}
\begin{proof} First define the set 
$$E_t^+
:=
\{z\in\mathcal A:d(z,E)\le t\} = (E+B(0,t))\cap\mathcal A.$$
    We first consider the LHS integral, which can be rewritten as
    \begin{align*}
\int_{\mathcal{A}\smallsetminus E} e^{-\lambda d(E, z)} \mu(dz)  &= \int_{\mathcal A\setminus E}
\int_0^\infty
\lambda e^{-\lambda t}1_{\{d(z,E)\le t\}}
\,dt\,\mu(dz).\\
\shortintertext{Tonelli's Theorem allows us to exchange the order of integration so that}
&= \int_0^\infty\lambda e^{-\lambda t}\left[
\int_{\mathcal A\setminus E}
1_{\{d(z,E)\le t\}}\,\mu(dz)
\right]dt\\
\shortintertext{Rewriting the indicator gives}
&=  \int_0^\infty\lambda e^{-\lambda t}\left[\mu(E_t^+\setminus E)
\right]dt\\
&= \int_0^\infty\lambda e^{-\lambda t}\left[\mu(E_t^+) - \mu(E)
\right]dt.
    \end{align*}
Therefore to continue from
\begin{equation}\label{e: continue from}
    \int_{\mathcal{A}\smallsetminus E} e^{-\lambda d(E, z)} \mu(dz) \leq \int_0^\infty\lambda e^{-\lambda t}\left[\mu(E_t^+) - \mu(E)
\right]dt
\end{equation}
we first need to find an upper bound of the form 
\begin{equation}\label{e: suff_vol_arg}
    \mu(E_t^+) \leq c\cdot\mu(E). 
\end{equation}
To do this we first define the set $E$ stretch by a factor $\gamma_t$, about the center $\vec{c}$ of $E$ as 
$$D_t(E) := \left\{\vec{c} +  \gamma_t(z - \vec{c}): z \in E \right\} = \vec{c} + \gamma_t(E-\vec{c}).$$
Take $\gamma_t = \left(1 + \frac{t}{\rho}\right)$ then for $z \in E_t^+$ since $E_t^+\subseteq E+B(0,t)\subseteq c+(E-c)+\frac{t}{\rho}(E-c)= c+\left(1+\frac{t}{\rho}\right)(E-c)=D_t(E),$ we have that $z \in D_t(E)$. So the choice of $\gamma_t = \left(1 + \frac{t}{\rho}\right)$ implies  $E_t^+ \subseteq D_t(E)$ and hence 
$$\mu(E_t^+) \leq \mu(D_t(E)).$$
We now proceed by bounding the following integral 
\begin{align*}
     \mu(D_t(E)) &= \frac{1}{Z_\mathcal{A}}\int_{D_t(E) \cap \mathcal{A}} e^{-A(z)}dz.\\
     \shortintertext{Using the change of variable $U_t(z) = \vec{c} + \gamma_t(z-\vec{c})$ with the change of volume $\det DU_t(z) = \gamma_t^d$ we have that}
     &=\frac{\gamma_t^d}{Z_\mathcal{A}}\int_{\{z \in E \;:\; U_t(x) \in \mathcal{A}\}} e^{-A(U_t(z))}dz. \\
     &\leq \frac{\gamma_t^d}{Z_\mathcal{A}}\int_{E} e^{-A(U_t(z))}dz.
\end{align*}
Using that $\{z \in E \;:\; U_t(x) \in \mathcal{A}\}\subseteq E$ gives
\begin{equation}\label{e: intermediate_step}
    \mu(E_t^+) \leq \frac{\gamma_t^d}{Z_\mathcal{A}}\int_{E} e^{-A(U_t(z))}dz.
\end{equation}
We consider that 
$$\left\|U_t(z) - z\right\| = \left\|\vec{c}+(1+\frac{t}{\rho})(z - \vec{c}) - z\right\| = \frac{t}{\rho}\left\|z - \vec{c}\right\|.$$
Since $\vec{c}$ is the center of $E$, $\left\|z - \vec{c}\right\|$ this is less than the outer radius $R$ and therefore 
$$\left\|U_t(z) - z\right\| \leq \frac{tR}{\rho} = \kappa t.$$
Lastly, recalling that the potential $A(z)$ is $G$-Lipschitz gives us that 
$$\left\|A(U_t(z)) - A(z)\right\| \leq G \left\|U_t(z) - z\right\| \leq G\kappa t.$$
In practicular, this gives $A(U_t(z)) \geq A(z) - G\kappa t$ which implies 
$$e^{-A(U_t(z))} \leq e^{G\kappa t}e^{-A(z)}.$$ 
Returning to eq. (\ref{e: intermediate_step}) gives 
$$\mu(E_t^+) \leq \frac{\gamma_t^d}{Z_\mathcal{A}}\int_{E} e^{-A(U_t(z))}dz. \leq \frac{\gamma_t^d}{Z_\mathcal{A}}e^{G\kappa t}\int_{E}e^{-A(U_t(z))} dz = \gamma_t^de^{G\kappa t}\mu(E),$$
as required by eq. (\ref{e: suff_vol_arg}). Using that 
$$\gamma_t^d = (1+\frac{t}{\rho})^d = e^{d\log(1+\frac{t}{\rho})} \leq e^{\frac{dt}{\rho}}$$
gives the bound 
$$\mu(E_t^+) \leq e^{\frac{dt}{\rho}+G\kappa t} \mu(E).$$
To finish the proof we continue from eq. (\ref{e: continue from}). Using the previous bound yields 
\begin{align*}
     \int_{\mathcal{A}\smallsetminus E} e^{-\lambda d(E, z)} \mu(dz) &\leq \int_0^\infty\lambda e^{-\lambda t}\left[\mu(E_t^+) - \mu(E)
\right]dt\\
&\leq \mu(E)\int_0^\infty\lambda e^{-\lambda t}\left[e^{\frac{dt}{\rho}+G\kappa t}- 1
\right]dt.\\
\shortintertext{Letting $a = \frac{d}{\rho}+ G \kappa$ gives}
&= \mu(E)\int_0^\infty\lambda e^{-(\lambda - a)t}dt - \mu(E)\int_0^\infty\lambda e^{-\lambda t}dt\\
&=\left(\frac{\lambda}{\lambda-a} - 1\right)\mu(E)\\
&= \frac{a}{\lambda-a}\mu(E).
\end{align*}
 \end{proof}

\begin{proof}(Proof of Proposition \ref{p: upper_bound_g_L})
    Define $h_\lambda=e^{-\lambda d(E, \theta(x))}$. Now we want to show that we can establish the following string of inequalities: 
    \[
\bar{\mathcal{S}}(\bar w, g) \le 
\bar{\mathcal{S}}(\bar w, h_\lambda)\le 
e^{\eta_{h_\lambda}^+}\mathcal{S}(w, h_\lambda)
\le e^{\eta_{h_\lambda}^+}e^{\eta_{g}^+} \mathcal{S}(w, f).
\]
The first inequality holds by monotonicity of $g \leq h_\lambda$ (Lemma \ref{l: monotonicity of S}). For the second inequality, since $g_\lambda$ is a log-Lipschitz function we can directly apply the first part of Theorem \ref{t: Explicit Error for Log-lipschitz observable} with the appropriate algorithmic constants to get: 
$$\bar{\mathcal{S}}(\bar w, h_\lambda) \le  e^{\frac{3\epsilon}{8}}\mathcal{S}(w,h_\lambda).$$
Moreover, since $h_\lambda$ is defined on the collective variable space the constants $L_{\bar h_\lambda}=\lambda$,  
$L_{\overline{h}_\lambda^2}=2\lambda$ and $L_{\mathrm{fiber}}=0$. Therefore $L_h= \max\{L_{\bar h_\lambda},L_{\overline{h}_\lambda^2}\} = 2\lambda$ in the algorithmic parameter choices. 
However, $\lambda$ is an analytical device to which we will pick the optimal $\lambda$ once the full bound is established. 

The third inequality requires some work. First we note the following: 
\begin{align*}
    \frac{\mathcal S(w,h_\lambda)}{\mathcal S(w,g)} &= \frac{\sum_k w_k \int_\mathcal{A}\frac{1}{S(z)} \frac{\tilde{\psi}_k(z)}{Z_k}e^{-\lambda d(E,z)}\mu(dz)}{\sum_k w_k \int_\mathcal{A}\frac{1}{S(z)} 1_E\frac{\tilde{\psi}_k(z)}{Z_k}\mu(dz)}= \frac{\sum_k cZ_k \int_\mathcal{A}\frac{1}{S(z)} e^{-\lambda d(E,z)}\frac{\tilde{\psi}_k(z)}{Z_k}\mu(dz)}{\sum_k cZ_k \int_\mathcal{A}\frac{1}{S(z)} 1_E\frac{\tilde{\psi}_k(z)}{Z_k}\mu(dz)}\\
    &= \frac{c\int_\mathcal{A}\frac{S(z)}{S(z)}e^{-\lambda d(E,z)}\mu(dz)}{c\int_\mathcal{A}\frac{S(z)}{S(z)}1_E\mu(dz)} = \frac{\mu(E) + \int_{\mathcal{A}\smallsetminus E}e^{-\lambda d(E,z)}\mu(dz)}{\mu(E)}\\
     &= 1 + \frac{\int_{\mathcal{A}\smallsetminus E}e^{-\lambda d(E,z)}\mu(dz)}{\mu(E)}
\end{align*}
Therefore its sufficient to find a bound of the form 
$$\int_{\mathcal{A}\smallsetminus E}e^{-\lambda d(E,z)}\mu(dz) \leq \epsilon'\mu(E).$$
This bound is established in Lemma \ref{l: vol_upper_bound_argument} and is given as 
$$\int_{\mathcal{A}\smallsetminus E} e^{-\lambda d(E, z)} \mu(dz) \leq \frac{a}{\lambda - a}\mu(E),$$
with $a = \frac{d}{\rho}+G\kappa$. Therefore 
$$ \frac{\mathcal S(w,h_\lambda)}{\mathcal S(w,g)} \leq 1 + \frac{a}{\lambda - a}.$$
To get a final bound of $e^\epsilon$ we require $1 + \frac{a}{\lambda - a} \leq e^{\frac{5}{8}\epsilon}$. This is equivalent to acquiring the bound $\log(1+\frac{a}{\lambda - a}) \leq \frac{5}{8}\epsilon$. We proceed by considering
\begin{align*}
    \log(1+\frac{a}{\lambda - a}) &\leq \frac{a}{\lambda - a}.\\ 
    \shortintertext{Then assuming $\lambda \geq a(1+\frac{8}{5\epsilon})$ yields }
    &\leq  \frac{5}{8}\epsilon
\end{align*}
as desired. Combining all 3 bounds gives
 \[
\bar{\mathcal{S}}(\bar w, g) \le 
\bar{\mathcal{S}}(\bar w, h_\lambda)\le 
e^{\frac{3}{8}\epsilon}\mathcal{S}(w, h_\lambda)
\le e^{\frac{3}{8}\epsilon} e^{\frac{5}{8}\epsilon} \mathcal{S}(w, g).
\]
Which implies 
$$\bar{\mathcal{S}}(\bar w, g) \leq e^\epsilon \mathcal{S}(w, g).$$
\end{proof}
\subsection{Lower Bound for Numerator}\label{ss:lower bound for numerator}
We state a proposition for the lower bound on $\bar{S}(\bar{w},g)$, and defer the proof until after we establish a few intermediary lemmas. 
\begin{proposition}\label{p:full_lower_bound}
    Let Assumptions \ref{a: good samples} and \ref{a: good path} hold. 
    Let $g = 1_E \circ \theta$ and suppose the assumptions of Lemmas \ref{l: lower bound good set}-\ref{l:lower bound volume} hold. 
    In particular, assume 
    \[
\beta>\ell,
\qquad
\eta_\Lambda\le\frac{1}{2\sqrt{d\beta}},
\]
and define 
\[q(r)
:=
2\sqrt{d\beta}\,\eta_\Lambda
+\frac52\beta\eta_\Lambda^2
+\exp\left(-\frac{\beta r^2}{32}\right).\]
Let $r_s \leq \exp\left(
\frac1{24}
+
\frac{5d}{3\sqrt\pi}e^{-9/2}
\right)$, as in Remark \ref{rem: ratio}. Suppose
\(
B(c,\rho)\subseteq E\subseteq B(c,R),
\) with \(
\kappa=\frac R\rho.
\) 
For $r \in \mathcal{R}$ with 
\[\mathcal{R} := \left\{r>0: r>2\sqrt{\frac d\beta},
\;
2r<\rho,
\;
r\ge \fc{ G+\sqrt{2d(\beta+\ell)}}{\beta}
\;
q(r)<1,\right\},\]
assume the sampling-size and sampling-bias conditions of Lemma \ref{l: lower bound good set} hold with $p_E = p(r) \geq \frac{1}{2}$, as in Lemma \ref{l:good-threshold}, let $ G_+=\max\{1,G\}$ and $\rho_- = \min\{1,\rho\}$.
Then with probability at least $1-\delta_w - \delta_f$, 

\[\overline S(\overline w,g)
\ge
e^{-\Gamma(p(r))}
\bigl(1-q(r)\bigr)
\left(1-\frac{2r}{\rho}\right)^d
e^{-2G_+\kappa r}
S(w,g).\]
In particular, for $\delta_w = \delta_f = \frac{\delta}{2}$ and if 
\[
N\ge
\max\left\{
\frac{64r_s^2}{3\delta},
\frac{36864mL^2}{\epsilon^2c_p\delta},
\frac{3072r_s^2}{\epsilon^2\delta}
\right\},
\]
\[
\widetilde\epsilon
\le
\min\left\{
\frac1{64},
\frac{3}{32r_s^2},
\frac{\epsilon^2c_p}{9316L^2},
\frac{\epsilon^2}{1536r_s^2}
\right\},
\] 
\[\beta > \max\left\{2\ell, \; \frac{2368 \cdot G_+^4d\kappa^2\log\left(\frac{2}{1-e^{-\epsilon/4}}\right)}{\rho_-^2\epsilon^2}, \;\frac{9472G_+^2d^3}{\rho_-^2\epsilon^2}\right\} \quad \text{ and } \quad \eta_\Lambda\le\frac{(1-e^{-\epsilon/4})}{6\sqrt{d\beta}}.\]
Then with probability at least $1-\delta$, 

\[\overline S(\overline w,g)
\ge
e^{-\epsilon}
S(w,g).\]
\end{proposition}

\begin{lem}[Monotonicity]\label{l: monotonicity of S}
    Suppose $f\le g$. Then 
    \[
\bar{\mathcal{S}}(w, f)\le \bar{\mathcal{S}}(w, g),
    \]
    when the same samples are used for the empirical estimates. The same holds for $\mathcal S$.
\end{lem}
Note this is an inequality between random variables that always holds.
    More generally, define 
    \[
\bar{\mathcal{S}}(w, (f_k)_{k\in W})
= \sum_{k\in W} w_k (\bar{f_k})^*_k, 
    \]
    that is, we have the freedom to let the functions $f_k$ be different in different windows. Then if $f_k\le g_k$ for all $k\in W$, then 
        \[
\bar{\mathcal{S}}(\bar w, (f_k)_{k\in W})\le \bar{\mathcal{S}}(\bar w, (g_k)_{k\in W}).
    \]
    In particular, this allows us to obtain a lower bound by zeroing out the windows where $\E_{\pi_k}[g]<\ep$; this throws out the terms in the sum that could have large relative variance.
    \begin{align}
    \label{e:sum-stdev-lb}
\bar{\mathcal{S}}^{\text{lower}}(\bar w, g)
:=\sum_{k\in W,\E_{\pi_k}[1_E]\ge \ep} \bar w_k \bar{g}^*_k
\le 
\bar{\mathcal{S}}(\bar w, g).
    \end{align}

\paragraph{Lower bound.}
Choose the windows and $\ep$ small enough so that we can get a good lower bound on the LHS of \eqref{e:sum-stdev-lb}.

TODO: work out explicit bounds when $E$ is a convex set with condition number $\ka$.

Note that if $S_{\min}\le S(x)\le S_{\max}$, then 
\begin{align}
\label{e:rel-std-dev-S}
\fc{\sqrt{\Var_\pi [1_E/S]}}{\E_\pi[1_E/S]} \le
\fc{\sqrt{\E_\pi [1_E/S^2]}}{\E_\pi[1_E/S]}
\le \fc{\sqrt{\E_\pi [1_E]}}{\E_\pi[1_E]}\cdot \fc{1/S_{\min}}{1/S_{\max}} 
= \fc{S_{\max}}{S_{\min}} \cdot \rc{\sqrt{\pi(E)}} 
\end{align}

Say that a window $k$ is good/bad at threshold $p$ if $\E_{\pi_k}[1_E]\ge p$ or $\E_{\pi_k}[1_E]<p$, respectively. Let $E_r^- = \seta{x\in E}{d(x,E^c)\ge r}$. We say that $x$ is good/bad at the threshold if the window centered at $x$ is good/bad.

\begin{rem}\label{r: family_of_functions}
    The error bounds in Lemmas \ref{l: weight erorr}, \ref{l: Sampling Error}, and \ref{l: bias error} continue to hold when the observable is allowed to depend on a window index. Specifically, for a collection of non-negative observes $(g_k)_{k \in W}$ define 
    \[
\bar{\mathcal{S}}(w, (g_k)_{k\in W})
= \sum_{k\in W} w_k (\bar{g}_k)^*. 
    \]   
    With this definition Lemmas \ref{l: Sampling Error}, and \ref{l: bias error} continue to hold as the analysis is performed on each separate window. 
    Additionally, Lemma \ref{l: weight erorr} has no dependence on the observable function. 
    These are the ingredients required to apply the first part of Theorem \ref{t: general error decomposition} and Corollary \ref{c: main_theorem} to a family of functions. 
\end{rem}
\begin{lem}\label{l: lower bound good set}
Let Assumptions \ref{a: good samples} and \ref{a: good path} hold. 
Let $\bar{\mathcal{S}}(\bar w, (g_k)_{k\in W})$ be the estimator of ${\mathcal{S}}( w, (f_k)_{k\in W})$ computed on $L$ windows using the sparse matrix $\bar F^0$, defined on $G$, with $m = |G|$. 
Define the family of functions
$$(g_k)_{k\in W}(x) = \begin{cases}
  g(x) & \text{if } k \in W \\
  0 & \text{otherwise}
\end{cases},$$
for $g:=1_E\circ\theta$ and $W= \left\{ k : \E_{\pi_k}[1_E\circ \theta]\ge p_E\right\}$. Let $r_s = \frac{s_{max}}{s_{min}}$.
Suppose every window contains $N_k \ge N$ samples. Choose $\epsilon_w>0$, a weight failure probability $\delta_w>0$, and an observable sampling failure probability $\delta_f>0$. If 
\begin{align*}
N&\ge \max\left\{\frac{16r_s^2}{3p_E\delta_f},\frac{32mL^2}{\epsilon_w^2c_p\delta_w}\right\}\\
\shortintertext{and}
\tilde\epsilon
&\le
\min\left\{\frac{p_E^2}{16},\frac{3p_E}{16r_s^2},\frac{\epsilon_w^2c_p}{16L^2}\right\}. 
\end{align*}
Defining 
$$\eta \vcentcolon =2\epsilon_w + \frac{4r_s}{\sqrt{3p_EN\delta_f}} + 4\sqrt{\widetilde\epsilon}\,\frac{r_s}{\sqrt{3p_E}},$$
yields, with probability $1 - \delta_w- \delta_f $,
$$\bar{\mathcal{S}}(\bar w, (g_k)_{k\in W}) \ge e^{-\eta}\mathcal{S}(w, (g_k)_{k\in W}) := S^{lower}(w,g).$$
\end{lem}

\begin{proof}
    Taking note of Remark \ref{r: family_of_functions}, the result follows directly from Corollary \ref{c: main_theorem}. 
    Where we make specific a bound on $C_{\mathcal{Q}_*}$ in the case of $(g_k)_{k\in W}$.
    The bound is given by considering
    \begin{align*}
        \mathcal{Q}_{\{\pi\}}(w,(g_k)_{k\in W}) = \frac{\sum_{k\in W} w_k\sqrt{\Var_{\pi_k}(g^*)}}{\sum_{k\in W}  w_k \mathbb{E}_{\pi_k}(g^*)} = \frac{\sum_{k\in W} w_k\mathbb{E}_{\pi_k}(g^*)\frac{\sqrt{\Var_{\pi_k}(g^*)}}{\mathbb{E}_{\pi_k}(g^*)}}{\sum_{k\in W}  w_k \mathbb{E}_{\pi_k}(g^*)}.
    \end{align*}
    Note that if $S_{\min}\le S(x)\le S_{\max}$ and $r_s = \fc{S_{\max}}{S_{\min}}$, for general $\nu$ and $f(x) = 1_B$ we have the bound:
\begin{align}
\label{e:rel-std-dev-S}
\fc{\sqrt{\Var_\nu [1_B/S]}}{\E_\nu[1_B/S]} \le
\fc{\sqrt{\E_\nu [1_B/S^2]}}{\E_\nu[1_B/S]}
\le \fc{\sqrt{\E_\nu [1_B]}}{\E_\nu[1_B]}\cdot \fc{1/S_{\min}}{1/S_{\max}} 
= \fc{S_{\max}}{S_{\min}} \cdot \rc{\sqrt{\nu(B)}} \leq \fc{r_s}{\sqrt{\nu(B)}}.
\end{align}
Applying this to $\pi$ with $f(x) = 1_E\circ\theta$ for $k \in W$ yields 
$$\fc{\sqrt{\Var_{\pi_k} [1_E\circ \theta/S]}}{\E_{\pi_k}[1_E \circ \theta/S]} \leq \fc{r_s}{\sqrt{\pi_k(E\circ\theta)}} \leq \fc{r_s}{\sqrt{p_E}}.$$
Combining this with the previous bound gives 
$$ \mathcal{Q}_{\{\pi\}}(w,(g_k)_{k\in W}) \leq  \frac{r_s}{\sqrt{p_E}}.$$
To make specific the bound for the approximate measure $\tilde{\pi}_k$ we note the following 
$$\left| \nu(B) - \tilde \nu(B)\right| \leq \int_B \left| \nu(x) - \tilde \nu(x)\right| dx \leq \int_\Omega \left| \nu(x) - \tilde \nu(x)\right| dx  \leq \chi^2(\tilde{\nu}, \nu)^\frac{1}{2}.$$
This implies that $\tilde \nu(B) \geq \nu(B) - \chi^2(\tilde{\nu}, \nu)^\frac{1}{2}$.
Applying this to $\nu = \pi_k$ for each $k \in W$, using Assumption \ref{a: good samples} to note that $\chi^2(\tilde \pi_k, \pi_k) \leq \tilde\epsilon$ and lastly assuming that $\tilde{\epsilon}^\frac{1}{2} \leq \frac{p_E}{4}$ yields 
$$\tilde\pi_k(1_E \circ \theta) \geq \pi_k(1_E \circ \theta) - \tilde \epsilon^\frac{1}{2} \geq p_E -\frac{p_E}{4} = \frac{3}{4}p_E.$$
Therefore
$$\max\{ \mathcal{Q}_{\{\tilde\pi\}}(w,(g_k)_{k\in W}),  \mathcal{Q}_{\{\pi\}}(w,(g_k)_{k\in W})\} \leq \frac{2r_S}{\sqrt{3p_E}}.$$
We note that that the assumption $\tilde{\epsilon}^\frac{1}{2} \leq \frac{p_E}{4}$ means we require $\tilde{\epsilon} \leq \frac{p_E^2}{16}$.
Substituting $C_{\mathcal Q_*} = \frac{2r_S}{\sqrt{3p_E}}$ into the statement of Corollary \ref{c: main_theorem} gives the final bound.
\end{proof}

\begin{lem} \label{l:good-threshold}
If $x\in E_r^-$, then $x$ is good at threshold $p_r=1-\exp\pa{-\fc{(\be - \ell)[\be r - G - \sqrt{2d(\be+\ell)}]^2}{2(\be+\ell)^2}}$. In particular, taking $r_{1/2}= \frac{
G+\sqrt{2d(\beta+\ell)}
+(\beta+\ell)\sqrt{\frac{2\log 2}{\beta-\ell}}
}{\beta}$, this is $\ge 1/2$.

\end{lem}

\begin{proof}
By \eqref{e:rel-std-dev-S} it suffices to bound $\mu_k(E)$, on the collective variable space. Write $\mu(x) \propto e^{-A(z)}$, and let $A_k(x) = A(x) + \fc{\be}{2}\ve{z-y_k}^2$, so that $\pi_k(x)\propto e^{-V_k(x)}$. Note $V_k$ is $(\be + \ell)$-smooth and $(\be - \ell)$-strongly convex. 
By \cite[Lemma 16]{chewi2025analysis}, \[
\E_{\mu_k} \ve{\nabla A_k(z)}
\le [\E_{\mu_k} \ve{\nabla A_k(z)}^2]^{1/2} \le \sqrt{2d(\be + \ell)}.\] 
Note that if $\ve{z-y_k}>r$, then because $\nb A$ is $G$-Lipschitz at $x_k$,
\begin{align*}
\ve{\nb A_k(z)}&\ge 
\ve{\be (z-y_k)} - 
\ve{\nb A(z)} >
\be r - G\\
\ve{\nb A_k(z)} - \E_{\mu_k} \ve{\nb A_k(z)}
&\ge \be r - G - \sqrt{2d(\be+\ell)}.
\end{align*}
Then because $\mu_k$ satisfies a log-Sobolev inequality with constant $\rc{\be-\ell}$ by Bakry-Emery and $\nb A_k$ is $(\be+\ell)$-Lipschitz, by sub-Gaussian concentration of Lipschitz functions under a log-Sobolev inequality,
\begin{align*}
\mu_k(E^c) \le 
    \mu_k (B_r(y_k)^c)
    &\le 
    \mu_k \pa{\ve{\nb A_k(z)} - \E_{\mu_k} \ve{\nb A_k(x)}
\ge \be r - G - \sqrt{2d(\be+\ell)}}
    \\
    & \le 
    \exp\pa{-\fc{(\be - \ell)[\be r - G - \sqrt{2d(\be+\ell)}]^2}{2(\be+\ell)^2}}.
\end{align*}
\end{proof}

The following shows that for any radius $r$ and error $\ep>0$, we can choose the Gaussian variances $\rc{\be}$ and lattice spacing $\eta$ small enough so that the lower estimate captures at least $1-\ep$ of the mass of the set shrunk by $r$.
\begin{lem} \label{l: good set to volume}
Suppose the windows are centered on a lattice $\La$ with covering radius $\eta$.
Suppose that 
$\eta\le\frac{1}{2\sqrt{d\beta}}$
and $r>2\sfc{d}{\be}$. Then 
    $\mathcal S^{\textup{lower}}(w,1_E) \ge (1-\ep_r) \mathcal S(w,1_{E_{2r}^-})$ for
$\epsilon_r:=2\sqrt{d\beta}\,\eta +\frac{5}{2}\beta\eta^2+e^{-\frac{\beta r^2}{32}}.$

\end{lem}
\begin{proof}
We have 
\begin{align*}
    \mathcal S^{\textup{lower}}(w,1_E)
    &= \sum_{k \text{ good at }p_r} \fc{\cancel{\int \pi \psi_k}}{\sum_{k'} \int\pi \psi_k} \int_{E} \fc{\pi \psi_k / \cancel{\int \pi \psi_k}}{\sum_{k'} \psi_{k'}}\\
    &\ge  \sum_{k \text{ good at }p_r} \rc{{\sum_{k'} \int\pi \psi_k}}
    \int_{E_{2r}^-} \fc{\pi \psi_k }{\sum_{k'} \psi_{k'}}\\
    &\ge \rc{\sum_{k'} \int\pi \psi_k}
    \int_{E_{2r}^-} \fc{\pi \sum_{k \text{ good at }p_r} \psi_k }{\sum_{k'} \psi_{k'}}
\end{align*}
If we can show that for all $x\in E_{2r}^-$,  $ \sum_{k \text{ good at }p_r} \psi_k(x) \ge (1-\ep_r) \sum_{k'} \psi_{k'}(x)$, then the conclusion follows.

Consider $x\in E_{2r}^-$. 
By Lemma \ref{l:good-threshold}, 
\begin{align*}
    \sum_{k\text{ good at }p_r}\psi_k(x)
    &\ge \sum_{k:x_k\in E_r^-} \psi_k(x)
    \ge \sum_{k: x_k\in B_r(x)}\psi_k(x)
    =\sum_{y\in B_r(x)\cap \La}\psi(y-x)
\end{align*}
Let $D_{\La}$ be the determinant of the fundamental parallelopiped for $\La$. When $\eta\le \rc{2\sqrt{d\be}}$, letting $\ep_1 =\sqrt{d\be}\eta + \fc32\be \eta^2$, $\ep_2 = \sqrt{d\be}\eta + \be \eta^2$, by Proposition \ref{p:sum-integral} and Lemma \ref{l:gaussian-error},
\begin{align*}
D_\La \sum_{y\in \La} \psi(y-x) - \ub{\int_{\R^n} \psi(y)\,dy}{1}
&\le  e^{\ep_1}-1 \\
\implies
\rc{D_\La} = \rc{D_\La} \int_{\R^n}\psi(y)\,dy 
&\ge e^{-\ep_1}\sum_{y\in \La}\psi(y-x)
\\
D_\La \sum_{x\in \La \cap B_r(x)} \psi(y-x) - \int_{B_r(x)} \psi(y)\,dy
&\ge  
-\ep_2 - 
\int_{B_{r}(0)\bs B_{r-\eta}(0)} \psi(y)\,dy.
\end{align*}
By a standard $\chi^2$ concentration bound \cite{laurent2000adaptive},
\begin{align*}
    \sum_{x\in \La \cap B_r(x)} \psi(y-x) 
    &\ge \rc{D_\La} \pa{\int_{B_r(0)} \psi(y)\,dy - \ep_2 - \int_{B_{r}(0)\bs B_{r-\eta}(0)} \psi(y)\,dy}\\
    &\ge \rc{D_\La} \pa{\int_{B_r(0)} \psi(y)\,dy - \ep_2 - \int_{B_{}(0)\bs B_{r-\eta}(0)} \psi(y)\,dy}\\
    & = \rc{D_\La} \pa{\int_{B_{r-\eta}(0)} \psi(y)\,dy - \ep_2\int_{\R^n} \psi(y)\,dy }\\
    &\ge \rc{D_\La}\pa{(1-e^{-(\sqrt{\rc 2 (\be (r-\eta)^2-\fc d2)} - \fc d2)}-\ep_2)\int_{\R^n} \psi(y)\,dy}\\
    &\ge e^{-\ep_1} (1-e^{-(\sqrt{\rc 2 (\be (r-\eta)^2-\fc d2)} - \fc d2)}-\ep_2) \sum_{y\in \La} \psi(y).
\end{align*}
Let $T = e^{-(\sqrt{\rc 2 (\be (r-\eta)^2-\fc d2)}-\fc{d}{2})}$ then $e^{-\varepsilon_1}(1-T-\varepsilon_2). \ge 1 - \epsilon_1 - \epsilon_2 - T$. For $\epsilon_1 = \sqrt{d\beta}\,\eta+\frac32\beta\eta^2$ and $\epsilon_2 = \sqrt{d\beta}\,\eta+\beta\eta^2$ we get \(e^{-\epsilon_1}(1-T-\epsilon_2)\ge 1 -2\sqrt{d\beta}\,\eta -\frac52\beta\eta^2 -T.\)
Applying $\eta \le \fc{1}{2\sqrt{d\beta}}$ and $r > 2\sqrt{\fc{d}{\beta}}$ yields $T \leq e^{-\frac{\beta r^2}{32}}$. 
This gives the final bound 
\(e^{-\epsilon_1}(1-T-\epsilon_2)\ge 1 -2\sqrt{d\beta}\,\eta -\frac52\beta\eta^2 -e^{-\frac{\beta r^2}{32}}.\)
\end{proof}
It remains to show that the set shrunk by $r$ still has at least $1-\ep$ portion of the mass, where $\ep$ has good dependence on the condition number of the set.
\begin{lem}\label{l:lower bound volume}
Let Assumptions \ref{a: collective variable} hold, with push-forward measure $\mu(dz)\propto e^{-A(z)}\,dz$ and $\sup_{z\in\mathcal A}\|\nabla A(z)\|\le G.$
Let \(E\subseteq\mathcal A\subseteq\mathbb R^d\) be convex, suppose that there exists $c \in E$ so that $B(c,\rho)\subseteq E\subseteq B(c,R)$ and $ \kappa=\frac{R}{\rho}. $ If $0 \leq r \leq \rho$ then 
    \[
\mathcal S(w, 1_{E_r^-}) \ge (1-\ep_r') \mathcal S(w, 1_E), 
    \]
with $\ep_r^- = 1 - e^{d\log(1-\frac{r}{\rho})-G\kappa r}.$
\end{lem}

\begin{proof}
Consider the ratio 
\begin{align*}
    \frac{\mathcal S(w,1_{E_r^-})}{\mathcal S(w,1_E)} &= \frac{\sum_k w_k \int_{E_r^-}\frac{1}{S(z)} \frac{\tilde{\psi}_k(z)}{Z_k}\mu(dz)}{\sum_k w_k \int_E\frac{1}{S(z)} \frac{\tilde{\psi}_k(z)}{Z_k}\mu(dz)}= \frac{\sum_k cZ_k \int_{E_r^-}\frac{1}{S(z)}\frac{\tilde{\psi}_k(z)}{Z_k}\mu(dz)}{\sum_k cZ_k \int_E\frac{1}{S(z)} \frac{\tilde{\psi}_k(z)}{Z_k}\mu(dz)}\\
    &= \frac{c\int_{E_r^-}\frac{S(z)}{S(z)}\mu(dz)}{c\int_E\frac{S(z)}{S(z)}\mu(dz)} = \frac{\mu({E_r^-})}{\mu(E)}.
\end{align*}
Therefore its sufficient to find a lower bound on the ratio $\frac{\mu({E_r^-})}{\mu(E)}.$
We decompose this by introducing the contraction $T_r(z) = c + (1-\frac{r}{\rho})(z-c)$ where $c$ is the center of $E$. 
This set satisfies $T_r(E) \subseteq E_r^-$ and consequently $\mu(E_r^-) \geq \mu(T_r(E))$. 
Therefore we have that 
$$\frac{\mu({E_r^-})}{\mu(E)} = \frac{\mu({E_r^-})}{\mu(T_r(E))}\cdot\frac{\mu(T_r(E))}{\mu(E)}\geq 1 \cdot \frac{\mu(T_r(E))}{\mu(E)}$$
and it remains to lower bound the ratio $\frac{\mu(T_r(E))}{\mu(E)}$. 
Thereby we consider the following 
\begin{align*}
    \mu(T_r(E)) &= \frac{1}{Z_\mathcal{A}}\int_{T_r(E)\cap \mathcal{A}}e^{-A(z)}dz\\
    &= \frac{(1-\frac{r}{\rho})^d}{Z_\mathcal{A}} \int_{\{z \in E \; : \; T_r(z) \in \mathcal{A}\} } e^{-A(T_r(z))}dz\\
    &= \frac{(1-\frac{r}{\rho})^d}{Z_\mathcal{A}} \int_E  e^{-A(T_r(z))}dz.
\end{align*}
We consider that 
$$\left\|T_r(z) - z\right\| = \left\|\vec{c}+(1+\frac{r}{\rho})(z - \vec{c}) - z\right\| = \frac{r}{\rho}\left\|z - \vec{c}\right\|.$$
Since $\vec{c}$ is the center of $E$, $\left\|z - \vec{c}\right\|$ this is less than the outer radius $R$ and therefore 
$$\left\|T_r(z) - z\right\| \leq \frac{rR}{\rho} = \kappa r.$$
Lastly, recalling that the potential $A(z)$ is $G$-Lipschitz gives us that 
$$\left\|A(U_t(z)) - A(z)\right\| \leq G \left\|U_t(z) - z\right\| \leq G\kappa r.$$
In particular, this gives us that $A(T_r(z)) \leq A(z) + G\kappa r$ which implies 
$$e^{-A(T_r(z))} \geq e^{-A(z)-G\kappa r}.$$
Substituting this back in yields 
$$  \mu(T_r(E)) = \frac{(1-\frac{r}{\rho})^d}{Z_\mathcal{A}} \int_E  e^{-A(T_r(z))}dz \geq (1-\frac{r}{\rho})^de^{-G\kappa r} \int_E \frac{e^{-A(z)}}{Z_\mathcal{A}}dz = e^{d\log(1-\frac{r}{\rho})-G\kappa r} \mu(E).$$
Which yields 
$$\frac{\mu({E_r^-})}{\mu(E)} \geq e^{d\log(1-\frac{r}{\rho})-G\kappa r}.$$
\end{proof}
\begin{proof}(Proof of Proposition \ref{p:full_lower_bound})
Let $\Gamma(p(r)) = 2\epsilon_w + \frac{4r_s}{\sqrt{3p_EN\delta_f}} + 4\sqrt{\widetilde\epsilon}\,\frac{r_s}{\sqrt{3p_E}}$, as in Lemma \ref{l: lower bound good set}.
    The proof is just the chain of inequalities
\[
\begin{aligned}
\overline S(\overline w,g)
&\ge
\overline S^{\mathrm{lower}}(\overline w,g)
\\
&\ge
e^{-\Gamma(p(r))}
S^{\mathrm{lower}}(w,g)
&&\text{by Lemma \ref{l: lower bound good set}}
\\
&\ge
e^{-\Gamma(p(r))}
(1-q(r))
S\left(w,\mathbf 1_{E^-_{2r}}\right)
&&\text{by Lemmas \ref{l:good-threshold} and \ref{l: good set to volume}}
\\
&\ge
e^{-\Gamma(p(r))}
(1-q(r))
\left(1-\frac{2r}{\rho}\right)^d
e^{-2G\kappa r}
S(w,g)
&&\text{by Lemma \ref{l:lower bound volume} applied with \(2r\).}
\end{aligned}
\]
Then let 
\[ 
\eta^-(r) = \Gamma(p(r)) - \log(1-q(r))-d\log(1-\fc{2r}{\rho})+2G\kappa r
\]
so that 
\[\overline S(\overline w,g) \ge e^{-\eta^-(r)}S(w,g).\]
Minimizing $\eta^-(r)$ over $r \in \mathcal{R}$ gives the optimal lower bound. 
Fixing $p(r) = 1/2$ we can then optimize over 
\[\max\left\{
r_{1/2},\,2\sqrt{\frac d\beta}
\right\}
<r<\frac{\rho}{2},\]
with $r_{1/2}$ as in Lemma \ref{l:good-threshold}.
To get that $\eta^-(r) = \epsilon$, we bound each component of the sum by $\epsilon/4$ and then show that the restricts this imposes on $r$ admits at least one valid $r$. 
for $\Gamma(\frac{1}{2}) = 2\epsilon_w + \frac{4r_s}{\sqrt{\frac{3}{2}N\delta_f}}+4\sqrt{\tilde\epsilon}\frac{r_s}{\sqrt{\frac{3}{2}}}$ we bound each individual component by $\epsilon/12$. 
This imposes additional requirements on $N$ and $\tilde{\epsilon}$. It is sufficient to choose
\[
\epsilon_w\le\frac{\epsilon}{24},
\]\[
N\ge
\frac{1536r_s^2}{\epsilon^2\delta_f},
\]and
\[
\widetilde\epsilon
\le
\frac{\epsilon^2}{1536r_s^2}.
\]These choices give
\[
\Gamma(1/2)\le\frac{\epsilon}{4}.
\]They should be combined with the original requirements of Lemma 5.5 by taking
\[
N\ge
\max\left\{
\frac{32r_s^2}{3\delta_f},
\frac{18432mL^2}{\epsilon^2c_p\delta_w},
\frac{1536r_s^2}{\epsilon^2\delta_f}
\right\},
\]and
\[
\widetilde\epsilon
\le
\min\left\{
\frac1{64},
\frac{3}{32r_s^2},
\frac{\epsilon^2c_p}{9216L^2},
\frac{\epsilon^2}{1536r_s^2}
\right\}.
\]
For the second term to get
\(-\log(1-q(r))\le\frac{\epsilon}{4} \) we need
\[
2\sqrt{d\beta}\,\eta_\Lambda
+\frac52\beta\eta_\Lambda^2
+e^{-\beta r^2/32}
\le 1 - e^{-\epsilon/4}.
\]
We can split this into two parts requiring
\[
2\sqrt{d\beta}\,\eta_\Lambda
+\frac52\beta\eta_\Lambda^2
\le\frac{ 1 - e^{-\epsilon/4}}{2} \quad
\text{ and }
e^{-\beta r^2/32}
\le\frac{ 1 - e^{-\epsilon/4}}{2}.
\]Both conditions are satisfied for 
\[
\eta_\Lambda
\le
\frac{
\sqrt{16d+20( 1 - e^{-\epsilon/4})}-4\sqrt d
}{
10\sqrt\beta
}, \quad
r
\ge
\sqrt{
\frac{32}{\beta}
\log\left(\frac{2}{ 1 - e^{-\epsilon/4}}\right)
}.
\]
For the last term require $2G\kappa r \leq \epsilon/4$ enforces 
$$r \leq \frac{\epsilon}{8G\kappa}.$$

In this construction $r$ is an auxillary variable though, $r$ must satisfy the contraints we've defined throughout the proof. Namely, 

\[r
\ge
\max\left\{
r_{1/2},
\;
2\sqrt{\frac d\beta},
\;
\sqrt{\frac{32}{\beta}
\log\left(\frac{2}{1-e^{-\epsilon/4}}\right)}
\right\}\]
and, 
\[r\le
\min\left\{
\frac{\rho}{2}
\left(1-e^{-\epsilon/(4d)}\right),
\;
\frac{\epsilon}{8G\kappa}
\right\}.\]
Using $\beta \geq 2\ell$,
$$ r_{1/2}= \frac{
G+\sqrt{2d(\beta+\ell)}
+(\beta+\ell)\sqrt{\frac{2\log 2}{\beta-\ell}}
}{\beta} \leq \frac{G}{\beta}+\frac{\sqrt{3d}}{\sqrt{\beta}}+\frac{3\sqrt{\log2}}{\sqrt{\beta}}\leq \frac{G + \sqrt{3d} + 3\sqrt{2}}{\sqrt{\beta}}.$$
There we require
$$r \geq \max\left\{\frac{G + \sqrt{3d} + 3\sqrt{2}}{\sqrt{\beta}}, \;
2\sqrt{\frac d\beta},
\;
\sqrt{\frac{32}{\beta}
\log\left(\frac{2}{1-e^{-\epsilon/4}}\right)} \right\}.$$
Which means we need to take
\begin{equation}\label{eq: lower bound beta}
\beta \geq \max\left\{\frac{(G + \sqrt{3d} + 3\sqrt{2})^2}{r^2}, \; \frac{4d}{r^2}, \; \frac{32\log\left(\frac{2}{1-e^{-\epsilon/4}}\right)}{r^2}  \right\}.
 \end{equation}
Using that $(1-e^{-\epsilon/4d}) \geq \frac{\epsilon}{8d}$ for $0 < \epsilon < 1$ and $d \geq 1$ we have 
 $$r \leq \min \left\{\frac{\rho \epsilon}{16d}, \frac{\epsilon}{8G\kappa}\right\}.$$
The upper bound restrictions on $r$ impose the following lower bound on $\beta$ in eq. (\ref{eq: lower bound beta}): 
$$\beta \geq \max\left\{\frac{2368 \cdot G^4d\kappa^2}{\epsilon^2}, \; \frac{32dG\kappa}{\epsilon^2}, \; \frac{2048G^2\kappa\log\left(\frac{2}{1-e^{-\epsilon/4}}\right)}{\epsilon^2},\;\frac{9472G^2d^3}{\rho^2\epsilon^2},\;\frac{64d^2}{\rho^2\epsilon^2}, \; \frac{512d\log\left(\frac{2}{1-e^{-\epsilon/4}}\right)}{\rho^2\epsilon^2} \right\},$$
where we use $(G + \sqrt{3d} + 3\sqrt{2})^2 \leq 37G,$ assuming $G >1$ and $\rho < 1$. Simplifying the lower bound on $\beta$ gives
\begin{equation}\label{e: beta final}
    \beta \geq \max\left\{\frac{2368 \cdot G^4d\kappa^2\log\left(\frac{2}{1-e^{-\epsilon/4}}\right)}{\rho^2\epsilon^2}, \;\frac{9472G^2d^3}{\rho^2\epsilon^2}\right\}.
\end{equation}
Lastly, we simplify combine our restricts on $\eta_\Lambda$ which are given as

$$\eta_\Lambda \le\max\left\{\frac{\sqrt{16d+20( 1 - e^{-\epsilon/4})}-4\sqrt d}{10\sqrt\beta}, \frac{1}{2\sqrt{d\beta}}\right\}.$$
Simplifying the first constraint gives 
$$\eta_\Lambda \leq \frac{\sqrt{16d+20( 1 - e^{-\epsilon/4})}-4\sqrt d}{10\sqrt\beta} = \frac{20( 1 - e^{-\epsilon/4})}{10\beta(\sqrt{16d+20( 1 - e^{-\epsilon/4})}+4\sqrt d)}.$$
Which is satisfied for $\eta_\Lambda \le \frac{(1-e^{-\epsilon/4)}}{6\sqrt{d\beta}} \leq \frac{1}{2\sqrt{d\beta}}.$
Therefore the final bound is satisfied when $\beta$ is chosen as (\ref{e: beta final}) and $\eta_\Lambda\le\frac{(1-e^{-\epsilon/4)}}{6\sqrt{d\beta}}$. Without loss of generality we replace $G \leq G_+=\max\{1,G\}$ and $\rho_- \geq \min\{1,\rho\}$ in the final bound. 
\end{proof}

\subsection{Proof of Section Main Result}
\begin{proof} We establish upper and lower bounds for the numerator \(\bar S(\bar w,g)\), followed by upper and lower bounds for the denominator \(\bar S(\bar w,1)\).

The upper bound for the numerator by letting $a = \frac{d}{\rho} + G_+\kappa$ and $\lambda = a(1+\frac{8}{5\epsilon_0})$ and defining 
$$h_\lambda := e^{-\lambda d(E, \theta(x))}.$$
We get that for both $h_\lambda(x)$ and the constant observable $1$ that these functions are constant on the fiber 
therefore $L_{fiber} = 0$ and $\gamma_0 >0$ defined in Algorithm \ref{alg:gaussian_net_construction}, to satisfy $\Ctheta L^2_{fiber} \leq \frac{\gamma_0}{2}$, can take any value. 
The log-Lipschitz constant $L_\theta = \max\{L_{\overline h_\lambda},\; L_{\overline{h_\lambda^2}}\}= 2\lambda$. 
Therefore the log-Lipschitz constant in Alg. (\ref{alg:gaussian_net_construction}) is given by 
$$L_g = 2\left(\frac d\rho+G^+\kappa\right)
\left(1+\frac{8}{5\varepsilon_0}\right).$$
and we choose 
\begin{equation}\label{eq:beta_u}
    \beta_U > \max\left\{\ell +\frac{4L_g^2}{\gamma_0}, 2\ell, {G_+^2}\right\}, \quad \Delta =\frac{\gamma_0\sqrt{\pi}}{120 d \sqrt{\beta_U}}, \quad N_p \geq \max\left\{\frac{3}{\sqrt{\beta_U\Delta^2}}, \frac{20d}{\sqrt{e^9\pi}}\sqrt{\frac{1}{\beta_U\Delta^2}}, \right\}
\end{equation}
as in Alg. (\ref{alg:gaussian_net_construction}). 
Then Proposition \ref{p: upper_bound_g_L} is satisfied with 
\begin{equation}\label{eq: parameters upper}
    \tilde{\epsilon} \le \min\left( \frac{e^{-3\gamma_0}}{4}, \, \frac{\epsilon_0^2c_p^*}{1024L^2C_{\mathcal Q_*}^2} \right), \quad N \ge \frac{6144mL^2C_{\mathcal Q_*}^2}{\epsilon_0^2c_p^*\delta_0},
\end{equation}
as in Theorem \ref{t: Explicit Error for Log-lipschitz observable}. This explicit choice of parameters gives,  with probability $1-\delta_0$,
$$\bar{S}(\bar{w},g) \leq e^\epsilon_0 S(w,g).$$
For the lower bound we can apply Proposition \ref{p:full_lower_bound} with parameters 
\begin{equation}\label{eq: paramters lower}
N\ge
\max\left\{
\frac{64r_s^2}{3\delta_0},
\frac{36864mL^2}{\epsilon_0^2c_p^*\delta_0},
\frac{3072r_s^2}{\epsilon_0^2\delta_0}
\right\}, \quad
\tilde\epsilon
\le
\min\left\{
\frac1{64},
\frac{3}{32r_s^2},
\frac{\epsilon_0^2c_p^*}{9216L^2},
\frac{\epsilon_0^2}{1536r_s^2}
\right\},
\end{equation}
and
\begin{equation}\label{eq: beta lower}
\beta_L > \max\left\{2\ell, \; \frac{2368 \cdot G_+^4d\kappa^2\log\left(\frac{2}{1-e^{-\epsilon_0/4}}\right)}{\rho^2\epsilon_0^2}, \;\frac{9472G_+^2d^3}{\rho^2\epsilon_0^2}\right\}, \quad \eta_\Lambda\le\frac{(1-e^{-\epsilon_0/4})}{6\sqrt{d\beta_L}}.
\end{equation}
To get that, with probability at least $1-\delta_0$, 
$$\bar{\mathcal{S}}(\bar w, g) \geq e^{-\epsilon_0} \mathcal{S}(w, g).$$
We note that the padding parameter is the same as in Alg. (\ref{alg:gaussian_net_construction}). 
It is left to establish $\Delta$ as a function of the covering radius. The covering radius is defined by
\[
\eta_\Lambda
:=
\sup_{z\in A}\min_{y_k\in\Lambda}\|z-y_k\|.
\] With coordinate spacing \(\Delta\), every coordinate of \(z-y_k\) has magnitude at most \(\Delta/2\). Therefore,
\[
\eta_\Lambda
\le
\sqrt{\sum_{i=1}^d\left(\frac{\Delta}{2}\right)^2}
=
\frac{\sqrt d}{2}\Delta.
\]
So the spacing $\Delta$ must satisfy 
$$\Delta = \frac{2\eta_\Lambda}{\sqrt{d}} = \frac{2(1-e^{-\epsilon_0/4)})}{6d\sqrt{\beta_L}}.$$
Lastly, we establish parameters for the constant function $f=1$. Again, for the constant function $L_{fiber} = 0$, so we have free choice of $\gamma_0 \in (0, 120d/\sqrt{\pi}]$. 
Moreover, $L_1 = 0$ and this constant disappears in the parameter choice. 
Then taking the parameters from Alg. (\ref{alg:gaussian_net_construction}) we have 
\begin{equation}\label{eq: beta 1}
    \beta_1 > \max\left\{2\ell, \frac{G_+^2+\ell}{2}\right\}, \; \Delta =\frac{\gamma_0\sqrt{\pi}}{120 d \sqrt{\beta_1}}, \; N_p \geq \max\left\{\frac{3}{\sqrt{\beta\Delta^2}}, \frac{20d}{\sqrt{e^9\pi}}\sqrt{\frac{1}{\beta\Delta^2}}\right\}.
\end{equation}
Then Theorem \ref{t: Explicit Error for Log-lipschitz observable} is satisfied with, the same parameters as in (\ref{eq: parameters upper}) 
and yields, with probability at least $1 - \delta_0$,
$$e^{-\frac{3\epsilon_0}{8}}\mathcal S(w,f)
\leq
\bar{\mathcal S}(\bar w,f)
\leq
e^{\frac{3\epsilon_0}{8}}\mathcal S(w,f).$$ 
Taking a maximum of parameters (\ref{eq:beta_u}), (\ref{eq: beta lower}) and (\ref{eq: beta 1}) so that 
\begin{align*}
    \beta > \max\{\beta_1, \beta_U, \beta_L\} &= \max\left\{2l, \; G_+^2, \; \ell +\frac{4L_g^2}{\gamma_0}, \; \frac{2368 G_+^4d\kappa^2\log\left(\frac{2}{1-e^{-\epsilon_0/4}}\right)}{\rho_-^2\epsilon_0^2}, \;\frac{9472G_+^2d^3}{\rho_-^2\epsilon_0^2}\right\},\\
    \Delta &= \min\left\{\frac{(1-e^{-\epsilon_0/4)}}{6d\sqrt{\beta}}, \frac{\gamma_0\sqrt{\pi}}{120d\sqrt{\beta}}\right\},\\
  N_p &\geq \max\left\{\frac{3}{\sqrt{\beta\Delta^2}}, \frac{20d}{\sqrt{e^9\pi}}\sqrt{\frac{1}{\beta\Delta^2}}\right\}
\end{align*}
and choosing the maximum of (\ref{eq: parameters upper}) and (\ref{eq: paramters lower}) so that
\begin{align*}
N&\ge
\max\left\{
\frac{64r_s^2}{3\delta_0},
\frac{36864mL^2}{\epsilon_0^2c_p^*\delta_0},
\frac{3072r_s^2}{\epsilon_0^2\delta_0},
\frac{6144mL^2C_{\mathcal Q_*}^2}{\epsilon_0^2c_p^*\delta_0},
\right\},\\
\tilde\epsilon
&\le
\min\left\{
\frac1{64},
\frac{3}{32r_s^2},
\frac{\epsilon_0^2c_p^*}{9216L^2},
\frac{\epsilon_0^2}{1536r_s^2}, \frac{e^{-3\gamma_0}}{4}, \, \frac{\epsilon_0^2c_p^*}{1024L^2C_{\mathcal Q_*}^2}
\right\}.
\end{align*}
Then applying Theorem \ref{t: general error decomposition} we get, with probability at least $1 - 3\delta_0$, 

$$\left|
\log\bar\pi[g]-\log\pi[g]
\right|
\leq\fc{11}8\epsilon_0
 + \max\left\{\log\frac{1}{1-\delta_\pi}, \log\frac{1}{1-\delta_g}\right\}.$$
 For the truncation term require \(\max\{\delta_\pi,\delta_g\}
\le
1-e^{-\epsilon/2}\) which yields $\max\left\{\log\frac{1}{1-\delta_\pi}, \log\frac{1}{1-\delta_g}\right\} \leq \frac{\epsilon}{2}.$
Letting $\delta_0 = \frac{1}{3}\delta$ and $\epsilon_0 = \frac{4}{11}\epsilon$ substituted into the final bounds yields,  with probability at least $1 - \delta$, 

$$\left|
\log\bar\pi[g]-\log\pi[g]
\right| < \epsilon.$$
The statement of the theorem following by substituting  $\delta_0 = \frac{1}{3}\delta$, $\epsilon_0 = \frac{4}{11}\epsilon$ and $\gamma_0 = 1$ into the parameter choices. 
\end{proof}
\section{Efficient Sampling}\label{s: sampling}
In this section, we establish mixing time guarantees for drawing samples from the truncated target distributions $\pi_k(x) \propto \psi_k(x)\pi(x)1_{K_U}(x)$.
The sampling process established in this section is independent of the observable $g(x)$ so that the analysis in this section applies to both the indicator observable (Section  \ref{s: indicator}) and the log-Lipschitz observable (Section \ref{s: smooth observable}). 
The sampling error depends exclusively on the geometry of the potential $V(x)$ and the biasing functions $\psi_k(x)$. 

This means that we cannot decouple the Log-Sobolev constants $C_{LS}^{(\theta)}$ and $C_{LS}^{(d)}$ as we did in Section \ref{s: 
smooth observable}. 
Instead we rely on Theorem \ref{t: LSI_bounds_CV_fiber} to express the log-Lipschitz constant $C_{LS}^{(D)}$ on the full space $\mathbb{R}^D$ as a function of $C_{LS}^{(\theta)}$ and $C_{LS}^{(d)}$. 
Therefore, our efficient sampling bounds have an unavoidable dependence on $C_{LS}^{(\theta)}$ the log-Lipschitz constant induced by the collective variable mapping. 

The bounds in this section then depend on the temperature choice for the biasing windows $\psi_k(x)$ to control the LSI constant $C_{LS}^{(\theta)}$ and the choice of CV mapping to control $C_{LS}^{(d)}$. 

\subsection{Truncated Windows \& Reflected Sampling}\label{s: truncated windows}

\blu{Because algorithmically we can only choose a finite collection of windows, any implementation of umbrella sampling for an unbounded distribution must address the tail of the distribution. Two options are to define terminal windows, a common approach in practice, or to introduce a truncation region $K_U$. We opt for the latter for ease of theoretical analysis.}
The truncation region $K_U$ must (i) capture a good portion of $\pi$, the target distribution's mass, and, (ii) capture the bulk of $g(x)$, the observable function.

Since Gaussian windows have full support, in the absence of truncation we have $U = \{x \in \Omega: S(x) > 0\}$. 
Taking $K_U = U$ therefore satisfies Assumptions \ref{a: good truncation} trivially with $\delta_\pi = \delta_g = 0$. 
While not strictly required, a well chosen truncation set $K_U$ is essential to the effectiveness of Gaussian EMUS scheme.
Importantly, the truncated set should be chosen to satisfy (Assumption \ref{a: good truncation}) within a well covered lattice of Gaussian windows.
This prevents the relative weights associated with the Gaussian umbrellas from becoming arbitrarily ill-conditioned in the tails. 
Specifically, we desire low variance estimates of the empirical quantities $\bar{F}_{ij}$ and $\bar{g}_i^*$, for each $i$. 


First, we assume $\mathcal{A}$ is a convex set in order to apply Proposition 3.1 of Bobkov-Ledoux \cite{BobkovLedoux2000}. 
Together with the strong convexity of the tilted marginal potential, this gives an LSI for the truncated biased marginal in collective-variable space and an upper bound on its LSI constant.
The marginal and fiber inequalities are subsequently combined in Section 3 to obtain a full-space LSI for $\pi_k$.

Second, in addition to convexity, we assume $K_U = \theta^{-1}(\mathcal{A})$ has $C^2$ boundary. 
Then define the normally reflected over-damped Langevin dynamics as 
$$dX_t = \nabla\log\pi_k(X_t) dt + \sqrt{2}dB_t - n_{K_U}(X_t)dL_t,$$
where $n_{K_U}$ is the outward normal so that
$n_{{K_U}(x)} \cdot \nabla f(x) = 0$ for $x \in \partial {K_U},$
and $L_t$ is the boundary local time.
Wang (\cite[Lemma 3.2]{Wang2014ModifiedCurvatures}) shows that the reflected over-damped Langevin process yields the appropriate Dirichlet form:
$$\mathcal{E}_{\pi_k}(f,g) = \int_{K_U} \nabla f \cdot \nabla g \; d\pi_k.$$
Recovering the appropriate Dirichlet form ensures that mixing is controlled by the aforementioned LSI inequality. 
Therefore our results hold for samples collected via reflected over-damped Langevin, as discussed in Section \ref{s: sampling}.


\subsection{Log-Sobolev Constants}
Since the Markov process for the sampling procedure runs in the full space $\mathbb{R}^D$ we require bounds on the LSI constant $C_{LS}^{(D)}$.
We apply Theorem \ref{t: LSI_bounds_CV_fiber} to bound the LSI on the full space in terms a function of $C_{LS}^{(\theta)}$ and $C_{LS}^{(d)}$, which are explicitly defined in Section \ref{s: collective variable mapping}. 

In the collective variable space $\mathbb{R}^d$, the tilted distribution is given by $\pi_k(x) \propto \exp(-V(x) - \frac{\beta}{2}\|\theta(x) - y_k\|^2)1_{K_U}(x)$. 
Since we assume the push-forward of the potential $A(z)$ is $\ell$-smooth, by choosing the temperature $\beta \geq \ell$, the window distribution on the collective variable space is strongly log-concave. 
By the Bobkov-Ledoux criterion (Remark \ref{r: Bakry-Emery}), the LSI of the marginal distribution on the collective variable space satisfies $C_{LS}^{(d)} \leq \frac{1}{\beta-\ell}$.

The LSI of the conditional measure along the fiber $\Sigma_z$ satisfies a constant $C_{LS}^{(\theta)}$ which cannot be further reduced. 
Therefore we have that the unconstrained tilted distribution $\pi_k$ satisfies the bound

\begin{equation}\label{e: LSI constant notation_sampling_section}
    \CD \leq \Ctheta + \frac{1}{(\beta - \ell)}\left(\frac{1}{m}+\Lambda^2(\Ctheta)^2\right).
\end{equation}

The expression for $\CD$ makes clear that the efficiency of the sampling process is heavily dependent on the choice of the collective variable mapping.
If $\theta(x)$ is poorly chosen so that the fibers have some implicit multimodality the sampling process may suffer from torpid mixing.

\subsection{Mixing Time Bounds}
Given a bound on the LSI constant $\CD$ for  distribution $\pi_k$ on the full space, we can now bound the mixing time for the continuous time sampling procedure. 
For a target distribution satisfying an LSI, the $\chi^2$-divergence converges exponentially fast. 
Since $\pi_k$ is defined over a truncated set $K_U$ we use reflected Langevin dynamics. 

\begin{proposition}\label{p: mixing time general}
    Let $\mu_T$ be the distribution obtained from running the reflected over-damped Langevin process with stationary distribution $\pi_k$, for time $T$, initialized at $\mu_0$. 
    Then to guarantee $\chi^2(\mu_T \; || \; \pi_k ) \leq \tilde{\epsilon}$ it is sufficient to run the Markov process for
    $$T \geq \frac{\CD}{2}\log\left( \frac{\chi^2(\mu_0 \; || \; \pi_k ) }{\tilde{\epsilon}}\right).$$
    Where $\CD \leq \Ctheta + \frac{1}{(\beta - \ell)}\left(\frac{1}{m}+\Lambda^2(\Ctheta)^2\right)$, as in equation \ref{e: LSI constant notation_sampling_section}. Choosing $\beta \geq \ell + \frac{1}{\Ctheta}\left(\frac{1}{m}+\Lambda^2(\Ctheta)^2\right)$ implies $\CD \leq 2\Ctheta$.
\end{proposition}

The mixing time provided in Proposition \ref{p: mixing time general} can be controlled by a good choice of CV mapping, the inverse temperature $\beta$ and the choice of initial distribution $\mu_0$. 
We previously discussed how the CV mapping and $\beta$ control the LSI constant so lastly we remark on a proper choice of initialization.
We leave this as a remark since it is clear that the $\log \chi^2(\mu_0 \; || \; \pi_k)$ dependence yields polynomial-time convergence.  

\begin{rem}
    While Proposition \ref{p: mixing time general} guarantees polynomial-time convergence, the practical run time is influenced by the initial choice of $\mu_0$, as to control $\chi^2(\mu_0 \;||\; \pi_k)$. 
    Since the Markov process runs over the full space, if the process is initialized naively from some generic prior, the $\log \chi^2(\mu_0 \; || \; \pi_k)$ scales conventionally as $O(D)$. 

    However, the overlap assumption $F_{ij} \geq c_p$, guarantees that each window has at least one neighbor with good overlap. 
    This suggests an a practical sequential initialization procedure by using samples from a neighboring window. 
With additional work its possible to show that this inductive procedure  reduces $\log \chi^2(\mu_0 \; || \; \pi_k)$ to $O(1)$.
\end{rem}
\section{Proof of Main Theorem}
\label{s:pf}
We prove Theorem \ref{t: main theorem alg1} with more precise dependencies. Let
\[
 G_+=\max\{1,G\},\qquad
 H_\theta=\frac{1}{m_\theta C_{LS}^{(\theta)}}
 +\Lambda^2C_{LS}^{(\theta)}, \qquad \rho_-=\min\{1,\rho\}.
\]
For the log-Lipschitz and indicator settings, let
\begin{align*}
 K_{LL}:&=1+\operatorname{diam}(A)
 \sqrt{\ell+G_+^2+L_g^2+H_\theta},\\
 K_{ind}(\epsilon):&=\frac{\sqrt{\log{2/\epsilon}}}{\epsilon^2}\left(1+\operatorname{diam}(A)
 \sqrt{\ell+G_+^2+\frac{G_+^4\kappa^2}{\rho_-^2}+H_\theta}\right).
 \end{align*}
We show that it suffices to choose $N$ and $T$ as follows in the two settings:
\begin{enumerate}
    \item For $g$ log-Lipschitz,
\[
 N=\Omega\!\left(
 \frac{e^{\widetilde O(d^3)}K_{LL}^{3d}}{\epsilon^2\delta}\right),
 \qquad
 T=\Omega\!\left(C_{LS}^{(\theta)}\log\!\left(
 \frac{Be^{\widetilde O(d^3)}K_{LL}^{2d}}{\epsilon^2}\right)\right).
\]
\item For $g=\mathbf1_E\circ\theta$ over a convex $E$,
\[
 N=\Omega\!\left(
 \frac{e^{\widetilde O(d^3)}K_{ind}(\epsilon)^{3d}}{\epsilon^{d+2 }\delta}\right),
 \qquad
 T=\Omega\!\left(C_{LS}^{(\theta)}\log\!\left(
 \frac{Be^{\widetilde O(d^3)}K_{ind}(\epsilon)^{2d}}{\epsilon^{d+2}}\right)\right).
\]
\end{enumerate}  

\begin{proof}(Proof of Theorem \ref{t: main theorem alg1})
    We show that in the log Lipschitz observable case Algorithm \ref{alg:gaussian_net_construction} parameters applied to Theorem \ref{t: Explicit Error for Log-lipschitz observable} yields the result and in the indicator observable case Algorithm \ref{alg:indicator_net_construction} applied to Theorem \ref{t:indicator final bound} yields the result. 

    We first start with $g$ log-Lipschitz:
    in this case, with parameters as in Algorithm \ref{alg:gaussian_net_construction}, Theorem \ref{t: Explicit Error for Log-lipschitz observable} with $N \ge \frac{6144mL^2C_{\mathcal Q_*}^2}{\epsilon^2c_p^*\delta}$ and $\tilde{\epsilon} \leq \min\left( \frac{e^{-3\gamma_0}}{4}, \, \frac{\epsilon^2c_p^*}{4\cdot16^2L^2C_{\mathcal Q_*}^2} \right)$ yields the result. 
 We note that $\gamma_0$ is a hyperparameter so we take it to be a fixed constant in $(0, 120d/\sqrt{\pi}]$, 
which means that as defined 
$$ c_{p,LL} \ge c_p^* =\frac{
\exp\left[-\max(\frac{9d}{2},\frac{d^3}{100})
-\frac14-\left(\frac{\gamma_0\sqrt{\pi}}{120d}\right)^2
-\frac{\gamma_0\sqrt{\pi}}{240d}
\right]
}{
2^d
\left(
1+\frac{120\sqrt{2}\,d}{\gamma_0}
\right)^d
}$$
can be written as $c_{p}^* = e^{-\widetilde{\mathcal{O}}(d^3)}$. Furthermore, 
$$ C_{\mathcal{Q}}^* := 2(1 + 0.25 e^{-3\gamma_0})^{1/4} \exp(6\gamma_0) $$
therefore with $\gamma_0$ taken to be a constant $C_{\mathcal{Q}}^* = O(1)$. Which means that 
$N = \Omega\left(\frac{L^3e^{\widetilde{\mathcal{O}}(d^3)}}{\epsilon^2\delta} \right)$. 
Lastly to connect this to the mixing time of the continuous chain we apply Proposition \ref{p: mixing time general} with $ \beta \geq \ell + \frac{1}{\Ctheta}\left(\frac{1}{m}+\Lambda^2(\Ctheta)^2\right)$ to get $T \geq \Ctheta\log\left( \frac{\chi^2(\mu_0 \; || \; \pi_k ) }{\tilde{\epsilon}}\right).$ 
Using that $\tilde{\epsilon} \leq \min\left( \frac{e^{-3\gamma_0}}{4}, \, \frac{\epsilon^2c_p^*}{4\cdot16^2L^2C_{\mathcal Q_*}^2} \right)$ and  $\chi^2(\mu_{k,0}, \pi_k) \leq B$ gives a required mixing time of  
$$T \geq \max\left\{\Ctheta\log\left( \frac{4\chi^2(\mu_0 \; || \; \pi_k )}{e^{-3\gamma_0}}\right), \Ctheta\log\left( \frac{4\cdot16^2L^2C_{\mathcal Q_*}^2\chi^2(\mu_0 \; || \; \pi_k ) }{\epsilon^2c_p^*}\right)\right\}.$$
Which reduces to $T=\Omega\left(\Ctheta \log \left(\frac{Be^{\widetilde{\mathcal{O}}(d^3)}L^2}{\epsilon^2} \right)\right)$ per window. 

For $g$ an indicator function: with parameters as in Algorithm \ref{alg:indicator_net_construction} applied to Theorem \ref{t:indicator final bound} with 
$$ \tilde{\epsilon} \le \min\left\{
\frac1{64},
\frac{3}{32r_{s,*}^2},
\frac{\epsilon^2c_p}{69696L^2},
\frac{\epsilon^2}{11616r_{s,*}^2}, \frac{e^{-3}}{4}, \, \frac{\epsilon^2c_p^*}{7744L^2C_{\mathcal Q_*}^2}
\right\}, $$
$$ N \ge \max\left\{
\frac{64r_{s,*}^2}{\delta},
\frac{836,352mL^2}{\epsilon^2c_p^*\delta},
\frac{69,696r_{s,*}^2}{\epsilon^2\delta},
\frac{139,392mL^2C_{\mathcal Q_*}^2}{\epsilon^2c_p^*\delta}
\right\},$$
we first consider $c_p^*$ chosen as 
$$c_{p,ind} \ge c_p^* =\frac{
\exp\left(
-\max(\frac{9d}{2},\frac{d^3}{100})
-\frac12
\right)
}{
2^d
\left(
1+\frac{72\sqrt{2\pi}d}{\epsilon}
\right)^d
}.$$
We note that $c_p^*$ differs in this case by a factor of $(\frac{1}{\epsilon})^d$ in the denominator. 
So the dependency is given as $c_p^*=\Omega({e^{-\widetilde{\mathcal{O}}(d)}}{\epsilon^d})$
As specified in Theorem \ref{t:indicator final bound} $r_{s,*} = \mathcal{O}(e^{d})$ removing these constants gives the same dependency of $N = \Omega\left(\frac{e^{\widetilde{\mathcal{O}}(d^3)}L^3}{\epsilon^{d+2}\delta}\right)$. 
Connecting back to the run time through Proposition \ref{p: mixing time general} in the same way as the log-Lipschitz case gives $T=\Omega\left(\Ctheta \log \left(\frac{Be^{\widetilde{\mathcal{O}}(d^3)}L^2}{\epsilon^{d+2}} \right)\right)$.

It remains to express the number of windows in terms of the algorithmic parameters. 
Since $g$ is constant along the fibers taking $\gamma_0 = 1$ suffices. 
Let $G_+ = \max\{1,G\}$, $\kappa = R/\rho$, $\rho_- = \min\{1, \rho\}$, and $H_\theta = \frac{1}{m_\theta \Ctheta}+\Lambda^2\Ctheta$.
Define 
\[
 K_{LL}:=1+\operatorname{diam}(A)
 \sqrt{\ell+G_+^2+L_g^2+H_\theta},
\]
\[
 K_{ind}(\epsilon):=\frac{\sqrt{\log{2/\epsilon}}}{\epsilon^2}\left(1+\operatorname{diam}(A)
 \sqrt{\ell+G_+^2+\frac{G_+^4\kappa^2}{\rho_-^2}+H_\theta}\right).
\]
In Remark \ref{rem: L number} the number of windows $L$ is expressed as 
$$L = \mathcal{O}\left((\text{diam}(\mathcal{A})d\sqrt{\beta} + d)^d\right).$$
Using 
\[
\begin{aligned}
 \beta_{LL}
 &=O\!\left(\ell+G_+^2+L_g^2+H_\theta\right),\\[4pt]
 \beta_{ind}
 &=O\!\left(
 \frac{d^3\log(2/\epsilon)}{\epsilon^2}
 \left(\ell+G_+^2+\frac{G_+^4\kappa^2}{\rho_-^2}+H_\theta\right)
 \right)
\end{aligned}
\]
applied to Remark \ref{rem: L number} yields 
\[
 L_{LL}\le e^{O(d\log(d+1))}
 \left(1+a\sqrt{\ell+G_+^2+L_g^2+H_\theta}\right)^d
 =e^{O(d\log(d+1))}K_{LL}^d
\]
and
\[
 L_{ind}\le e^{O(d\log(d+1))}
 \left(\frac{1+a\sqrt{\beta_{ind}}}{\epsilon}\right)^d
 \le e^{O(d\log(d+1))}K_{ind}(\epsilon)^d.
\]
Substituting these bounds into the previously established dependencies yields the final result.
\end{proof}

\printbibliography

\appendix
\section{Appendix}
\begin{lem}\label{l: herbst argument}(Herbst Argument) Suppose that $\mu$ satisfies a log-Sobolev inequality with constant $C_{LS}$. Let $f: \mathbb{R}^d \rightarrow \mathbb{R}$ be an $L$-Lipschitz function. Then 
$$\mathbb{E}_{\mu}\left[e^{t(f-\mathbb{E}_\mu f)}\right] \leq e^\frac{C_{LS}L^2t^2}{2}.$$
\end{lem}
\begin{corollary}\label{c: Herbst}(Herbst Argument for L log-Lipshitz) Suppose that $\mu$ satisfies a log-Sobolev inequality with constant $C_{LS}$. let $f$ be an L log-Lipshitz function. Then 
$$e^{t\mathbb{E}_\mu\log f - \frac{C_{LS}L^2t^2}{2}}\leq\mathbb{E}_\mu f^t \leq e^{t\mathbb{E}_\mu\log f + \frac{C_{LS}L^2t^2}{2}}$$
\end{corollary}
\begin{proof}
    Let $f$ be L log-Lipshitz then consider, 
    $$\mathbb{E}_\mu f^t = \mathbb{E}\left[ e^{t\log f} \right].$$
    Since $g = \log f$ is an L-lipshitz function by Herbst Argument (Lemma \ref{l: herbst argument}),
    $$\mathbb{E}\left[ e^{t\left(\log f-\mathbb{E}_\mu \log f\right)} \right] \leq e^\frac{C_{LS}L^2t^2}{2}.$$
\end{proof}
\begin{lem}\label{l: hypercontractivity}(Hypercontractivity) Suppose that $\mu$ satisfies a log-Sobolev inequality with constant $C_{LS}$. Let $f$ be an L log-Lipshitz function. Then 
$$\frac{\left\| f \right\|_{L^4(\mu)}}{\left\|f\right\|_{L^2(\mu)}} \leq e^{ 3C_{LS}L^2}.$$
\end{lem} 
\begin{proof}
Consider by Corollary \ref{c: Herbst}, 
$$\frac{\left\| f \right\|_{L^4(\mu)}}{\left\|f\right\|_{L^2(\mu)}} = \frac{\mathbb{E}_\pi\left(f^4\right)^\frac{1}{4}}{\mathbb{E}_\pi\left(f^2\right)^\frac{1}{2}}\leq \frac{e^{\mathbb{E}_\mu\log f + \frac{4C_{LS}L^2}{2}}}{e^{\mathbb{E}_\mu\log f - \frac{2C_{LS}L^2}{2}}} \leq e^{ 3C_{LS}L^2}.$$
\end{proof}
    
\begin{lem}\label{l: gross_hypercontractive}(Gross's Hypercontractivity) Let $P_t$ be a reversible Markov semigroup with stationary measure $\mu$. If $\mu$ satisfies a log Sobolev inequality with constant $C_{LS}$ then 
$$\left\|P_Tf\right\|_4 \leq ||f||_2$$
when $T \geq \frac{C_{LS}}{2}\log(3)$.
\end{lem}
\begin{lem}\label{l: chisq MCMC}
 Let $P_t$ be a reversible ergodic Markov semigroup with stationary distribution $\mu$. Let $\mu_0$ be the initialization of $P_t$ and $\mu_t = \mu_0P_t$. Then 
    $$\chi^2(\mu_t, \mu) \leq e^\frac{-2t}{C_{LS}}\chi^2(\mu_0, \mu).$$
\end{lem}

\begin{lem}\label{l: Lipshitz generic bound}
    Let $f:\mathbb{R}^d \rightarrow \mathbb{R}$ so that $f(x) \leq f_{\max}$ for $x\in \mathbb{R}^d$, $f(x) \geq f_{\min}$ for all $x \in A$, where $A \subseteq \mathbb{R}^d$ is a bounded, convex domain and $\nabla^2f(x) \succeq -aI_d$. Then 
    $$\left\| \nabla f(x)\right\|^2 \leq 2a\left(f_{\max} - f_{\min}\right).$$
    In particular, $f(x)$ is Lipschitz with $L = \sqrt{2a\left(f_{\max} - f_{\min}\right)}$.
\end{lem}
\begin{proof}
    Let $x_0 \in A$ define the function $\phi(t) = f(x_0 + tu)$, where $u = \frac{\nabla f(x_0)}{\left\|\nabla f(x_0)\right\|}$. Then we have 
    $$\phi'(0) = \nabla f(x_0)\cdot u = \nabla f(x_0) \cdot \frac{\nabla f(x_0)}{\left\|\nabla f(x_0)\right\|} = \left\|\nabla f(x_0)\right\|.$$
    Thereby we reduce the problem to upper bounding $\phi'(0)$. Moreover, we have 
    $$\phi''(t) = u^T\nabla^2f(x_0+tu)u\succeq -a\frac{\nabla f(x_0)^T}{\left\|\nabla f(x_0)\right\|}\frac{\nabla f(x_0)}{\left\|\nabla f(x_0)\right\|} = -a.$$
    Consider the integral, 
    $$\int_{0}^z \phi''(s) ds \geq \int_0^z -a ds \implies \phi'(z) - \phi'(0) \geq -az.$$
    Therefore $\phi'(z) \geq \phi'(0) - az$ and 
    $$\int_0^t\phi'(z) dz \geq \int_0^t\left(\phi'(0) - az\right)dz \implies \phi(t) - \phi(0) \geq \phi'(0)t - \frac{a}{2}t^2.$$
    Recalling that $\phi(t) = f(x_0 + tu)$ and $f(x) \leq f_{\max}$ for all $x\in \mathbb{R}^d$ yields
    $$f_{\max} \geq f(x_0) + \left\| \nabla f(x_0)\right\|t - \frac{a}{2}t^2\geq f_{\min} + \left\| \nabla f(x_0)\right\|t - \frac{a}{2}t^2.$$
    Rearranging terms gives 
    $$g(t) = \frac{a}{2}t^2 - \left\| \nabla f(x_0)\right\|t + (f_{\max} - f_{\min}) \geq 0.$$
    This is a quadratic function in $t$, with $a > 0$, therefore the minimum occurs at $g'(t) = 0$ which is given by 
    $$g'(t) = at - \left\| \nabla f(x_0)\right\| = 0 \implies t = \frac{\left\| \nabla f(x_0)\right\|}{a}.$$
    Substituting this back into $g(t)$ gives 
    $$\frac{\left\| \nabla f(x_0)\right\|^2}{2a} - \frac{\left\| \nabla f(x_0)\right\|^2}{a} + (f_{\max} - f_{\min})  \geq 0.$$
Rearranging terms yields, 
$$\left\| \nabla f(x_0)\right\|^2 \leq 2a(f_{\max} - f_{\min}).$$ 
\end{proof}

\begin{lem}\label{l: log_pert_bound}(Logarithmic perturbation bound) Let \(x\in(-1,1)\). Then

$$
|\log(1+x)|
\le
-\log(1-|x|)
\le
\frac{|x|}{1-|x|}.
$$

In particular, if \(|x|\le 1/2\), then

$$
|\log(1+x)|\le 2|x|.
$$
\end{lem}
\begin{proof}

 Set \(r=|x|\). Since \(1-r\le 1+x\le 1+r\),

$$
|\log(1+x)|
\le
\max\{-\log(1-r),\log(1+r)\}.
$$

Moreover,

$$
-\log(1-r)-\log(1+r)
=
-\log(1-r^2)\ge0,
$$

so the maximum is \(-\log(1-r)\). Finally,

$$
-\log(1-r)
=
\int_0^r\frac{dt}{1-t}
\le
\frac{r}{1-r}.
$$

When \(r\le1/2\), \(r/(1-r)\le2r\), which proves the final claim. \(\square\)
\end{proof}

\end{document}